\documentclass{amsart}
\usepackage{amsmath}
\usepackage{amssymb}
\usepackage{amsthm}
\usepackage[msc-links,alphabetic,abbrev]{amsrefs}
\usepackage{hyperref}
\usepackage[noabbrev,capitalize]{cleveref}
\usepackage{mathrsfs}
\usepackage{mathtools}

\DeclareMathOperator{\divsymb}{div}
\DeclareMathOperator{\odiv}{\overline{div}}

\DeclareMathOperator{\tr}{tr}
\DeclareMathOperator{\Vol}{Vol}
\DeclareMathOperator{\dvol}{dV}
\DeclareMathOperator{\darea}{dA}

\DeclareMathOperator{\Ric}{Ric}
\DeclareMathOperator{\Rm}{Rm}
\newcommand{\oRm}{\overline{\mathrm{Rm}}}
\newcommand{\Wm}{\mathcal{I}}
\newcommand{\Wn}{\mathcal{J}}
\newcommand{\Pf}{\mathrm{Pf}}
\newcommand{\oPf}{\overline{\mathrm{Pf}}}

\DeclareMathOperator{\lots}{l.o.t.}

\DeclareMathOperator{\vspan}{span}

\DeclareMathOperator{\Contr}{Contr}

\newcommand{\mfdim}{n}
\newcommand{\submfdim}{k}
\newcommand{\pedimplaceholder}{n}
\newcommand{\pedim}{n}
\newcommand{\pedimparen}{n}
\newcommand{\pedimminusone}{n-1}
\newcommand{\pedimminustwo}{n-2}
\newcommand{\pedimminusthree}{n-3}
\newcommand{\pedimminusoneparen}{(n-1)}
\newcommand{\pedimplusone}{n+1}
\newcommand{\pedimiseven}{$n$ is even}
\newcommand{\pedimeven}{$n$ even}
\newcommand{\pedimevengeqfour}{$n \geq 4$ even}
\newcommand{\pedimequalsthree}{$n=3$}
\newcommand{\pedimequalsfour}{$n=4$}
\newcommand{\pedimgeqthree}{$n \geq 3$}
\newcommand{\pedimequalsfive}{$n = 5$}
\newcommand{\pedimgeqfive}{$n \geq 5$}
\newcommand{\penologterms}{$n$ is even or $n=3$}
\newcommand{\pelogterms}{$n > 3$ is odd}

\newcommand{\pesubmfdim}{k}
\newcommand{\pesubmfdimplaceholder}{k}

\newcommand{\pesubmfdimparen}{k}
\newcommand{\pesubmfdimplusone}{k+1}
\newcommand{\pesubmfdimplustwo}{k+2}
\newcommand{\pesubmfdimminusone}{k-1}
\newcommand{\pesubmfdimminusoneparen}{(k-1)}
\newcommand{\pesubmfdimminustwo}{k-2}
\newcommand{\pesubmfdimminustwoparen}{(k-2)}
\newcommand{\pesubmfdimminusthree}{k-3}
\newcommand{\pesubmfdimtwo}{$k=2$}
\newcommand{\pesubmfdimiseven}{$k$ is even}
\newcommand{\pesubmfdimeven}{$k$ even}

\newcommand{\minuspedim}{-n}
\newcommand{\minuspesubmfdim}{-k}
\newcommand{\minussubmfdim}{-k}
\newcommand{\pecodim}{n-k}
\newcommand{\pecodimone}{$n = k+1$}
\newcommand{\pecodimpositive}{$n > k$}
\newcommand{\codimpositive}{$n > k$}
\newcommand{\codimone}{$n = k+1$}

\newcommand{\Sch}{\mathsf{P}}
\newcommand{\trSch}{\mathsf{J}}
\newcommand{\Fi}{\mathsf{F}}
\newcommand{\trFi}{\mathsf{G}}
\newcommand{\Di}{\mathsf{D}}
\newcommand{\oSch}{\overline{\Sch}}
\newcommand{\otrSch}{\overline{\trSch}}
\newcommand{\onf}{\mathsf{n}}
\newcommand{\PtrFi}{P_{\dagger}}
\newcommand{\QtrFi}{Q_{\dagger}}

\newcommand{\Int}[1]{\mathring{#1}}

\newcommand{\defn}[1]{{\boldmath\bfseries#1}}

\newcommand{\tfss}{\mathring{L}}

\newcommand{\oP}{\overline{P}}
\newcommand{\oQ}{\overline{Q}}
\newcommand{\oR}{\overline{R}}
\newcommand{\oW}{\overline{W}}

\newcommand{\oh}{\overline{h}}

\newcommand{\oDelta}{\overline{\Delta}}

\newcommand{\onabla}{\overline{\nabla}}

\newcommand{\cmI}{\widetilde{\mathcal{I}}}
\newcommand{\cmJ}{\widetilde{\mathcal{J}}}

\newcommand{\cmN}{\widetilde{\mathcal{N}}}

\newcommand{\hh}{\widehat{h}}

\newcommand{\lp}{\langle}
\newcommand{\rp}{\rangle}
\newcommand{\lv}{\lvert}
\newcommand{\rv}{\rvert}

\newcommand{\mA}{\mathcal{A}}
\newcommand{\mB}{\mathcal{B}}
\newcommand{\mC}{\mathcal{C}}

\newcommand{\mI}{\mathcal{I}}
\newcommand{\mJ}{\mathcal{J}}
\newcommand{\mK}{\mathcal{K}}

\newcommand{\mN}{\mathcal{N}}

\newcommand{\mP}{\mathcal{P}}

\newcommand{\mR}{\mathcal{R}}

\newcommand{\mV}{\mathcal{V}}
\newcommand{\mW}{\mathcal{W}}

\newcommand{\mZ}{\mathcal{Z}}

\newcommand{\kc}{\mathfrak{c}}

\newcommand{\bN}{\mathbb{N}}

\newcommand{\bR}{\mathbb{R}}

\newcommand{\sI}{\mathscr{I}}

\def\sideremark#1{\ifvmode\leavevmode\fi\vadjust{\vbox to0pt{\vss
 \hbox to 0pt{\hskip\hsize\hskip1em
 \vbox{\hsize3cm\tiny\raggedright\pretolerance10000
 \noindent #1\hfill}\hss}\vbox to8pt{\vfil}\vss}}}

\newtheorem{theorem}{Theorem}[section]
\newtheorem{proposition}[theorem]{Proposition}
\newtheorem{lemma}[theorem]{Lemma}
\newtheorem{corollary}[theorem]{Corollary}

\theoremstyle{definition}

\newtheorem{conjecture}[theorem]{Conjecture}

\theoremstyle{remark}
\newtheorem{remark}[theorem]{Remark}

\numberwithin{equation}{section}

\makeatletter
\@namedef{subjclassname@2020}{%
  \textup{2020} Mathematics Subject Classification}
\makeatother

\begin{document}

\title[Gauss--Bonnet formula for renormalized area]{A Gauss--Bonnet formula for the renormalized area of minimal submanifolds of Poincar\'e--Einstein manifolds}
\author[J.\ S.\ Case]{Jeffrey S.\ Case}
\address{Department of Mathematics \\ Penn State University \\ University Park, PA 16802 \\ USA}
\email{jscase@psu.edu}
\author[C R.\ Graham]{C Robin Graham}
\address{Department of Mathematics \\ University of Washington \\ Box 354350 \\ Seattle, WA 98195-4350 \\ USA}
\email{robin@math.washington.edu}
\author[T.-M.\ Kuo]{Tzu-Mo Kuo}
\address{Department of Mathematics \\ National Taiwan University \\ Taipei 10617 \\ Taiwan}
\email{tzumokuo@ntu.edu.tw}
\author[A.\ J.\ Tyrrell]{Aaron J.\ Tyrrell}
\address{Department of Mathematics and Statistics \\ Texas Tech University \\ Lubbock, TX 79409-1042 \\ USA}
\email{aatyrrel@ttu.edu}
\author[A.\ Waldron]{Andrew Waldron}
\address{Center for Quantum Mathematics and Physics (QMAP) \\ Department of Mathematics \\ University of California \\ Davis, CA 95616 \\ USA}
\email{wally@math.ucdavis.edu}
\begin{abstract}
 Assuming the extrinsic $Q$-curvature admits a decomposition into the Pfaffian, a scalar conformal submanifold invariant, and a tangential divergence, we prove that the renormalized area of an even-dimensional minimal submanifold of a Poincar\'e--Einstein manifold can be expressed as a linear combination of its Euler characteristic and the integral of a scalar conformal submanifold invariant.
 We derive such a decomposition of the extrinsic $Q$-curvature in dimensions two and four, thereby recovering and generalizing results of Alexakis--Mazzeo and Tyrrell, respectively.
 We also conjecture such a decomposition for general natural submanifold scalars whose integral over compact submanifolds is conformally invariant, and verify our conjecture in dimensions two and four.
 Our results also apply to the area of a compact even-dimensional minimal submanifold of an Einstein manifold. 
\end{abstract}
\maketitle

\tableofcontents

\section{Introduction}
\label{sec:intro}

Poincar\'e--Einstein manifolds are asymptotically hyperbolic $\pedim$-manifolds
with constant Ricci curvature $\Ric = -\pedimminusoneparen g$;
see \cref{subsec:asymptotics} for a precise definition.  The motivating
example is the Poincar\'e ball model  
of hyperbolic space.
They have a boundary at infinity which inherits a conformal class of metrics.
Poincar\'e--Einstein manifolds were introduced by Fefferman and Graham~\cite{FeffermanGraham1985} as a means of studying conformal geometry by
associating a single Riemannian metric to a conformal class, and have been
intensely studied from that point of view.
Later it was realized that their
geometry is precisely that underlying the AdS/CFT correspondence in
physics~\cites{Maldacena1998,GubserKlebanovPolyakov1998,Witten1998,HenningsonSkenderis1998}. 

The boundary problem of finding a minimal submanifold of hyperbolic
space with prescribed boundary at infinity has also been much studied,
beginning with work of Anderson~\cites{Anderson1982,Anderson1983}.   
In the AdS/CFT correspondence, such minimal submanifolds $Y^{\pesubmfdim}$ of 
Poincar\'e--Einstein manifolds arise as dual objects to submanifold
observables in the boundary theory~\cite{GrahamWitten1999}.  


Poincar\'e--Einstein $\pedim$-manifolds have infinite volume.  But 
they have a renormalized volume $\mV$, which is an invariant when
\pedimiseven.  Similarly, a minimal submanifold $Y^{\pesubmfdim}$ 
with prescribed boundary at infinity has infinite area, but there is a 
renormalized area $\mA$ which is an invariant when \pesubmfdimiseven.  The 
renormalized volume and renormalized area both arose first in the physics  
literature and they are now regarded as fundamental geometric invariants;
see~\cite{Graham2000} for further references and discussion.
One would like to determine the relationship between these and other geometric invariants of Poincar\'e--Einstein manifolds and their minimal submanifolds.

Chang, Qing, and Yang~\cite{ChangQingYang2006} derived a formula of 
Gauss--Bonnet type for the renormalized volume:
if $(X^{\pedim},g_+)$ is 
Poincar\'e--Einstein with \pedimevengeqfour, then
\begin{equation}\label{eqn:cqygb}
  \mV =  c_{\pedimplaceholder} \chi(X) + \int_X\mZ\dvol ,\qquad\qquad
  c_{\pedimplaceholder} = \frac{(-2\pi)^{\pedimparen/2}}{\pedimminusoneparen!!}.
\end{equation}
Here $\chi(X)$ denotes the Euler characteristic of $X$ and $\mZ$ is a 
natural scalar that is pointwise conformally invariant of weight $\minuspedim$.
The result for \pedimequalsfour\ was first proved by Anderson~\cite{Anderson2001b}.  In this   
case $\mZ= - |W|^2/24$, where $W$ denotes the Weyl tensor.  Conformal  
invariance implies convergence of the integral in Equation~\eqref{eqn:cqygb}:  the 
expression $\mZ\,\dvol$ does not change upon conformally rescaling the
metric, so it equals the same expression when evaluated on a compactification
of~$g_+$.  

A main ingredient in Chang, Qing, and Yang's derivation of Equation~\eqref{eqn:cqygb} is a result
of Alexakis~\cite{Alexakis2012}, motivated by a conjecture of Deser and
Schwimmer~\cite{DeserSchwimmer1993},
that establishes a decomposition of any natural scalar whose 
integral over compact manifolds is conformally invariant.
Namely, any such scalar $I$ on manifolds of even dimension $\pedim$ can be written  
\begin{equation}\label{eqn:Adecomp}
I = c\,\Pf + \mZ_I + \divsymb V,  
\end{equation}  
where $c \in \bR$, $\Pf$ denotes the Pfaffian of the curvature tensor,
$\mZ_I$ is a natural scalar that is pointwise conformally invariant of
weight $\minuspedim$, and $V$ is a natural vector field.
Branson's~\cites{Branson1995,ChangEastwoodOrstedYang2008} critical $Q$-curvature is a natural scalar that plays a fundamental role in conformal geometry.
It is not conformally invariant, but it transforms by addition of a divergence under conformal change.
Hence its integral over compact manifolds is conformally invariant.
Chang, Qing, and Yang applied Alexakis' result with $I=Q$ in their proof of Equation~\eqref{eqn:cqygb}.
The proof shows that the conformal
invariant in \eqref{eqn:cqygb} is given by
$\mZ= \frac{(-1)^{\pedimparen/2}}{\pedimminusoneparen!} \mZ_Q$.  

Our main result is an analogue of Equation~\eqref{eqn:cqygb} for even-dimensional
minimal submanifolds of Poincar\'e--Einstein manifolds, assuming a special
case of a submanifold version of Alexakis' result.  We first formulate this 
submanifold version as a conjecture.
Our convention is that manifolds and submanifolds are without boundary unless otherwise specified.
Natural submanifold scalars and tensors are defined in \cref{subsec:natural}.

\begin{conjecture}\label{submanifold-alexakis}
  Let $\submfdim,\mfdim\in \bN$ with $\submfdim$ even and \pecodimpositive.  Suppose that $I$ is a
  natural scalar on $\submfdim$-dimensional submanifolds $Y$ of  
  $\mfdim$-dimensional Riemannian manifolds $(X,g)$ whose integral over compact
  $Y$ is invariant under conformal rescaling of~$g$.  Then there is a natural submanifold scalar 
  $\mW_I$ that is pointwise conformally invariant of weight $-\submfdim$, a
  natural submanifold vector field $V$, and a constant $c \in \bR$ so that  
\begin{equation}\label{eqn:submanifold-decomposition}
I = c\,\oPf + \mW_I +\odiv V. 
\end{equation}
Here $\odiv$ and $\oPf$ are the divergence operator and the
Pfaffian of the Riemannian curvature tensor, respectively, for the metric 
induced on $Y$ by $g$.  
\end{conjecture}

An extrinsic submanifold version of Branson's critical $Q$-curvature was
defined by Case, Graham, and Kuo~\cite{CaseGrahamKuo2023}.  For $\submfdim,\mfdim\in \bN$ with $\submfdim$ even
and $\mfdim>\submfdim$, this $Q$-curvature is a natural submanifold scalar whose 
integral over compact $Y$ is conformally invariant and which has a nice factorization for minimal submanifolds of Einstein manifolds.  It will be  
reviewed in \cref{sec:q-bg}.  Henceforth in this paper, by $Q$ we will
always mean this critical extrinsic submanifold $Q$-curvature. 

Our main theorem is the following.

\begin{theorem}\label{main-theorem}
Let $\pesubmfdimplaceholder,\pedimplaceholder \in \bN$ with \pesubmfdimeven\ and \pecodimpositive.
Suppose that
\eqref{eqn:submanifold-decomposition} holds for $I=Q$ on $\pesubmfdimparen$-dimensional submanifolds of $\pedimparen$-manifolds, with $\mW_Q$ and $V$ as in the statement of 
Conjecture~\ref{submanifold-alexakis}.
If $(X^{\pedim},g_+)$ is a
Poincar\'e--Einstein manifold,
$Y^{\pesubmfdim}$ is a smooth compact manifold with boundary, and
$i \colon Y\rightarrow (X,g_+)$ is a polyhomogeneous minimal immersion, then
\begin{equation}\label{eqn:gb}
\mA = c_{\pesubmfdimplaceholder} \chi(Y) + \frac{(-1)^{\pesubmfdimparen/2}}{\pesubmfdimminusoneparen!}\int_Y\mW_Q\darea ,
\end{equation}
where $\mA$ denotes the renormalized area of $Y$ and $c_{\pesubmfdimplaceholder}$ is as in Equation~\eqref{eqn:cqygb}.    
\end{theorem}

\noindent
Importantly, the conformal invariance of $\mW_Q$ implies that  
$\int_Y\mW_Q\darea$ converges.
See \cref{subsec:asymptotics} for the definition of polyhomogeneous minimal immersions.

Explicit formulas for $Q$ for $\submfdim=2,4$ are derived
in~\cite{CaseGrahamKuo2023}.
The formula~\cite{CaseGrahamKuo2023}*{Equation~(5.14)} for $k=2$ already exhibits a decomposition
of the form~\eqref{eqn:submanifold-decomposition}:
\[
Q=\oPf+\mW_Q,\qquad\qquad
\mW_Q=\frac12\lv \tfss \rv^2 - W^T.
\]
Here $\lv\tfss\rv^2$ is the squared norm of the trace-free part of the
second fundamental form and $W^T:=W(e_1,e_2,e_1,e_2)$ is the tangential
component of the background Weyl tensor, where $e_1$, $e_2$ is an
orthonormal basis for $TY$.  In this case, the resulting Gauss--Bonnet
formula \eqref{eqn:gb} was derived by Alexakis and
Mazzeo~\cite{AlexakisMazzeo2010} using another method.

When $\submfdim = 4$, the formula for $Q$ in~\cite{CaseGrahamKuo2023}*{Equation~(5.14)} is not
decomposed in the
form~\eqref{eqn:submanifold-decomposition}.  The main issue is  
recognizing $\mW_Q$, since the obvious scalar conformal submanifold
invariants   
constructed from the trace-free second fundamental form and the Weyl tensor 
are not sufficient to produce such a decomposition for $Q$.
In Subsection~\ref{subsec:invariants/construct} below we identify four non-obvious scalar
conformal submanifold invariants of weight $-4$ for submanifolds of general
dimension and codimension.  One of them, which we denote $\Wm$,
takes the following form when $\submfdim = 4$: 
\begin{equation*}
 \Wm = -3 \oDelta\trFi - 4\lp\Fi - \trFi i^\ast g, \oSch + \Fi \rp -
 2\onabla^\alpha\big(\mathcal{C}_\alpha
 -\Di^{\beta\alpha^\prime}\tfss_{\alpha\beta\alpha^\prime}\big)
 -\frac{2}{\mfdim-4}\mB_\alpha{}^\alpha + 2\lv\Di\rv^2. 
\end{equation*}
Here $\Fi$ is the conformally invariant Fialkow
tensor given in Equation~\eqref{eqn:defn-Fialkow}, $\trFi$ is 
its trace, $\oSch$ is the Schouten tensor for the induced metric, $\tfss$
the trace-free second fundamental form, and 
\begin{align*}
 \Di_{\alpha\alpha^\prime} & := \Sch_{\alpha\alpha^\prime} - \onabla_\alpha H_{\alpha^\prime} , \\
 \mC_{\alpha} &:
 =C_{\beta\alpha}{}^\beta-H^{\alpha'}W_{\beta\alpha}{}^\beta{}_{\alpha'} ,\\ 
 \mB_\alpha{}^\alpha & := B_\alpha{}^\alpha -
 2(\mfdim-4)H^{\alpha^\prime}C_{\alpha\alpha^\prime}{}^{\alpha} +
 (\mfdim-4)H^{\alpha^\prime}H^{\beta^\prime}W_{\alpha^\prime\alpha\beta^\prime}{}^{\alpha},
\end{align*}
where $\Sch$ is the background Schouten tensor, $H$ is the mean curvature
vector, and $B$, $C$, and $W$ are the background Bach, Cotton, and Weyl tensors,
respectively.
(Our notational conventions are described in \cref{sec:bg}.)
In Subsection
\ref{subsec:invariants/q} we show that
the formula \cite{CaseGrahamKuo2023}*{Equation~(5.14)} for $Q$ when
  $\submfdim =4$ can be rewritten in the form 
  \eqref{eqn:submanifold-decomposition} with 
  \begin{equation*}
  \mathcal{W}_Q=- \frac{1}{4}\lv \oW \rv^2 + \Wm + 2\lv \Fi\rv^2 - 2\trFi^2 ,
 \end{equation*}
where $\lv\oW\rv^2$ is the squared length of the Weyl tensor of the induced metric.
 Theorem~\ref{main-theorem} then immediately implies the following: 
  
\begin{corollary}
 \label{renormalized-area-special}
 Let $i \colon Y^{4} \to (X^{\pedim},g_+)$ be a polyhomogeneous minimal
 immersion into a Poincar\'e--Einstein manifold of dimension \pedimgeqfive.  Then
 \begin{equation}
  \label{eqn:minimal-renorm-gbc4}
  \mA  =  \frac{4\pi^2}{3}\chi(Y) -  \frac{1}{6}\int_Y \left( \frac{1}{4}\lv 
  \oW\rv^2 - \Wm - 2\lv\Fi\rv^2 + 2\trFi^2 \right) \darea.
 \end{equation}
\end{corollary}

Tyrrell~\cite{Tyrrell2022} derived a formula equivalent to Equation~\eqref{eqn:minimal-renorm-gbc4} in the \pedimequalsfive\ hypersurface case.
Tyrrell's
integrand is a natural submanifold scalar that agrees with our integrand 
for minimal hypersurfaces of Poincar\'e--Einstein manifolds.  But his integrand is not  
conformally invariant and an analysis of the asymptotics of its summands was required to establish convergence of the integral.  See  
Remark~\ref{rk:tyrrell} below for further discussion.   

Decompositions of the
form~\eqref{eqn:submanifold-decomposition} are not unique when $\submfdim = 4$.
In Proposition~\ref{conformally-invariant-divergence} we identify two
scalar conformal submanifold invariants $\mK_1$ and $\mK_2$ of weight $-4$ that are
divergences when $\submfdim = 4$.  Theorem~\ref{main-theorem}
implies that the Gauss--Bonnet formula~\eqref{eqn:gb} holds for any
$\mathcal{W}_Q$ that arises in such a decomposition for $Q$.  

The main reason that it is important to have a conformally invariant
integrand in Equations \eqref{eqn:cqygb} and \eqref{eqn:gb} is to
ensure convergence of the integral.  The formulas then give a
concrete expression for the outcome of the renormalization process applied
to volume or area.  There are other formulas of Gauss--Bonnet type in this  
setting, for instance in~\cite{Albin2009} for Poincar\'e--Einstein manifolds and in~\cite{TaylorToo2020} for minimal
submanifolds.  But these formulas involve renormalized
quantities other than just the volume or area.  A drawback of Equations~\eqref{eqn:cqygb}
and \eqref{eqn:gb} is that they require identifying
a conformal invariant in the decomposition of $Q$-curvature to become
explicit.  

Our proof of Theorem~\ref{main-theorem} follows the same outline as the
proof in~\cite{ChangQingYang2006} of Equation~\eqref{eqn:cqygb}.
But our geometric situation is more complicated and we introduce two
modifications which simplify the argument.  Important 
properties of the extrinsic $Q$-curvature which enter are its linear transformation law 
in terms of a critical extrinsic GJMS operator and the fact that both of these objects have 
an explicit factorization for minimal submanifolds of Einstein manifolds. 
These properties are established in \cite{CaseGrahamKuo2023}.  The other 
main ingredient in the proof is the existence and properties of what we
call the scattering potential.  This is a solution of a linear scalar
equation on the submanifold whose asymptotic expansion contains a term which produces 
the renormalized area when integrated over the boundary.  The   
scattering potential was introduced in \cite{FeffermanGraham2002} in the
setting of renormalized volume.  The scattering potential determines a
specific compactification of the metric induced on $Y$ by $g_+$, called the
scattering compactification, whose $Q$-curvature vanishes identically.
(In the setting of Poincar\'e--Einstein manifolds, this is
sometimes called the Fefferman--Graham compactification.)
In our formulation of the argument, the 
decomposition of $Q$ is applied to a geodesic compactification and the renormalized
area arises via the integrated asymptotics of the scattering potential upon
conformally transforming to the scattering compactification.

The following theorem is a result analogous to \cref{main-theorem} for compact minimal submanifolds of Einstein manifolds, in  
which the renormalized area is replaced by the area.  In this case there
are no convergence issues and the result follows immediately upon
integrating the decomposition 
\eqref{eqn:submanifold-decomposition} for the extrinsic 
$Q$-curvature determined by the background Einstein metric itself.  

\begin{theorem}\label{closed case}
Let $\pesubmfdimplaceholder , \pedimplaceholder \in \bN$ with \pesubmfdimeven\ and \pecodimpositive.
Suppose that
\eqref{eqn:submanifold-decomposition} holds for $I=Q$ on $\pesubmfdimparen$-dimensional submanifolds of $\pedimparen$-manifolds, with $\mW_Q$ and $V$ as in the statement of 
Conjecture~\ref{submanifold-alexakis}.
If $(X^{\pedim},g)$ satisfies $\Ric(g)=\lambda \pedimminusoneparen g$ and if 
$i \colon Y^{\pesubmfdim} \rightarrow (X,g)$ is a minimal immersion of a compact manifold, then
\begin{equation}\label{eqn:closed gb}
\lambda^{\pesubmfdimparen/2} A = \frac{(2\pi)^{\pesubmfdimparen/2}}{\pesubmfdimminusoneparen!!} \chi(Y) + \frac{1}{\pesubmfdimminusoneparen!}\int_Y\mW_Q\darea , 
\end{equation}
where $A$ denotes the area of $Y$.    
\end{theorem}

In particular, if $\dim Y = 2$, then
\begin{equation}
 \label{eqn:minimal-gbc2}
 \lambda A = 2\pi\chi(Y) + \int_Y \left( \frac{1}{2}\lv\tfss\rv^2 - W^T \right) \darea , 
 \end{equation}
and if $\dim Y = 4$, then
\begin{equation}
 \label{eqn:minimal-gbc4}
 \lambda^2A = \frac{4\pi^2}{3}\chi(Y) - \frac{1}{6}\int_Y \left( \frac{1}{4}\lv \oW\rv^2
- \Wm - 2\lv\Fi\rv^2 + 2\trFi^2 \right) \darea .
\end{equation}
Equation~\eqref{eqn:minimal-gbc2} follows immediately by integrating the Gauss equation~\eqref{eqn:gcJ}.
It has been used, for example, in the classification of minimal surfaces in $S^3$ of index at most five~\cite{Urbano1990}*{p.\ 991} and in the study of a class of immersed surfaces in self-dual Einstein manifolds~\cite{Friedrich1984}*{Section~2}.

Our second result is the verification of \cref{submanifold-alexakis}
for submanifolds of dimension two or four.  

\begin{theorem}
 \label{low-dimension-submanifold-alexakis}
 \Cref{submanifold-alexakis} is true when $\submfdim = 2$ and $\submfdim = 4$.
\end{theorem}

The main ingredient in our proof of
\cref{low-dimension-submanifold-alexakis} is the identification of a
well-chosen spanning set for the space of natural submanifold scalars of
weight $\minussubmfdim$
modulo scalar conformal submanifold invariants and tangential divergences.
When $\submfdim = 4$, we need to use two non-obvious scalar conformal
submanifold invariants $\mI$ and $\mJ$ identified in Subsection~\ref{subsec:invariants/construct} 
to eliminate potential elements from the spanning set.  
The cardinality of our spanning set is 3 when $\submfdim = 2$ and 33 when 
$\submfdim = 4$.  We then calculate directly the 
conformal variation of the integral of a linear combination of the 
elements of this set to argue that if this integral is conformally
invariant, then the linear combination must be proportional to the Pfaffian modulo
a conformal invariant.   

Various cases of \cref{submanifold-alexakis} have been considered in the  
literature.  Mondino and Nguyen~\cite{MondinoNguyen2018} discussed
natural submanifold scalars whose integrals are conformally invariant.
They gave a characterization of a family of such scalars involving the curvature and second fundamental form, but not their derivatives, for surfaces of codimension at most two and hypersurfaces in general dimension. 
Juhl~\cite{Juhl2022} verified \cref{submanifold-alexakis} when $\submfdim = 4$
for the singular Yamabe extrinsic $Q$-curvature on hypersurfaces
derived by Gover and Waldron~\cite{GoverWaldron2015} and conjectured the
decomposition~\eqref{eqn:submanifold-decomposition} of this invariant for
even $\submfdim > 4$.    

The invariants that enter the Gauss--Bonnet formula~\eqref{eqn:gb} are even;
i.e.\ they are unchanged under changes of orientation.
Throughout this paper we only consider even invariants.

This paper is organized as follows:

In \cref{sec:q-bg} we define natural submanifold tensors and differential
operators and review properties of the extrinsic $Q$-curvature and the
scattering compactification which we will use in the proof of
\cref{main-theorem}. 

In \cref{sec:main-thm} we prove Theorems~\ref{main-theorem} and \ref{closed
  case}.  

In \cref{sec:bg} we review background material concerning Riemannian and
conformal submanifold geometry and fix our notational conventions.
In Subsection~\ref{subsec:conformal-submanifold-invariants} we 
also introduce three tensors $\mP$, $\mC$, and $\mB$ that are modifications of 
projections of the background Schouten, Cotton, and Bach tensors,
respectively.  These tensors play an important role in our subsequent analysis; one
reason for their importance is that their conformal transformation laws
only involve tangential derivatives of the conformal factor.  

\cref{sec:invariants} contains three subsections.  In
Subsection~\ref{subsec:invariants/construct} 
we introduce four non-obvious scalar conformal submanifold invariants $\mK_1$, $\mK_2$,
$\mI$, $\mJ$ of weight $-4$ in general dimension and codimension.  
In Subsection~\ref{subsec:invariants/q} we use $\mI$ to derive a
decomposition in 
general dimension of the fourth-order $Q$-curvature defined in 
\cite{CaseGrahamKuo2023}; this specializes to the 
form \eqref{eqn:submanifold-decomposition} when $\submfdim = 4$.
\Cref{renormalized-area-special} and Equation~\eqref{eqn:minimal-gbc4} are immediate consequences.
In Subsection~\ref{subsec:invariants/comparison} we state without proof how
some invariants previously
found in special cases by Blitz, Gover, and
Waldron~\cite{BlitzGoverWaldron2021}, Juhl~\cite{Juhl2022}, Astaneh and  
Solodukhin~\cite{AstanehSolodukhin2021}, and Chalabi, Herzog,
O'Bannon, Robinson, and Sisti~\cite{ChalabiHerzogOBannonRobinsonSisti2022}
can be written in terms of our invariants constructed in Subsection~\ref{subsec:invariants/construct},
thus generalizing their constructions to general dimension and codimension. 

In \cref{sec:alexakis} we prove \cref{low-dimension-submanifold-alexakis}.
As indicated above, the proof uses our invariants $\mI$ and $\mJ$ from
Subsection~\ref{subsec:invariants/construct}.  

In \cref{sec:extra} we give the details of the computations omitted in
Subsection~\ref{subsec:invariants/comparison}. 

\section{Extrinsic $Q$-curvature and scattering compactification}
\label{sec:q-bg}

\subsection{Natural submanifold tensors and extrinsic $Q$-curvature}
\label{subsec:natural}

We will be dealing with immersions $i \colon Y^{\submfdim}\to (X^{\mfdim},g)$ into 
Riemannian manifolds.  The pullback bundle $i^*TX$ admits the
$g$-orthogonal splitting $i^*TX=TY\oplus NY$.
We use the induced metrics on $TY$ and $NY$ to identify these bundles with their respective duals, $T^\ast Y$ and $N^\ast Y$.
The Levi-Civita connection 
$\nabla$ of $g$ determines connections $\onabla$ on $TY$ and $NY$ by
projection.
On $TY$ this connection is the Levi-Civita
connection of $i^*g$.  We use bars to denote intrinsic quantities on~$Y$.
For example, $\Rm$ denotes the curvature tensor of $g$ and $\oRm$ the
curvature tensor of $i^*g$;
also, $\Delta$ denotes the Laplacian of $g$ and $\oDelta$ the
Laplacian of $i^*g$.  Our sign convention is $\Delta = \sum \partial_i^2$
on Euclidean space.  The second fundamental form $L:S^2TY\to NY$ is  
defined by $L(U,V)=(\nabla_UV)^\perp$.  
If $i \colon Y^{\submfdim} \to (X^{\mfdim},[g])$ is an immersion into a conformal manifold,
then there is an induced conformal class $i^*[g]$ on $Y$.  In $i^*[g]$, we allow
rescalings by arbitrary functions in $C^\infty(Y)$, not just pullbacks of
functions in $C^\infty(X)$.  

Let $\submfdim$, $\mfdim \in \bN$ with \codimpositive.
A natural tensor on $\submfdim$-dimensional
submanifolds of $\mfdim$-dimensional Riemannian manifolds, or a \defn{natural submanifold tensor}, is an assignment to
each immersion $i \colon Y^{\submfdim} \to (X^{\mfdim},g)$ of a  
section of $(T^*Y)^{\otimes r}\otimes (N^*Y)^{\otimes s}$ for some integers $r,s \geq 0$ which can be
expressed as an $\bR$-linear combination of partial contractions of tensors  
\begin{equation}\label{eqn:partial contraction}
  \pi_1(\nabla^{M_1}\Rm) \otimes \dotsm \otimes \pi_p(\nabla^{M_p}\Rm)
  \otimes \onabla^{N_1}L \otimes \dotsm \otimes \onabla^{N_q}L \otimes \pi(g^{\otimes P}) .
\end{equation}
Here $M_j$, $N_j$, and $P$ denote powers, and $\pi$ and $\pi_j$ denote restriction to $Y$ followed
by projection to either $TY$ or $NY$ in each index.
The tensor displayed at~\eqref{eqn:partial contraction} is viewed as covariant in all indices
(that is, with all indices lowered), and the contractions
are taken with respect to the metrics induced by $g$ on $TY$ and $NY$ for 
partial pairings of tangential and normal indices.
The special case $r,s=0$ defines \defn{natural submanifold scalars}.
A natural submanifold
tensor $T$ has \defn{weight} $w\in \mathbb{R}$ if $T^{c^2 g}=c^wT^g$
for all $c>0$.  A natural submanifold tensor $T$ is \defn{conformally
invariant} of weight $w$, or a \defn{conformal submanifold invariant}, if
$T^{e^{2\Upsilon}g} = e^{wi^\ast\Upsilon}T^g$ for all immersions
$i\colon Y^{\submfdim} \to (X^{\mfdim},g)$ and all
$\Upsilon \in C^\infty(X)$.  

A (linear) \defn{natural submanifold differential
operator} on $\submfdim$-dimensional submanifolds of $\mfdim$-dimensional Riemannian
manifolds is an assignment to each immersion $i \colon Y^{\submfdim} \to (X^{\mfdim},g)$ of a
differential operator
\begin{equation*}
 D \colon C^\infty(Y) \to C^\infty(Y;(T^*Y)^{\otimes
  r}\otimes (N^*Y)^{\otimes s})
\end{equation*}
for some integers $r,s \geq 0$ which can be
expressed as 
an $\bR$-linear combination of partial contractions of terms of the form  
\begin{equation*}
  \pi_1(\nabla^{M_1}\Rm) \otimes \dotsm \otimes \pi_p(\nabla^{M_p}\Rm)
  \otimes \onabla^{N_1}L \otimes \dotsm \otimes \onabla^{N_q}L \otimes \pi(g^{\otimes P}) \otimes 
  \onabla^Q .
\end{equation*}
(In this paper, we only need to consider natural operators acting on scalars.)
For both natural submanifold tensors and natural submanifold differential operators, some of the free
indices can be raised to view the tensor as contravariant in these
indices.

The paper~\cite{CaseGrahamKuo2023} uses a different definition of natural submanifold scalar differential operators, and hence of natural submanifold scalars, that is formulated in terms of isometry invariant assignments depending polynomially on the metric.
Graham and Kuo extended~\cite{GrahamKuo2023} that definition to tensors and showed that it is equivalent to the definition in this paper.

Case, Graham, and Kuo~\cite{CaseGrahamKuo2023} constructed \defn{extrinsic GJMS operators} and
\defn{extrinsic $Q$-curvatures} associated to a submanifold of a 
conformal manifold which satisfy covariance relations under 
conformal change.  The objects of critical weight
are the ones that are relevant to the proof of \cref{main-theorem}.
As discussed in~\cite{CaseGrahamKuo2023}, the construction applies to
immersed submanifolds and the conclusion can be formulated as follows: 

\begin{proposition}[\cite{CaseGrahamKuo2023}*{Theorem~1.1}]
\label{properties}
Let $(X^{\mfdim},[g])$, $\mfdim\geq3$, be a conformal manifold and let $i \colon Y^{\submfdim}\to X^{\mfdim}$ be an   
immersion with $2 \leq \submfdim < \mfdim$ and $\submfdim$ even.  For each
$h\in i^*[g]$, there is a formally self-adjoint differential operator 
$P_{\submfdim} \colon C^\infty(Y) \to C^\infty(Y)$ and a scalar function
$Q \in C^\infty(Y)$ such that $P_{\submfdim}(1)=0$,    
\begin{equation}
\label{eqn:gjms-leading-term}
P_{\submfdim} = (-\oDelta)^{\submfdim/2} + \lots ,
\end{equation}
and if $\widehat{h} = e^{2\Upsilon}h$ with $\Upsilon\in C^\infty(Y)$, then     
\begin{equation}
\label{eqn:conformal-transformation-law}
\begin{split}
   e^{\submfdim\Upsilon}P^{\widehat{h}}_{\submfdim} & = P^h_{\submfdim}, \\
   e^{\submfdim\Upsilon}Q^{\widehat{h}} & = Q^h + P^h_{\submfdim}(\Upsilon) . 
\end{split}
\end{equation}
\end{proposition}

\noindent
In Equation~\eqref{eqn:gjms-leading-term}, $\lots$ denotes a differential operator on $Y$ of order at most $k-2$.

Note that $P_\submfdim$ and $Q$ depend on the conformal class $[g]$ on $X$ and
the choice of $h \in i^\ast[g]$.  We can also view $P_{\submfdim}$ and $Q$ as determined simply by
a choice of metric $g$ on $X$, since $g$ determines $[g]$ and the induced
metric $h=i^*g$.  When viewed this way, $P_{\submfdim}$ is a natural submanifold
differential operator and $Q$ is a natural submanifold
scalar.

\Cref{properties} implies that if $Y$ is compact, then the total $Q$-curvature is conformally invariant:
\begin{equation*}
\int_Y Q^{\widehat{h}} \darea_{\hh} = \int_Y Q^h \darea_{h}.  
\end{equation*}

There are many operators and related curvatures which satisfy the conclusions of \cref{properties};
see~\cite{GoverWaldron2017,GoverWaldron2015} for a different family and the introduction of~\cite{CaseGrahamKuo2023} for a detailed comparison of these two families.
An essential property of the operators and $Q$-curvatures constructed 
in~\cite{CaseGrahamKuo2023} is their factorization for minimal submanifolds
of Einstein manifolds.  
We rely on the specialization of this fact in  the critical-order case:  

\begin{proposition}[\cite{CaseGrahamKuo2023}*{Theorem~1.2}]
  \label{factorization}
  Let $i \colon Y^{\submfdim} \to (X^{\mfdim},g)$ be a minimal immersion, where $\submfdim$ is
  even and $g$ is Einstein with $\Ric(g) = \lambda (\mfdim-1)g$.
  Then
\begin{align*}
  P_{\submfdim} & = \prod_{j=1}^{\submfdim/2} \left( -\oDelta
  + \lambda (\tfrac{\submfdim}{2}+j-1)(\tfrac{\submfdim}{2}-j) \right) , \\
Q & = \lambda^{\submfdim/2} (\submfdim-1)!.
\end{align*}
\end{proposition}

\subsection{Asymptotics and scattering compactification}
\label{subsec:asymptotics}
If $X$ is a manifold with nonempty boundary, by a \defn{boundary identification}
we mean a diffeomorphism from a collar neighborhood of $\partial X$
to $\partial X \times [0,\epsilon)_r$ for some $\epsilon >0$, for which  
$\partial X\ni p \to (p,0)$.  A function or tensor defined on $X$ is
\defn{polyhomogeneous} if in a boundary identification it has an asymptotic 
expansion in powers of $r$ and nonnegative integral powers of $\log r$
whose coefficients are smooth on $\partial X$.  This is an informal formulation;
see, for example, \cite{Grieser2001} for
a detailed presentation.  This polyhomogeneity condition is independent of
the choice of boundary identification.

Let $(X^{\pedim},g_+)$, \pedimgeqthree, be a \defn{Poincar\'e--Einstein manifold} with
conformal infinity $(\partial X,\kc)$.
That is, $X$ is a compact connected manifold with nonempty boundary, $g_+$ is a complete metric in the interior of $X$ for which $\Ric_{g_+} = -(n-1)g_+$, and $r^2g_+$ extends to a $C^{2,\alpha}$-metric on $X$ such that $\kc := \bigl[ r^2g_+ \rv_{T\partial X} \bigr]$ is a smooth conformal class whenever $r$ is a defining function for $\partial X$.
Poincar\'e--Einstein manifolds are necessarily \defn{asymptotically hyperbolic};
i.e.\ $\lvert dr \rvert_{r^2g_+} = 1$ along $\partial X$ for any defining function $r$ for $\partial X$.
A representative $g_{(0)} \in \kc$ uniquely determines~\cite{GrahamLee1991}*{Lemma~5.2} a defining function $r$ near $\partial X$ such that
$r^2g_+|_{T\partial X}=g_{(0)}$ and
$\lv dr \rv^2_{r^2g_+}=1$.
We call $r$ the \defn{geodesic defining function} and $r^2g_+$ the \defn{geodesic compactification}
determined by $g_{(0)}$.  The metric $g_{(0)}$ also determines, for some
$\epsilon>0$, a boundary identification with respect to which    
\begin{equation*}
r^2g_+ = dr^2 + g_r .
\end{equation*}
Here~\citelist{\cite{ChruscielDelayLeeSkinner2005}*{Theorem~A}
  \cite{BiquardHerzlich2014}*{Th\'eor\`eme~1} } $g_r$ is
a one-parameter family of metrics on $\partial X$ which is smooth if \penologterms, and is polyhomogeneous if \pelogterms.  The    
expansion of $g_r$ has the form~\cite{FeffermanGraham2012}*{Theorem~4.8}
\begin{equation}
 \label{eqn:pe-expansion}
 \begin{aligned}
 g_r & = g_{(0)} + g_{(2)}r^2 + \dotsm, && \text{\pedimequalsthree} , \\
 g_r & = g_{(0)} + \dotsm + g_{(\pedimminustwo)}r^{\pedimminustwo} + g_{(\pedimminusone)}r^{\pedimminusone} + \dotsm, && \text{\pedimeven} , \\
 g_r & = g_{(0)} + \dotsm + g_{(\pedimminusthree)}r^{\pedimminusthree} + \kappa r^{\pedimminusone}\log r
 + g_{(\pedimminusone)}r^{\pedimminusone}+\dotsm,  && \text{otherwise} ,
 \end{aligned}
\end{equation}
where the coefficients $g_{(j)}$ and $\kappa$ are smooth symmetric 2-tensors on
$\partial X$.  Terms $r^j$ do not occur for odd $j<\pedimminusone$.    

Let $(X^{\pedim},g_+)$ be Poincar\'e--Einstein and let $Y^{\pesubmfdim}$ be a compact
manifold with nonempty boundary.  Let $i \colon Y \to X$ be an
immersion that is $C^\infty$ on the interior~$\mathring{Y}$, is a $C^1$ embedding in
a neighborhood of $\partial Y$, satisfies $i(\partial Y) \subset \partial X$ with $i(Y)$ transverse to $\partial X$, and 
is such that $i\rv_{\partial Y}$ is a $C^\infty$ embedding into  
$\partial X$.  If $(x^\alpha,u^{\alpha'})$, $1\leq \alpha\leq \pesubmfdimminusone$,
$1\leq \alpha'\leq \pecodim$, are 
local coordinates on $\partial X$ with $i(\partial Y) = \{u=0\}$, then in
the boundary identification induced by $g_{(0)}\in \kc$,  
we may write $i(Y)$ in the form $i(Y)=\{u=u(x,r)\}$.  
We say that $i$ is a \defn{polyhomogeneous immersion} if the graphing
map $u(x,r)$ is polyhomogeneous.     
This condition is independent of the choice of coordinates $(x,u)$ on
$\partial X$.  Since 
$r^2g_+$ is itself polyhomogeneous, the
transition maps relating boundary identifications determined by different 
choices of $g_{(0)}$ themselves have polyhomogeneous expansions, so the
condition that $i$ is a polyhomogeneous immersion is also independent of
the choice of $g_{(0)}$.  

Our analysis requires that the minimal immersions under consideration 
are polyhomogeneous (at least to some order).  
We will simply assume this to be the case.  There are results establishing
the polyhomogeneity of minimal submanifolds under certain initial 
regularity hypotheses~\citelist{  
  \cite{AlexakisMazzeo2010}*{Proposition~2.2}
  \cite{HanJiang2023}*{Theorem~1.1}
  \cite{Marx-Kuo2021}*{Theorem~3.1} }.

We need a global invariant description of the asymptotics of minimal
submanifolds.  We use the normal exponential map of $\Sigma := i(\partial Y)$ for
this purpose as in~\cite{GrahamReichert2017}.
Let $i:Y^{\pesubmfdim}\to (X^{\pedim},g_+)$ be a polyhomogeneous immersion.   
Choose $g_{(0)}\in \kc$.  Use the boundary 
identification determined by $g_{(0)}$ to identify a neighborhood of
$\partial X$ in $X$ with $\partial X \times [0,\epsilon)_r$.  For $r\geq 0$
small, let $\Sigma_r\subset \partial X$ denote the slice of $i(Y)$ at
height $r$; i.e.\ $i(Y)\cap (\partial X \times \{r\}) = \Sigma_r\times \{r\}$.
Then $\Sigma_r$ is a smooth submanifold of $\partial X$ of dimension $\pesubmfdimminusone$
and $\Sigma_0=\Sigma$.  The normal exponential map of
$\Sigma$ with respect to $g_{(0)}$, denoted $\exp_\Sigma$, defines a 
diffeomorphism of a neighborhood of the zero section in $N\Sigma$ to a
neighborhood of $\Sigma$ in $\partial X$.  So there is a unique section 
$U_r\in\Gamma(N\Sigma)$ near $r=0$ such that
$\exp_\Sigma\{U_r(p):p \in \Sigma\}=\Sigma_r$.  This defines a
polyhomogeneous 1-parameter family $U_r$ of sections of $N\Sigma$ for which
we have 
\begin{equation}\label{eqn:Y}
i(Y) = \bigl\{ \big(\exp_\Sigma U_r(p),r\big) : p\in \Sigma, r\geq 0 \bigr\}.
\end{equation}

The inverse normal exponential map can be used to define a boundary 
identification for $i(Y)$.  Define a diffeomorphism $\psi$ 
from a neighborhood of $\Sigma$ in $i(Y)$ to $\Sigma\times [0,\epsilon)$
for some $\epsilon >0$ by    
\[
\psi(q,r) := (\pi((\exp_\Sigma)^{-1}q),r),\qquad
(q,r)\in i(Y)\subset \partial X\times [0,\epsilon) ,
\]
where $\pi \colon N\Sigma \to \Sigma$ is the canonical projection.
Then $\psi$ is a boundary identification for $i(Y)$.  

Choose local coordinates 
$\{x^\alpha: 1\leq \alpha\leq \pesubmfdimminusone \}$ for $\Sigma$ and a local frame 
$\{e_{\alpha'}(x): 1\leq \alpha'\leq  \pecodim \}$ for $N\Sigma$.  
The map $\exp_\Sigma \big(u^{\alpha'}e_{\alpha'}(x)\big)\mapsto (x,u)$
defines a geodesic normal coordinate system for $\partial X$ near $\Sigma$
with respect to which $\Sigma = \{u=0\}$.
Extend the coordinates $(x,u)$ to
$\partial X \times [0,\epsilon_0)\subset X$   
to be constant in $r$.  In these coordinates, the diffeomorphism $\psi$ is
given by $\psi(x,u,r) = (x,r)$.
Express the section $U_r$ in~\eqref{eqn:Y} as
$U_r(x) = u^{\alpha'}(x,r)e_{\alpha'}(x)$; then $i(Y)$ is locally the   
graph $\{u=u(x,r)\}$.  A function defined on $i(Y)$ near $\partial X$ can
be uniquely extended to a neighborhood of $i(Y)$ in $X$ near $\partial X$
by making it independent of the $u$ variables.  Since varying $u$ gives the fibers
of $\pi:N\Sigma\to \Sigma$, this extension is independent of the choice of 
coordinates and is globally defined near $\partial X$. 

If $i \colon Y^{\pesubmfdim}\to (X^{\pedim},g_+)$ is a polyhomogeneous minimal immersion, the
form of its expansion can be formally calculated~\cites{GrahamWitten1999,GrahamReichert2017} from the minimal submanifold equation $H=0$.
Henceforth we assume that \pesubmfdimiseven.  In this case,
\begin{equation}
 \label{eqn:U-expansion}
U_r= U_{(2)}r^2+ U_{(4)}r^4 +\dotsm + U_{(\pesubmfdim)}r^{\pesubmfdim} + U_{(\pesubmfdimplusone)}r^{\pesubmfdimplusone} +
O(r^{\pesubmfdimplustwo}\log r),
\end{equation}
where the $U_{(j)}$ are global sections of $N\Sigma$.  
Equivalently, in local coordinates 
$(x,u,r)$ as above, the expansion of $u(x,r)$ takes the form   
\begin{equation}
 \label{eqn:u-expansion}
 u = u_{(2)}r^2 + u_{(4)}r^4 + \dotsm + u_{(\pesubmfdim)}r^{\pesubmfdim} + u_{(\pesubmfdimplusone)}r^{\pesubmfdimplusone}
 + O(r^{\pesubmfdimplustwo}\log r),  
\end{equation}
where the $u_{(j)}$ are functions of $x$.
The log terms in~\eqref{eqn:U-expansion} and~\eqref{eqn:u-expansion} come only from the log terms in $g_r$:
If \penologterms, then $g_r$ is smooth, and hence the expansion of $u(x,r)$ has no log terms.  
If \pelogterms, then the $r^{\pedimminusone}\log r$ term in the expansion of $g_r$ 
can generate a $r^{\pedimplusone}\log r$ term in the expansions of $U_r$ and
$u(x,r)$.  Thus the $O(r^{\pesubmfdimplustwo}\log r)$ terms are actually $O(r^{\pesubmfdimplustwo})$
unless \pelogterms\ and \pecodimone.

Set $h_+:=i^*g_+$.  The metric $\oh:=r^2h_+$ can be written in the
coordinates $(x,r)$:
\begin{equation}\label{eqn:hbar1}
  \oh = \oh_{00}dr^2 + 2\oh_{\alpha 0}dr dx^\alpha + \oh_{\alpha\beta}dx^\alpha dx^\beta,
\end{equation}
where
\begin{equation}\label{eqn:hbar}
\begin{aligned}
 \oh_{\alpha\beta} & = g_{\alpha\beta} +
 2g_{\alpha^\prime(\alpha}u^{\alpha^\prime}{}_{,\beta)} +
 g_{\alpha^\prime\beta^\prime}u^{\alpha^\prime}{}_{,\alpha}u^{\beta^\prime}{}_{,\beta}
 , \\ 
 \oh_{\alpha 0} & = g_{\alpha\alpha^\prime}u^{\alpha^\prime}{}_{,r} +
 g_{\alpha^\prime\beta^\prime}u^{\alpha^\prime}{}_{,\alpha}u^{\beta^\prime}{}_{,r}
 , \\ 
 \oh_{00} & = 1 + g_{\alpha^\prime\beta^\prime}u^{\alpha^\prime}{}_{,r}u^{\beta^\prime}{}_{,r} .
\end{aligned}
\end{equation}
We have used a `0' index for the $r$-direction, commas denote partial derivatives with respect to the coordinates
$(x^\alpha,r)$ on $Y$, and round parentheses denote symmetrization as in~\eqref{eqn:symmetrize}.  The components of $\oh$ and the derivatives of $u$
are evaluated at $(x,r)$.  The above formulas for the components of $\oh$
were obtained from the pullback of $r^2g_+$ upon writing
\[
g_r=g_{\alpha\beta}(x,u,r)dx^\alpha dx^\beta
+2g_{\alpha\alpha'}(x,u,r)dx^\alpha du^{\alpha'}
+g_{\alpha'\beta'}(x,u,r)du^{\alpha'} du^{\beta'} , 
\]
and all $g_{ij}$ in \eqref{eqn:hbar} are understood to be evaluated at
$(x,u(x,r),r)$.

The expansions of $\oh$ are obtained by substituting  
\eqref{eqn:pe-expansion} and \eqref{eqn:u-expansion} into
\eqref{eqn:hbar}.
Observe first that $\oh_{\alpha 0}=0$ and $\oh_{00}=1$ at $r=0$.  Thus
$|dr|^2_{\oh}=1$ on $\partial Y$, so $h_+$ is asymptotically hyperbolic.
One also finds that 
the first odd term in the expansions of $\oh_{\alpha\beta}$ and $\oh_{00}$
occurs at order $\pesubmfdimplusone$, and the first even term in the expansion of  
$\oh_{\alpha 0}$ occurs at order $\pesubmfdim+2$.  (Once again there can be log terms
if \pelogterms:  in this case the expansion of $\oh_{\alpha\beta}$ can
have a $r^{\pedimminusone}\log r $ term, the expansion of $\oh_{\alpha 0}$ can have a 
$r^{\pedim}\log r$ term, and the expansion of $\oh_{00}$ can have a
$r^{\pedimplusone}\log r$ term.)

Set $h_{(0)}:=i^*g_{(0)}$.  Let $\rho$ be the geodesic defining function for 
$h_+$ determined by $h_{(0)}$ and let $(y,\rho)$, $y\in \partial Y$, be the
associated boundary identification for $Y$.  For this discussion we
identify $Y$ with $i(Y)$ near $\partial Y$.  The map $(x,r)\to (y,\rho)$
relating the two boundary identifications is uniquely determined by the
requirement that
\begin{equation}\label{eqn:hplus}
\rho^2 h_+ = d\rho^2 +h_\rho
\end{equation}
relative to the $(y,\rho)$ boundary identification.  The defining function
$\rho$ is determined by the eikonal equation $|d\rho|^2_{\rho^2h_+}=1$ and
then $y$ is extended to $Y$ by following the gradient flow of 
$\operatorname{grad}_{\rho^2 h_+}(\rho)$ (see 
\cite{GrahamLee1991}*{Lemma~5.2 and the subsequent paragraph}).  Analysis
of the eikonal and gradient flow equations as in
\cite{Guillarmou2005}*{Lemma~2.1} shows that $y=y(x,r)$,
$\rho = \rho(x,r)$, where the expansions of $y$ and $\rho$ have the form    
\begin{equation}\label{eqn:yrhoexpand}
\begin{aligned}
  y(x,r) & =x + y_{(2)}r^2+\dotsm +y_{(\pesubmfdimminustwo)}r^{\pesubmfdimminustwo} + O(r^{\pesubmfdim}\log r) , \\ 
  \rho(x,r) & =r\big(1 + \rho_{(3)}r^2+\dotsm +\rho_{(\pesubmfdimminusone)}r^{\pesubmfdimminustwo} +
  O(r^{\pesubmfdim}\log r)\big).  
\end{aligned}
\end{equation}
Here $y \in \partial Y$ is described in terms of its $x$-coordinate, and 
the coefficients $y_{(2j)}$ and $\rho_{(2j+1)}$ are functions of $x$.

The expansions of the inverse map have the same form:
\begin{equation}\label{eqn:xrexpand}
\begin{aligned}
  x(y,\rho) & =y + x_{(2)}\rho^2+\dotsm +x_{(\pesubmfdimminustwo)}\rho^{\pesubmfdimminustwo} +
  O(\rho^{\pesubmfdim}\log \rho) , \\
  r(y,\rho) & =\rho\big(1 + r_{(3)}\rho^2+\dotsm +r_{(\pesubmfdimminusone)}\rho^{\pesubmfdimminustwo} +
  O(\rho^{\pesubmfdim}\log \rho)\big), 
\end{aligned}
\end{equation}
where the coefficients $x_{(2j)}$, $r_{(2j+1)}$ are functions of $y$.  Now   
pull back $h_+=r^{-2}\oh$ by the transformation \eqref{eqn:xrexpand} and use the
parity of the components of $\oh$ discussed above to deduce that
$h_\rho$ in Equation~\eqref{eqn:hplus} has the expansion
\begin{equation}\label{eqn:poincare-expansion}
  h_\rho  = h_{(0)} + h_{(2)}\rho^2 + \dotsm + h_{(\pesubmfdimminustwo)}\rho^{\pesubmfdimminustwo} +
  O(\rho^{\pesubmfdim}\log \rho).
\end{equation}
There are no log terms in any of the expansions \eqref{eqn:yrhoexpand},
\eqref{eqn:xrexpand}, \eqref{eqn:poincare-expansion} if \penologterms.
If \pelogterms, then the first log term in these expansions
occurs at order~$\pedimminusone$.  Thus the remainder terms are actually $O(r^{\pesubmfdim})$ in
\eqref{eqn:yrhoexpand} or $O(\rho^{\pesubmfdim})$ in \eqref{eqn:xrexpand},
\eqref{eqn:poincare-expansion} unless \pelogterms\ and \pecodimone.      

The renormalized area $\mA$ of $Y$ was defined in \cite{GrahamWitten1999}
as the constant term in the asymptotic expansion of  
$\operatorname{Area}_Y\{r>\epsilon\}$ as $\epsilon \to 0$, and it was 
shown that $\mA$ is independent of the choice of geodesic defining function
$r$ for $g_+$ on $X$.  The same argument shows that $\mA$ also equals 
the constant term in the asymptotic expansion of  
$\operatorname{Area}_Y\{\rho>\epsilon\}$, where as above $\rho$ is a
geodesic defining function for $h_+$ on $Y$.  The argument only
uses the parity of the terms in the expansions \eqref{eqn:yrhoexpand}, 
\eqref{eqn:xrexpand} of $r$ and $\rho$ in terms of one another and the
parity of the terms in \eqref{eqn:poincare-expansion}.
Since $\rho$ is an intrinsic geodesic defining function on $Y$, the 
constant term in the asymptotic expansion of
$\operatorname{Area}_Y\{\rho>\epsilon\}$ can also be interpreted as a 
renormalized volume of the asymptotically hyperbolic manifold $(Y,h_+)$.

In \cite{FeffermanGraham2002}*{Theorems 4.1 and 4.3}, the scattering theory
of \cite{GrahamZworski2003} was applied to show that the renormalized
volume of an asymptotically hyperbolic approximately Einstein manifold can be
calculated  
as an integral over the boundary of a function that appears in the 
asymptotic expansion of a solution of a particular linear scalar equation.
The same argument applies to any asymptotically hyperbolic metric that is
even to the appropriate order.  In particular, it applies to calculate the 
renormalized area of a minimal submanifold:  

\begin{proposition}
  \label{fg}
  Let $i \colon Y^{\pesubmfdim} \to (X^{\pedim},g_+)$ be a polyhomogeneous minimal
  immersion of an even-dimensional manifold into a Poincar\'e--Einstein
  manifold and let $h_{(0)}$ be a representative of the conformal infinity
  of $(Y,h_+:=i^\ast g_+)$. 
  Then there is a unique $v\in C^\infty(\mathring{Y})$ such that
  \begin{equation}
   \label{eqn:scattering}
   \begin{cases}
    -\Delta_{h_+}v = \pesubmfdimminusone , & \text{in $\Int{Y}$}, \\
    v = \log \rho + o(1) , & \text{near $\partial Y$} ,
   \end{cases}
  \end{equation}
  where $\rho$ is the geodesic defining function determined by $h_{(0)}$.
  Moreover, near $\partial Y$,
  \begin{equation}
   \label{eqn:fg-potential-expression}
   v = \log \rho + F ,
  \end{equation}
where $F$ is polyhomogeneous with expansion
  \begin{equation}
   \label{eqn:expand F}
   F= F_{(2)}\rho^2 +\dotsm + F_{(\pesubmfdimminustwo)}\rho^{\pesubmfdimminustwo} + B\rho^{\pesubmfdimminusone} +O(\rho^{\pesubmfdim}\log\rho).
  \end{equation}
The renormalized area of $Y$ is given by:
  \begin{equation}\label{eqn:Bint}
   \mA = \int_{\partial Y} \, B\rv_{\partial Y} \, d\sigma ,
  \end{equation}
  where $d\sigma$ is the volume form of $h_{(0)}$.
 \end{proposition}

\noindent
Some remarks are in order.  The existence of $v$ uses scattering
theory for the metric $h_+$.  If \pelogterms, then $\rho^2h_+$ is
polyhomogeneous but not necessarily smooth.  The extension of the scattering
theory to the case 
of polyhomogeneous metrics is addressed in~\cite{ChangMcKeownYang2022}.
Also, the same comments as above apply to the remainder term in
\eqref{eqn:expand F}.  Namely, $F$ is smooth if \penologterms, and 
if \pelogterms, then the first log term occurs at order $\pedimminusone$, so that the
remainder term is $O(\rho^{\pesubmfdim})$ unless \pecodimone.  Taking this into
consideration, when \pesubmfdimtwo, Equation~\eqref{eqn:expand F} should be interpreted as  
$F=B\rho +O(\rho^2)$.     

We call $v$ the \defn{scattering potential} determined by $h_{(0)}$.  
Equations~\eqref{eqn:fg-potential-expression} and \eqref{eqn:expand F}
imply that $e^v$ is a (polyhomogeneous) defining function for
$\partial Y$.  Therefore $\hh:=e^{2v}h_+$ 
is a compactification of $h_+$, which we call the
\defn{scattering compactification}.  The scattering compactification of a
Poincar\'e--Einstein metric was introduced in \cite{ChangQingYang2006}, 
where it was observed that it has vanishing Branson's $Q$-curvature.  The
attempt to find an analogy in the setting of minimal submanifolds of 
Poincar\'e--Einstein manifolds led to the introduction of the extrinsic
$Q$-curvature in \cite{CaseGrahamKuo2023}, and the corresponding vanishing
statement is likewise essential for the proof of \cref{main-theorem}:  

\begin{proposition}
  \label{cqy}
  Let $i \colon Y^{\pesubmfdim} \to (X^{\pedim},g_+)$ be a polyhomogeneous minimal
  immersion of an even-dimensional manifold into a Poincar\'e--Einstein
  manifold and let $h_{(0)}$ be a representative of the conformal infinity
  of   $(Y,h_+:=i^\ast g_+)$. 
Let $v$ be the associated scattering potential and 
set $\hh := e^{2v}h_+$.  Then $Q^{\hh} = 0$.
\end{proposition}
\begin{proof}
Applying Equation~\eqref{eqn:conformal-transformation-law},
\cref{factorization} and Equation~\eqref{eqn:scattering} to $h_+$ and $\hh
= e^{2v}h_+$ yields 
\begin{align*}
  e^{\pesubmfdimparen v}Q^{\hh} & =
Q^{h_+}+P^{h_+}_{\pesubmfdim}v\\
  &=(-1)^{\frac{\pesubmfdim}{2}}\pesubmfdimminusoneparen! +
\left(\prod_{j=1}^{(\pesubmfdimminustwo)/2} \biggl(-\Delta_{h_+} -
\big(\tfrac{\pesubmfdimminustwo}{2}+j\big)\big(\tfrac{\pesubmfdim}{2}-j\big)\biggr)\right)
(-\Delta_{h_+}v) \\  
   & = 0 . \qedhere
\end{align*}
\end{proof}

\section{Proofs of Theorems~\ref{main-theorem} and \ref{closed case}} 
\label{sec:main-thm}

We turn first to the proof of \cref{main-theorem}.  
Our first objective is to compute the coefficient $c$ of the 
Pfaffian in the decomposition~\eqref{eqn:submanifold-decomposition} for
$I=Q$. Our normalization of the Pfaffian is such that the
Chern--Gauss--Bonnet formula reads 
\begin{equation}\label{eqn:genericcgb}
(2\pi)^{\submfdim/2} \chi(Y) = \int_Y \Pf \dvol 
\end{equation}
for compact Riemannian manifolds $(Y^{\submfdim},h)$ of even dimension $\pesubmfdim$.   

\begin{lemma}
  \label{q-formula-modulo-alexakis}
Let $\submfdim,\mfdim\in \bN$ with $\submfdim$ even and $\mfdim>\submfdim$.  Suppose that
 there is a constant $c_{\mfdim,\submfdim} \in \bR$, a scalar conformal submanifold
 invariant $\mW_Q$,  
 and a natural submanifold vector field $V$ such that 
\begin{equation}\label{eqn:alexakis-assumption}
Q = c_ {\mfdim,\submfdim}\,\oPf + \mW_Q +\odiv V
\end{equation}
for all Riemannian manifolds $(X^{\mfdim},g)$ and embedded submanifolds
$Y^{\submfdim}\subset X^{\mfdim}$.  Then 
\begin{equation*}
  c_{\mfdim, \submfdim} = \frac{(\submfdim-1)!}{(\submfdim-1)!!}.
\end{equation*}
\end{lemma}

\begin{proof}
 Consider an equatorial sphere $i \colon S^{\submfdim} \to (S^{\mfdim},g)$ in the
 round $\mfdim$-sphere.  
 By stereographic projection, the equatorial sphere is locally equivalent
 to the embedding $\bR^{\submfdim} \hookrightarrow (\bR^{\mfdim}, \lv dx \rv^2)$ of
 $\bR^{\submfdim}$ as an affine subspace in flat Euclidean $\mfdim$-space. 
 The latter is a totally geodesic submanifold of a flat manifold, and
 hence, by naturality, $\mW_Q^{\lv dx \rv^2} = 0$.
 By conformal invariance, $\mW_Q^{g} = 0$.
 \Cref{factorization} implies that $Q^{i^\ast g}=(\submfdim-1)!$.
 We conclude by integration that
 \begin{equation*}
  2(2\pi)^{\submfdim/2}c_{\mfdim, \submfdim} = c_{\mfdim, \submfdim} \int_{S^{\submfdim}} \oPf \darea =
  (\submfdim-1)!\Vol(S^{\submfdim}) = (\submfdim-1)! \times
  \frac{2(2\pi)^{\submfdim/2}}{(\submfdim-1)!!}. 
 \end{equation*}
 Therefore $c_{\mfdim, \submfdim} = \frac{(\submfdim-1)!}{(\submfdim-1)!!}$.
\end{proof}

We introduce two simplifications to the argument of Chang, Qing, and Yang
to help deal with the more complicated submanifold setting.  
The first is that we
make a second conformal change so as to apply the conjectured
decomposition~\eqref{eqn:submanifold-decomposition} to the geodesic compactification instead of the
scattering compactification.  This gives a more direct route to
the term producing the renormalized area in  
the Chern--Gauss--Bonnet formula.  The second is that we provide a simpler
version of the parity argument for the vanishing of the other terms.

\begin{proof}[Proof of \cref{main-theorem}]
Set $h_+=i^*g_+$.  Pick a metric $g_{(0)}\in \kc$ and set
$h_{(0)}=i^*g_{(0)}$.  Let $\rho$ be the geodesic defining 
function for $h_+$ determined by $h_{(0)}$.  Regard $\rho$ as defined on
$i(Y)$ near $i(\partial Y)$.  Extend $\rho$ to a
neighborhood of $i(Y)$ in $X$ near $\partial X$ by requiring it to be
independent of the $u$ variables in coordinates $(x,u,r)$ as described
in \cref{sec:q-bg}.  Now choose some positive 
smooth extension of $\rho$ to a neighborhood of $i(Y)$ in all of $X$ and
set $g:=\rho^2g_+$ and $h:=i^*g$.  Then $h= d\rho^2 +h_\rho$ near $\partial Y$
in the boundary identification determined by $h_{(0)}$.

Let $v$ be the scattering potential determined by 
$h_{(0)}$ as in \cref{fg} and let $\hh=e^{2v}h_+$ be the associated
scattering compactification.  Extend $F$ so that Equation~\eqref{eqn:fg-potential-expression} holds on all of $Y$;
then $e^v = \rho e^F$, and so $\hh = e^{2F}h$ on all of $Y$.  Extend $F$ to a neighborhood of $i(Y)$
in $X$ near $\partial X$ by requiring it to be independent of the $u$
variables.    

\Cref{cqy} followed by
Equation~\eqref{eqn:conformal-transformation-law} gives 
\[
0=e^{\pesubmfdimparen F}Q^{\hh}= Q^h+P_{\pesubmfdim}^h (F).
\]
Now write the assumed decomposition \eqref{eqn:submanifold-decomposition}
for $Q^h$ and use 
\cref{q-formula-modulo-alexakis} to obtain
\begin{equation}\label{eqn:Pf}
\frac{\pesubmfdimminusoneparen!}{\pesubmfdimminusoneparen!!}\oPf^h + \mW_{Q}^g + \odiv^h V^g +P_{\pesubmfdim}^h (F) = 0.
\end{equation}
On a manifold with boundary, there is an additional boundary integral in
the Chern--Gauss--Bonnet formula~\eqref{eqn:genericcgb}, but it vanishes 
if the boundary is totally geodesic.  The expansion~\eqref{eqn:poincare-expansion} of $h_\rho$
implies that $\partial Y$ is
totally geodesic for $h$.  So integrating Equation~\eqref{eqn:Pf} gives
\begin{equation}\label{eqn:almostthere}
  \frac{\pesubmfdimminusoneparen!}{\pesubmfdimminusoneparen!!}(2\pi)^{\pesubmfdimparen/2}\chi(Y) + \int_Y \mW_Q\darea
  +\int_{\partial Y} \langle \onf,V^g \rangle\, d\sigma
  +\int_YP_{\pesubmfdim}^h(F)\darea = 0,
\end{equation}  
where $\onf = -\partial_\rho$ is the $h$-outward pointing unit normal along $\partial Y$.
The proof will be completed by showing that $\langle \onf,V^g \rangle=0$ and
$\int_YP_{\pesubmfdim}^h(F)\darea = -(-1)^{\pesubmfdimparen/2} \pesubmfdimminusoneparen!\,\mA$.

By definition, $V=V^g$ is a linear combination of partial contractions with
one free raised tangential index of tensors of the form
\eqref{eqn:partial contraction}.  
The condition that $\odiv V$ has weight $\minuspesubmfdim$ is equivalent to  
\[
\sum_{i=1}^p (M_i+2) +\sum_{j=1}^q (N_j+1) = \pesubmfdimminusone . 
\]
It follows that each $M_i\leq \pesubmfdimminusthree$ and each $N_j\leq \pesubmfdimminustwo$.  Since
$\nabla^M\Rm$ and $\onabla^NL$ depend on at most $M+2$ and $N+1$
derivatives of $g$, respectively, we see that $V$
depends on at most $\pesubmfdimminusone$ derivatives of $g$.  Also, since
$\onabla^NL$ depends on at most $N+2$ derivatives of (defining functions 
for) $Y$, we see that $V$ depends on at most $\pesubmfdim$ derivatives of $Y$. 

Consider the expansion at $r=0$ of $g=(\rho/r)^2(dr^2+g_r)$ in the
$(x,u,r)$ coordinates.  Using the expansion~\eqref{eqn:pe-expansion} for $g_r$, 
the expansion~\eqref{eqn:yrhoexpand} for $\rho/r$, and the 
fact that $\rho$ was extended to be independent of $u$, it follows that the
expansion of $g$ is even through order $\pesubmfdimminustwo$ and has no $r^{\pesubmfdimminusone}$ term.
An evenness argument implies $\langle \onf,V^g\rangle=0$:  
Define $g^0$ by truncating the expansion of $g$ at order $\pesubmfdimminusone $.  Likewise,
define $U^0_r$ by 
truncating the expansion \eqref{eqn:U-expansion} of $U_r$ at order $\pesubmfdim$ 
(keeping the term of order $\pesubmfdim$).  Define $Y^0$ by \eqref{eqn:Y} with
$U_r$ replaced by   
$U^0_r$ and define $V^0$ to be the natural vector field determined by 
$g^0$ and $Y^0$.  Then $V=V^0$ at $r=0$.  Since they are
polynomial and even in $r$, both $g^0$ and $U^0$ extend to
$r\in (-\epsilon,\epsilon)$ and are invariant under the reflection
$\mR(p,r):=(p,-r)$ on $\partial X\times (-\epsilon,\epsilon)_r$.      
Extend $Y^0$ to $r<0$ by \eqref{eqn:Y} with $Y$ replaced by $Y^0$ and $U_r$ 
replaced by $U_r^0$.  Naturality implies that $\mR^*V^0 = V^0$.  Hence  
$\mR^*V = V$ at $r=0$.  But $\mR^*\onf=-\onf$, and hence
$\mR^*\langle \onf,V\rangle = -\langle \onf,V\rangle$ at $r=0$.  We
conclude that $\langle \onf,V\rangle=0$.  

Now $F$ has the expansion \eqref{eqn:expand F} in the boundary
identification $(y,\rho)$ induced by $h_{(0)}$.  This is related to the
$(x,r)$ boundary identification on $Y$ by \eqref{eqn:yrhoexpand}.  Since
$y(x,r)$ and $\rho/r$ have even expansions to order $\pesubmfdimminustwo$ with no $r^{\pesubmfdimminusone}$
term and $\rho/r = 1$ on $\partial Y$, it follows that the expansion of $F$
in the $(x,r)$ boundary identification has the same form: 
\[
F= \widetilde{F}_{(2)}r^2 +\dotsm + \widetilde{F}_{(\pesubmfdimminustwo)}r^{\pesubmfdimminustwo} +
Br^{\pesubmfdimminusone} +O(r^{\pesubmfdim}\log r),
\]
where the $\widetilde{F}_{(2j)}$ are functions of $x$ and $B$ is the same
as in \eqref{eqn:expand F}.  This expansion also holds in a neighborhood of   
$i(Y)$ in $X$ near $\partial X$ since $F$ was extended to be independent of
the $u$ variables.

Recall that $P_{\pesubmfdim}$ is a natural formally self-adjoint operator with
leading term $(-\oDelta)^{\pesubmfdimparen/2}$ that annihilates constants.
Therefore there is a natural  $TY$-valued differential operator
$T \colon C^\infty(Y)\to C^\infty(Y;TY)$ of order at most $\pesubmfdimminusthree$ such that
\[
P_{\pesubmfdim}= (-\oDelta)^{\pesubmfdimparen/2}+ \odiv \mathop{\circ} T.
\]
Therefore
\[
\int_Y P_{\pesubmfdim}^h(F)\darea =
\int_{\partial Y} \left(
-\langle\onf,\operatorname{grad}_h(-\Delta_h)^{\pesubmfdimminustwoparen/2}(F^0 + Br^{\pesubmfdimminusone})\rangle
+ \langle\onf, T(F^0)\rangle \right) \, d\sigma , 
\]
where
$F^0 := \widetilde{F}_{(2)}r^2 + \dotsm + \widetilde{F}_{(\pesubmfdimminustwo)}r^{\pesubmfdimminustwo}$.  
On the one hand, the evenness argument above
shows that
\[
\langle\onf,\operatorname{grad}_h(\Delta_h)^{(\pesubmfdimminustwo)/2}(F^0)\rangle =0,\qquad
 \langle\onf, T(F^0)\rangle =0.
\]
On the other hand, there is a differential operator $T'$
of transverse order strictly less than $\pesubmfdimminusone$ so that 
\begin{align*}
 \langle\onf,\operatorname{grad}_h(-\Delta_h)^{(\pesubmfdimminustwo)/2}(Br^{\pesubmfdimminusone})\rangle   
 & = \big((-1)^{\pesubmfdimparen/2}\partial_r^{\pesubmfdimminusone} + T'\big)(Br^{\pesubmfdimminusone})\big|_{r=0} \\
 & = (-1)^{\pesubmfdimparen/2}\pesubmfdimminusoneparen!\,B.
\end{align*}
Integrating using Equation~\eqref{eqn:Bint} shows that
$\int_YP_{\pesubmfdim}^h(F)\darea = -(-1)^{\pesubmfdimparen/2}\pesubmfdimminusoneparen!\,\mA$, as claimed.  
\end{proof}

The proof of \cref{closed case} is much simpler:
\begin{proof}[Proof of \cref{closed case}]
Proposition~\ref{factorization} shows that
$Q^{i^*g} = \lambda^{\pesubmfdimparen/2} \pesubmfdimminusoneparen!$.  Recalling 
Lemma~\ref{q-formula-modulo-alexakis}, the assumed 
decomposition~\eqref{eqn:submanifold-decomposition} reads 
\[
\lambda^{\pesubmfdimparen/2} = \frac{1}{\pesubmfdimminusoneparen!!} \oPf + \frac{1}{\pesubmfdimminusoneparen!}\mW_Q + \frac{1}{\pesubmfdimminusoneparen!}\odiv V.
\]
The result follows upon integrating over $Y$. 
\end{proof}

\section{Riemannian and conformal submanifold geometry}
\label{sec:bg} 

\subsection{Conventions from Riemannian geometry}

Let $(X^{\mfdim},g)$ be a Riemannian manifold.
We always assume that $\mfdim \geq 3$.
The \defn{Riemann curvature tensor} is determined by
\begin{equation}
 \label{eqn:riemann}
 \nabla_a\nabla_b\tau_c - \nabla_b\nabla_a\tau_c = R_{abc}{}^d\tau_d
\end{equation}
for all one-forms $\tau_a$, where $\nabla_a$ is the Levi-Civita connection,
we raise and lower indices using $g$, and we employ abstract index
notation~\cite{Penrose1984}.
The \defn{Ricci tensor} and \defn{scalar curvature} of $g$ are the
contractions $R_{ab} := R_{acb}{}^c$ and $R := R_a{}^a$, respectively, of
$R_{abcd}$.
The \defn{Weyl tensor} is
\begin{equation}
 \label{eqn:defn-weyl}
 W_{abcd} := R_{abcd} - \Sch_{ac}g_{bd} - \Sch_{bd}g_{ac} + \Sch_{ad}g_{bc} + \Sch_{bc}g_{ad} ,
\end{equation}
where
\begin{equation*}
 \Sch_{ab} := \frac{1}{\mfdim-2}\left( R_{ab} - \trSch g_{ab} \right)
\end{equation*}
is the \defn{Schouten tensor} and $\trSch := \frac{1}{2(\mfdim-1)}R$.
Note that $\trSch = \Sch_a{}^a$.
Recall that the Weyl tensor is trace-free;
i.e.\ $W_{acb}{}^c=0$.
As discussed further below, the interest in $W_{abcd}$ stems from its conformal invariance.
The \defn{Cotton tensor} and the \defn{Bach tensor} are
\begin{align*}
 C_{abc} & := \nabla_a \Sch_{bc} - \nabla_b \Sch_{ac} , \\
 B_{ab} & := \nabla^c C_{cab} + W_{acbd}\Sch^{cd} ,
\end{align*}
respectively.
Note that our convention for the Cotton tensor differs from that used in~\cite{CaseGrahamKuo2023}.
Clearly $C_{abc} = -C_{bac}$.
We use square brackets (resp.\ round parentheses) to denote skew symmetrization (resp.\ symmetrization) of the enclosed indices;
e.g.
\begin{equation}
 \label{eqn:symmetrize}
 \begin{split}
  T_{[abc]} & := \frac{1}{3!}\left( T_{abc} - T_{bac} - T_{cba} - T_{acb} + T_{bca} + T_{cab} \right) , \\
  T_{(abc)} & := \frac{1}{3!}\left( T_{abc} + T_{bac} + T_{cba} + T_{acb} + T_{bca} + T_{cab} \right) .
 \end{split}
\end{equation}
The Bianchi identities imply that
\begin{align}
 \label{eqn:weyl-bianchi} \nabla_{[a}W_{bc]}{}^{de} & = -2C_{[ab}{}^{[d}g_{c]}{}^{e]} , \\
 \label{eqn:W-to-C} \nabla^e W_{abec} & = (\mfdim-3)C_{abc} ,
\end{align}
and also that
\begin{align*}
 C_{ba}{}^b & = 0 , & C_{[abc]} & = 0 , \\
 B_a{}^a & = 0 , & B_{[ab]} & = 0 .
\end{align*}
We refer to Equation~\eqref{eqn:weyl-bianchi} as the Weyl--Bianchi identity.

\subsection{Riemannian submanifolds}

Let $i \colon Y^{\submfdim} \to (X^{\mfdim},g)$ be an immersion into a Riemannian manifold.
We always assume that $1 \leq \submfdim < \mfdim$ and $\mfdim \geq 3$.
As above, we use abstract indices, with a lowercase Latin letter ($a,b,c,\dotsc$) labeling a section of $i^\ast TX$ or its dual.
Recall the $g$-orthogonal splitting $i^\ast TX = TY \oplus NY$.
We use a lowercase Greek letter ($\alpha,\beta,\gamma,\dotsc$) to label a section of $TY$ or its dual, and we use a primed
lowercase Greek letter ($\alpha^\prime,\beta^\prime,\gamma^\prime,\dotsc$) to label a section of $NY$ or its dual.
We also use Greek or primed Greek indices implicitly to denote composition with the projections to $TY$ or $NY$, respectively, or their duals.
With these conventions, $g_{ab}$ denotes the metric on $X$, while $g_{\alpha\beta}$ and $g_{\alpha^\prime\beta^\prime}$ denote the induced metrics on $TY$ and $NY$, respectively.
Note that $g_{\alpha\alpha^\prime}=0$.

Recall from \cref{sec:q-bg} that the Levi-Civita connection of $g_{ab}$ determines connections $\onabla_\alpha$ on $TY$ and $NY$ by projection.
Moreover, the connection $\onabla_\alpha$ on $TY$ is the Levi-Civita connection of $g_{\alpha\beta}$.

Our convention from \cref{sec:q-bg} for the \defn{second fundamental form} $L_{\alpha\beta\alpha^\prime}$ of $Y$ is that if $\tau_a$ is a section of $i^\ast T^\ast X$ with projections $\tau_\alpha$ and $\tau_{\alpha^\prime}$, then
\begin{equation}
 \label{eqn:second-fundamental-form-dual}
 \nabla_\alpha\tau_{\alpha^\prime} = \onabla_\alpha\tau_{\alpha^\prime} + L_{\alpha\beta\alpha^\prime}\tau^\beta .
\end{equation}
Recall that $L_{\alpha\beta\alpha^\prime}=L_{\beta\alpha\alpha^\prime}$.
Since the splitting $i^\ast TX = TY \oplus NY$ is $g$-orthogonal, we deduce that
\begin{equation}
 \label{eqn:second-fundamental-form}
 \nabla_\alpha\tau_\beta = \onabla_\alpha\tau_\beta - L_{\alpha\beta\alpha^\prime}\tau^{\alpha^\prime} .
\end{equation}
These identities extend to higher rank tensors by the Leibniz rule;
e.g.\
\begin{equation*}
 \nabla_\beta\tau_{\alpha\alpha^\prime} = \onabla_\beta\tau_{\alpha\alpha^\prime} - L_{\alpha\beta}{}^{\beta^\prime}\tau_{\beta^\prime\alpha^\prime} + L_\beta{}^\gamma{}_{\alpha^\prime}\tau_{\alpha\gamma} .
\end{equation*}

The \defn{mean curvature} of $Y$ is the vector field
\begin{equation*}
 H^{\alpha^\prime} := \frac{1}{\submfdim}L_\alpha{}^{\alpha\alpha^\prime} .
\end{equation*}
We denote by
\begin{equation*}
 \tfss_{\alpha\beta\alpha^\prime} := L_{\alpha\beta\alpha^\prime} - H_{\alpha^\prime} g_{\alpha\beta}
\end{equation*}
the trace-free part of the second fundamental form.
As discussed further below, the interest in $\tfss_{\alpha\beta\alpha^\prime}$ stems from its conformal invariance.

When a function $u \in C^\infty(X)$ is given, we denote $u_{ab\dotsm c} := \nabla_c \dotsm \nabla_b\nabla_au$.
Thus, by the conventions described above,
\begin{equation*}
 u_{\alpha\beta} = \nabla_\beta\nabla_\alpha u = \onabla_\beta \onabla_\alpha u - L_{\alpha\beta\alpha^\prime}\nabla^{\alpha^\prime}u = \onabla_\beta\onabla_\alpha u - L_{\alpha\beta\alpha^\prime}u^{\alpha^\prime} .
\end{equation*}

The Gauss--Codazzi--Ricci equations (see, e.g., \cite{DajczerTogeiro2019}*{Section~1.3}) and their various contracted forms are obtained by combining
Equations~\eqref{eqn:riemann}, \eqref{eqn:second-fundamental-form-dual},
and~\eqref{eqn:second-fundamental-form}: 
\begin{align*}
 R_{\alpha\beta\gamma\delta} & = \oR_{\alpha\beta\gamma\delta} -
 L_{\alpha\gamma\alpha^\prime}L_{\beta\delta}{}^{\alpha^\prime} +
 L_{\alpha\delta\alpha^\prime}L_{\beta\gamma}{}^{\alpha^\prime} , \\ 
 R_{\alpha\beta} & = \oR_{\alpha\beta} +
 R_{\alpha\alpha^\prime\beta}{}^{\alpha^\prime} +
 L_{\alpha\gamma\alpha^\prime}L_\beta{}^{\gamma\alpha^\prime} - \submfdim
 H^{\alpha^\prime} L_{\alpha\beta\alpha^\prime} , \\ 
 R & = \oR + 2R_{\alpha^\prime}{}^{\alpha^\prime} +
 L_{\alpha\beta\alpha^\prime}L^{\alpha\beta\alpha^\prime} -
 \submfdim^2H_{\alpha^\prime} H^{\alpha^\prime} -
 R_{\alpha^\prime\beta^\prime}{}^{\alpha^\prime\beta^\prime} , \\ 
 R_{\alpha\beta\alpha^\prime\gamma} & = 2\onabla_{[\alpha}L_{\beta]\gamma\alpha^\prime} , \\
 R_{\alpha\alpha^\prime} & = R_{\alpha\beta^\prime\alpha^\prime}{}^{\beta^\prime} - \onabla^\beta L_{\beta\alpha\alpha^\prime} + \submfdim\onabla_\alpha H_{\alpha^\prime} , \\
 R_{\alpha\beta\alpha^\prime\beta^\prime} & = \oR_{\alpha\beta\alpha^\prime\beta^\prime} - L^\gamma{}_{\alpha\alpha^\prime}L_{\gamma\beta\beta^\prime} + L^\gamma{}_{\alpha\beta^\prime}L_{\gamma\beta\alpha^\prime} ,
\end{align*}
where $\oR_{\alpha\beta\alpha^\prime\beta^\prime}$ is the curvature of the connection $\onabla_\alpha$ on $NY$.
It is useful to rewrite these in terms of the Schouten tensor, the Weyl tensor, and the second fundamental form.
To that end, set
\begin{equation}
 \label{eqn:defn-Di}
 \Di_{\alpha\alpha^\prime} := \Sch_{\alpha\alpha^\prime} - \onabla_\alpha H_{\alpha^\prime} .
\end{equation}
If $\submfdim \geq 3$, then the Gauss--Codazzi--Ricci equations are equivalent to (cf.\ \cite{Fialkow1944})
\begin{subequations}
 \label{eqn:gauss-codazzi}
 \begin{align}
  \label{eqn:gcW} W_{\alpha\beta\gamma\delta} & = \oW_{\alpha\beta\gamma\delta} - \tfss_{\alpha\gamma\alpha^\prime}\tfss_{\beta\delta}{}^{\alpha^\prime} + \tfss_{\alpha\delta\alpha^\prime}\tfss_{\beta\gamma}{}^{\alpha^\prime} - 2\Fi_{\alpha[\gamma}g_{\delta]\beta} + 2\Fi_{\beta[\gamma}g_{\delta]\alpha} , \\
  \label{eqn:gcP} \Sch_{\alpha\beta} & = \oSch_{\alpha\beta} -
 H^{\alpha^\prime}\tfss_{\alpha\beta\alpha^\prime} - \frac{1}{2}
 H^{\alpha^\prime} H_{\alpha^\prime} g_{\alpha\beta} + \Fi_{\alpha\beta} ,
 \\ 
  \label{eqn:gcJ} \trSch & = \otrSch + \Sch_{\alpha^\prime}{}^{\alpha^\prime} - \frac{\submfdim}{2}H^{\alpha^\prime} H_{\alpha^\prime} + \trFi , \\
  \label{eqn:gcdL} W_{\alpha\beta\alpha^\prime\gamma} & =
 2\onabla_{[\alpha}\tfss_{\beta]\gamma\alpha^\prime} + 2g_{\gamma[\alpha}\Di_{\beta]\alpha^\prime} , \\
  \label{eqn:gcD} (\submfdim-1)\Di_{\alpha\alpha^\prime} & = -\onabla^\beta\tfss_{\beta\alpha\alpha^\prime} - W_{\alpha\beta\alpha^\prime}{}^\beta , \\
  \label{eqn:gcnc} W_{\alpha\beta\alpha^\prime\beta^\prime} & = \oR_{\alpha\beta\alpha^\prime\beta^\prime} - \tfss^\gamma{}_{\alpha\alpha^\prime}\tfss_{\gamma\beta\beta^\prime} + \tfss^\gamma{}_{\alpha\beta^\prime}\tfss_{\gamma\beta\alpha^\prime} ,
 \end{align}
\end{subequations}
where
\begin{equation}
 \label{eqn:defn-Fialkow}
 \Fi_{\alpha\beta} := \frac{1}{\submfdim-2}\left(
 \tfss_{\alpha\gamma\alpha^\prime}\tfss_\beta{}^{\gamma\alpha^\prime} -
 W_{\alpha\gamma\beta}{}^{\gamma} - \trFi g_{\alpha\beta} \right) 
\end{equation}
is the manifestly conformally invariant (of weight 0) \defn{Fialkow tensor}~\cite{Fialkow1944} and
\begin{equation*}
 \trFi := \Fi_\alpha{}^\alpha = \frac{1}{2(\submfdim-1)}\left(
 \tfss_{\alpha\beta\alpha^\prime}\tfss^{\alpha\beta\alpha^\prime} -
 W_{\alpha\beta}{}^{\alpha\beta} \right) 
\end{equation*}
is its trace.
Equation~\eqref{eqn:gcnc} follows from the Gauss--Codazzi--Ricci equations and the fact that $R_{\alpha\beta\alpha^\prime\beta^\prime}=W_{\alpha\beta\alpha^\prime\beta^\prime}$.
Notably, it recovers the fact~\cite{Chen1974} that the curvature of the normal bundle is conformally invariant.
Of course, equations involving the trace of the Fialkow tensor, but not the Fialkow tensor itself, require only $\submfdim \geq 2$;
and Equations~\eqref{eqn:gcdL}--\eqref{eqn:gcnc} hold for $\submfdim=1$, but are trivial.

\subsection{Conformal submanifolds}
\label{subsec:conformal-submanifold-invariants}

Let $L^g \colon C^\infty(Y;T_1) \to C^\infty(Y;T_2)$ be a metric-dependent differential operator
on sections of vector bundles $T_1$ and $T_2$ over $Y$, where $g$ is a metric on $X$.
We say 
that $L$ is \defn{homogeneous} if there is an $h \in \bR$
so that $L^{c^2g} = c^h L^g$ for all $c>0$. 
In this case we call $h$ the \defn{homogeneity} of $L$.
Note that if $L_1$ and $L_2$ have homogeneities $h_1$ and $h_2$,
respectively, then $L_1 \circ L_2$ has homogeneity $h_1+h_2$, provided the composition makes sense.
We say that $L$ is \defn{conformally covariant} if there are constants $a,b \in \bR$ such that
\begin{equation*}
 L^{e^{2u}g}(v) = e^{-bi^\ast u} L^g (e^{ai^\ast u}v)
\end{equation*}
for all $u \in C^\infty(X)$ and all $v \in C^\infty(Y;T_1)$.
Note that conformally covariant operators are necessarily homogeneous.

Branson showed~\cite{Branson1985}*{Section~1} that homogeneous differential operators on $X$ are conformally covariant if and only if their conformal linearizations are zero.
We develop the analogous framework for metric-dependent differential operators $L$ associated to an immersion.
Suppose that $L$ has homogeneity $h$.  Fix $w \in \bR$.
Given an immersion $i \colon Y^{\submfdim} \to (X^{\mfdim},g)$
and a function $\Upsilon \in C^\infty(X)$, define the \defn{conformal
  linearization of $L$}, regarded as an operator on densities of weight
$w$, by 
\begin{equation*}
 L^\bullet := \left.\frac{\partial}{\partial t}\right|_{t=0} e^{-t(w+h)i^\ast\Upsilon} \circ L^{e^{2t\Upsilon}g} \circ e^{twi^\ast\Upsilon} ,
\end{equation*}
where $e^{ci^\ast\Upsilon}$ acts as a multiplication operator and
$\left.\frac{\partial}{\partial t}\right|_{t=0}A^t$ is the evaluation at
$t=0$ of the derivative in $t$ of the one-parameter family of
operators $A^t$.
(We suppress the dependence of $L^\bullet$ on $g$, $\Upsilon$, and $w$ to simplify our notation.)
We define the conformal linearization of metric-dependent tensor fields by regarding them as zeroth-order metric-dependent differential operators; note that $L^\bullet$ is independent of $w$ in this case.
We use the same notation for the conformal linearizations of metric-dependent differential operators, including tensors, defined on $X$.

It is easy to see
(cf.\ \cite{Branson1985}*{Corollary~1.14}) that~$L^\bullet=0$ if and only
if
\begin{equation*}
 L^{e^{2\Upsilon}g} = e^{(w+h)i^\ast\Upsilon} \circ L^g \circ e^{-wi^\ast\Upsilon}
\end{equation*}
for all Riemannian metrics $g$ on $X$ and all $\Upsilon \in C^\infty(X)$.
It is straightforward to check that ${}^\bullet$ satisfies a Leibniz rule:
Suppose that $L_1 \colon C^\infty(Y;T_1) \to C^\infty(Y;T_2)$ and $L_2
\colon C^\infty(Y;T_2) \to C^\infty(Y;T_3)$ have homogeneities $h_1$ and $h_2$, respectively. 
Fix $w \in \bR$.
Then
\begin{equation*}
 (L_2 \circ L_1)^\bullet = L_2 \circ L_1^\bullet + L_2^\bullet \circ L_1 ,
\end{equation*}
where $L_1$ and $L_2 \circ L_1$ are regarded as operators on densities of
weight $w$, and $L_2$ is regarded as an operator on densities of weight $w
+ h_1$. 

Since the difference of two connections is a tensor, it makes sense to
consider the conformal linearization $\nabla_a^\bullet \colon
C^\infty(X;T^\ast X) \to C^\infty(X; (T^\ast X)^{\otimes 2})$ of the
Levi-Civita connection. 
Indeed, given $w \in \bR$, the Koszul formula implies that
\begin{equation*}
 \nabla_a^\bullet \tau_b = (w-1)\Upsilon_a \tau_b - \Upsilon_b\tau_a + \Upsilon^c\tau_c g_{ab} .
\end{equation*}
From this one recovers the formulas
\begin{equation}
 \label{eqn:ambient-conformal-change}
 \begin{split}
  W_{abcd}^\bullet & = 0 , \\
  \Sch_{ab}^\bullet & = -\Upsilon_{ab} , \\
  C_{abc}^\bullet & = -\Upsilon^d W_{abcd} , \\
  B_{ab}^\bullet & = 2(\mfdim - 4)\Upsilon^cC_{c(ab)} ,
 \end{split}
\end{equation}
for the conformal linearizations of the Weyl, Schouten, Cotton, and Bach tensors, respectively, where round parentheses denote symmetrization.
One also recovers the formulas
\begin{equation}
 \label{eqn:submanifold-conformal-change}
 \begin{split}
  \onabla_\alpha^\bullet\tau_\beta & = (w-1)\Upsilon_\alpha\tau_\beta - \Upsilon_\beta\tau_\alpha + \Upsilon^\gamma\tau_\gamma g_{\alpha\beta} , \\
  \onabla_\alpha^\bullet\tau_{\alpha^\prime} & = (w-1)\Upsilon_\alpha\tau_{\alpha^\prime} , \\
  L_{\alpha\beta\alpha^\prime}^\bullet & = -\Upsilon_{\alpha^\prime} g_{\alpha\beta} ,
 \end{split}
\end{equation}
for the conformal linearizations of the connections $\onabla_\alpha$ on
$TY$ and $NY$, and of the second fundamental form
$L_{\alpha\beta\alpha^\prime}$. 
In particular, the various projections of the Weyl tensor $W_{abcd}$ and the
trace-free part $\tfss_{\alpha\beta\alpha^\prime}$ of the second
fundamental form are conformal submanifold invariants of weight $2$, as defined in \cref{sec:q-bg}.
Equations~\eqref{eqn:submanifold-conformal-change} yield formulas for the conformal linearization of the tangential divergence of tangential one-forms, the tangential Laplacian on functions, and the mean curvature:
\begin{equation}
 \label{eqn:submanifold-basic-operator-linearizations}
 \begin{split}
  (\onabla^\alpha)^\bullet\tau_\alpha & = (\submfdim + w - 2)\Upsilon^\alpha\tau_\alpha , \\
  \oDelta^\bullet u & = (\submfdim + 2w - 2)\Upsilon^\alpha u_\alpha + w(\oDelta\Upsilon)u , \\
  H_{\alpha^\prime}^\bullet & = -\Upsilon_{\alpha^\prime} .
 \end{split}
\end{equation}

We conclude this subsection by introducing four more tensors, the first of which is a variant of a tensor introduced by Blitz, Gover, and Waldron~\cite{BlitzGoverWaldron2021}*{Lemma~6.1}:
\begin{align}
 \label{eqn:mP} \mP_{\alpha\beta} & := \Sch_{\alpha\beta} +
 H^{\alpha^\prime}\tfss_{\alpha\beta\alpha^\prime} +
 \frac{1}{2}H^{\alpha^\prime}H_{\alpha^\prime}g_{\alpha\beta} , \\ 
 \label{eqn:mC} \mC_{abc} & := C_{abc} - H^{\alpha^\prime}W_{abc\alpha^\prime} , \\
 \label{eqn:mC-tr} \mC_{a} & := \mC_{\beta a}{}^\beta , \\
 \label{eqn:mB} \mB_{\alpha\beta} & := B_{\alpha\beta} + 2(\mfdim -
 4)H^{\alpha^\prime}C_{\alpha^\prime(\alpha\beta)} + (\mfdim -
 4)H^{\alpha^\prime}H^{\beta^\prime}W_{\alpha\alpha^\prime\beta\beta^\prime}
 .
\end{align}
There are two key points.
First, these tensors make sense in all submanifold dimensions $\submfdim \geq 1$;
in particular, Equation~\eqref{eqn:gcP} implies
that $\mP_{\alpha\beta}$ generalizes $\oSch_{\alpha\beta} +
\Fi_{\alpha\beta}$ to all dimensions. 
Second, \cref{new-tangential-invariants} below shows that, under conformal change, these tensors depend only on tangential derivatives of the conformal factor.
In particular, they depend only on an immersion $i \colon Y^{\submfdim} \to (X^{\mfdim} , [g])$ and a choice of representative $h \in
i^\ast[g]$, in the sense that they are independent of the choice of local extension of $h$ to a metric in $[g]$ defined on a neighborhood of $i(Y)$.

\begin{lemma}
 \label{new-tangential-invariants}
 Let $i \colon Y^{\submfdim} \to (X^{\mfdim},g)$ be an immersion with $1 \leq \submfdim < \mfdim$ and $\mfdim \geq 3$.
 Then
 \begin{align*}
  \mP_{\alpha\beta}^\bullet & = -\onabla_\alpha\onabla_\beta\Upsilon , \\
  \mC_{abc}^\bullet & = -\Upsilon^\alpha W_{abc\alpha} , \\
  \mC_{a}^\bullet & = -\Upsilon^\gamma W_{\beta a}{}^\beta{}_\gamma , \\
  \mB_{\alpha\beta}^{\bullet} & = 2(\mfdim - 4)\Upsilon^\gamma \mC_{\gamma(\alpha\beta)} , \\
  \Di_{\alpha\alpha^\prime}^\bullet & = -\Upsilon^\beta\tfss_{\beta\alpha\alpha^\prime} .
 \end{align*}
\end{lemma}

\begin{remark}
 \label{rk:conformal-geodesic}
 If $\submfdim=1$, then each of $\tfss_{\alpha\beta\alpha^\prime}$, $W_{\alpha\beta ab}$, and $\mC_{\alpha\beta a}$ vanishes.
 \Cref{new-tangential-invariants} thus implies that $\Di_{\alpha\alpha^\prime}$ and $\mB_{\alpha\beta}$ are conformal invariants of immersed curves.
 Indeed, the condition $\Di_{\alpha\alpha^\prime}=0$ characterizes unparameterized conformal circles~\citelist{ \cite{Belgun2015}*{Theorem~5.4} \cite{CurryGoverSnell2023}*{Theorem~1.2 and Proposition~5.4} }.
\end{remark}

\begin{proof}
 Equations~\eqref{eqn:ambient-conformal-change} and~\eqref{eqn:submanifold-conformal-change} imply that
 \begin{align*}
  \mP_{\alpha\beta}^\bullet & = -\nabla_\alpha\nabla_\beta\Upsilon - \Upsilon^{\alpha^\prime}L_{\alpha\beta\alpha^\prime} , \\
  \mC_{abc}^\bullet & = -\Upsilon^\alpha W_{abc\alpha} , \\
  \mB_{\alpha\beta}^\bullet & = 2(\mfdim - 4)\Upsilon^\gamma\mC_{\gamma(\alpha\beta)} , \\
  \Di_{\alpha\alpha^\prime}^\bullet & = -\nabla_\alpha\nabla_{\alpha^\prime}\Upsilon + \onabla_\alpha\nabla_{\alpha^\prime}\Upsilon + H_{\alpha^\prime}\nabla_\alpha\Upsilon .
 \end{align*}
 Combining these with Equations~\eqref{eqn:second-fundamental-form-dual}
 and~\eqref{eqn:second-fundamental-form} yields the desired
 conclusion. 
\end{proof}

\section{New conformal submanifold invariants}
\label{sec:invariants}

In this section we identify and apply four new non-obvious scalar conformal submanifold invariants.
To that end, we introduce some useful notation:
\begin{equation}
 \label{eqn:conventions}
 \begin{aligned}
  \tfss_{\alpha\beta}^2 & := \tfss^\gamma{}_{\alpha\alpha^\prime}\tfss_{\gamma\beta}{}^{\alpha^\prime} , & \lv\tfss\rv^2 & = \tfss_{\alpha\beta\alpha^\prime}\tfss^{\alpha\beta\alpha^\prime} , \\
 \lp \tfss^2 , \oSch \rp & := \tfss_{\alpha\beta}^2\oSch^{\alpha\beta} , & \tr\tfss_{\alpha^\prime}^3 & := \tfss_\alpha{}^\beta{}_{\beta^\prime}\tfss_\beta{}^{\gamma\beta^\prime}\tfss_\gamma{}^\alpha{}_{\alpha^\prime} , \\
  \oDelta \tfss_{\alpha\beta\alpha^\prime} & := \onabla^\gamma\onabla_\gamma\tfss_{\alpha\beta\alpha^\prime} .
 \end{aligned}
\end{equation}
We use similar notation to denote other inner products and squared lengths.

\subsection{Identification of invariants}
\label{subsec:invariants/construct}

We begin by identifying four scalar conformal submanifold invariants of weight $-4$ in general dimension and codimension.
Our first step is to compute some useful tangential divergences.

\begin{lemma}
 \label{div-F}
 Let $i \colon Y^{\submfdim} \to (X^{\mfdim},g)$ be an immersion with $1 \leq \submfdim < \mfdim$ and $\mfdim \geq 3$.
 Then
 \begin{align}
  \label{eqn:div-mP} \onabla^\beta \mP_{\alpha\beta} & = \onabla_\alpha\mP_\beta{}^\beta + \mC_\alpha  - \Di^{\beta\alpha^\prime}\tfss_{\beta\alpha\alpha^\prime} , \\
  \label{eqn:div-tfss2} \onabla^\beta\tfss_{\alpha\beta}^2 & = \frac{1}{2}\onabla_\alpha\lv\tfss\rv^2 - (\submfdim - 2)\Di^{\beta\alpha^\prime}\tfss_{\alpha\beta\alpha^\prime} - W^{\beta\gamma\alpha^\prime}{}_\gamma\tfss_{\alpha\beta\alpha^\prime} - \tfss^{\beta\gamma\alpha^\prime}W_{\beta\alpha\gamma\alpha^\prime} , \\
  \label{eqn:div-W} \onabla^\beta W_{\alpha\gamma\beta}{}^\gamma & = \frac{1}{2}\onabla_\alpha W_{\beta\gamma}{}^{\beta\gamma} - (\submfdim - 2)\mC_\alpha - \tfss^{\beta\gamma\alpha^\prime}W_{\beta\alpha\gamma\alpha^\prime} - \tfss_\alpha{}^{\beta\alpha^\prime}W_{\beta\gamma\alpha^\prime}{}^\gamma .
 \end{align}
\end{lemma}

\begin{proof}
 First we compute the tangential divergence of $\mP_{\alpha\beta}$.
 Equation~\eqref{eqn:second-fundamental-form} implies that
 \begin{equation*}
  \onabla^\beta\Sch_{\alpha\beta} - \onabla_\alpha\Sch_\beta{}^\beta = C_{\beta\alpha}{}^\beta - \Sch^{\beta\alpha^\prime}\tfss_{\alpha\beta\alpha^\prime} + (\submfdim - 1)H^{\alpha^\prime}\Sch_{\alpha\alpha^\prime} .
 \end{equation*}
 Combining this with the definitions~\eqref{eqn:defn-Di} and~\eqref{eqn:mP} of $\Di_{\alpha\alpha^\prime}$ and $\mP_{\alpha\beta}$, respectively, yields
 \begin{align*}
  \onabla^\beta\mP_{\alpha\beta} - \onabla_\alpha\mP_\beta{}^\beta & = C_{\beta\alpha}{}^\beta - \Di^{\beta\alpha^\prime}\tfss_{\alpha\beta\alpha^\prime} + (\submfdim - 1)H^{\alpha^\prime}\Di_{\alpha\alpha^\prime} + H^{\alpha^\prime}\onabla^\beta\tfss_{\alpha\beta\alpha^\prime} .
 \end{align*}
 Combining this with Equations~\eqref{eqn:gcD} and~\eqref{eqn:mC-tr} yields Equation~\eqref{eqn:div-mP}.

 Second we compute the tangential divergence of $\tfss_{\alpha\beta}^2$.
 Direct computation using Equation~\eqref{eqn:gcD} yields
 \begin{equation*}
  \onabla^\beta \tfss^2_{\alpha\beta} = -(\submfdim - 1)\Di^{\beta\alpha^\prime}\tfss_{\alpha\beta\alpha^\prime} - W^{\beta\gamma\alpha^\prime}{}_\gamma \tfss_{\alpha\beta\alpha^\prime} + \tfss^{\beta\gamma\alpha^\prime}\onabla_\beta \tfss_{\alpha\gamma\alpha^\prime} .
 \end{equation*}
 Rewriting the last summand using Equation~\eqref{eqn:gcdL} yields Equation~\eqref{eqn:div-tfss2}.

 Third we compute the tangential divergence of $W_{\alpha\gamma\beta}{}^\gamma$.
 On the one hand, Equation~\eqref{eqn:second-fundamental-form} implies that
 \begin{align*}
  \onabla^\beta W_{\alpha\gamma\beta}{}^\gamma & = \nabla^\beta W_{\alpha\gamma\beta}{}^\gamma + \tfss_\alpha{}^{\beta\alpha^\prime}W_{\beta\gamma\alpha^\prime}{}^\gamma - \tfss^{\beta\gamma\alpha^\prime}W_{\beta\alpha\gamma\alpha^\prime} + \submfdim H^{\alpha^\prime} W_{\alpha\beta\alpha^\prime}{}^\beta , \\
  \onabla_\alpha W_{\beta\gamma}{}^{\beta\gamma} & = \nabla_\alpha W_{\beta\gamma}{}^{\beta\gamma} + 4\tfss_\alpha{}^{\beta\alpha^\prime}W_{\beta\gamma\alpha^\prime}{}^\gamma + 4H^{\alpha^\prime}W_{\alpha\beta\alpha^\prime}{}^{\beta} .
 \end{align*}
 On the other hand, the Weyl--Bianchi identity~\eqref{eqn:weyl-bianchi} implies that
 \begin{equation}
  \label{eqn:weyl-bianchi-divW}
  2\nabla^\beta W_{\alpha\gamma\beta}{}^\gamma = \nabla^\beta W_{\alpha\gamma\beta}{}^\gamma + \nabla_\gamma W^\beta{}_{\alpha\beta}{}^\gamma = \nabla_\alpha W_{\beta\gamma}{}^{\beta\gamma} - 2(\submfdim - 2)C_{\beta\alpha}{}^\beta .
 \end{equation}
 Combining these three identities with Equation~\eqref{eqn:mC-tr} yields Equation~\eqref{eqn:div-W}.
\end{proof}

Our first two scalar conformal submanifold invariants involve a single derivative of the trace-free second fundamental form or the Weyl tensor.
These invariants are tangential divergences in the critical case $\dim Y = 4$.
This is in contrast to the intrinsic case:
there is no nonzero scalar conformal invariant of weight~$-4$ that is a natural divergence.

\begin{proposition}
 \label{conformally-invariant-divergence}
 Let $i \colon Y^{\submfdim} \to (X^{\mfdim},g)$ be an immersion with $1 \leq \submfdim < \mfdim$ and $\mfdim \geq 3$.
 Then
 \begin{align*}
  \mK_1 & := \onabla^\alpha\left( \tfss^{\beta\gamma\alpha^\prime}W_{\alpha\beta\alpha^\prime\gamma} \right) + (\submfdim - 4)\tfss^{\alpha\beta\alpha^\prime}\mC_{\alpha\alpha^\prime\beta} , \\
  \mK_2 & := \onabla^\alpha\left( \tfss_\alpha{}^{\beta\alpha^\prime}W_{\beta\gamma\alpha^\prime}{}^\gamma \right) + (\submfdim - 4)\Di^{\alpha\alpha^\prime}W_{\alpha\beta\alpha^\prime}{}^{\beta} ,
 \end{align*}
 are conformally invariant of weight $-4$.
\end{proposition}

\begin{remark}
 \Cref{rk:conformal-geodesic} implies that if $\submfdim=1$, then $\mK_1=0$ and $\mK_2=0$.
\end{remark}

\begin{remark}
 \label{rk:mK}
 Equations~\eqref{eqn:gcdL} and~\eqref{eqn:gcD} imply that
 \begin{align*}
  \mK_1 & = \tfss^{\beta\gamma\alpha^\prime}\onabla^\alpha W_{\alpha\beta\alpha^\prime\gamma} + \frac{1}{2}W_{\alpha\beta\alpha^\prime\gamma}W^{\alpha\beta\alpha^\prime\gamma} + \Di^{\alpha\alpha^\prime}W_{\alpha\beta\alpha^\prime}{}^\beta + (\submfdim-4)\tfss^{\alpha\beta\alpha^\prime}\mC_{\alpha\alpha^\prime\beta} , \\
  \mK_2 & = \tfss^{\alpha\beta\alpha^\prime}\onabla_\alpha W_{\beta\gamma\alpha^\prime}{}^\gamma - 3\Di^{\alpha\alpha^\prime}W_{\alpha\beta\alpha^\prime}{}^\beta - W_{\alpha\beta\alpha^\prime}{}^\beta W^{\alpha\gamma\alpha^\prime}{}_{\gamma} .
 \end{align*}
 Note that $\mK_2=0$ on hypersurfaces.
 Also, $\mK_1 = \tfss^{\beta\gamma\alpha^\prime}\onabla^\alpha W_{\alpha\beta\alpha^\prime\gamma} + \frac{1}{2}W_{\alpha\beta\alpha^\prime\gamma}W^{\alpha\beta\alpha^\prime\gamma}$ when $\submfdim=4$ and $\mfdim=5$; hence $\tfss^{\beta\gamma\alpha^\prime}\onabla^\alpha W_{\alpha\beta\alpha^\prime\gamma}$ is conformally invariant in this case.
\end{remark}

\begin{proof}
 Recall that $\tfss_{\alpha\beta\alpha^\prime}$ and $W_{\alpha\beta\gamma\alpha^\prime}$ are conformal submanifold invariants of weight $2$.
 We deduce from Equations~\eqref{eqn:submanifold-basic-operator-linearizations} that
 \begin{align}
  \label{eqn:div-LW-Wfree} \left( \onabla^\alpha (\tfss^{\beta\gamma\alpha^\prime}W_{\alpha\beta\alpha^\prime\gamma}) \right)^\bullet & = (\submfdim - 4)\Upsilon^\alpha\tfss^{\beta\gamma\alpha^\prime} W_{\alpha\beta\alpha^\prime\gamma} , \\
  \label{eqn:div-LW-Dfree} \left( \onabla^\alpha (\tfss_\alpha{}^{\beta\alpha^\prime} W_{\beta\gamma\alpha^\prime}{}^\gamma) \right)^\bullet & = (\submfdim - 4)\Upsilon^\alpha\tfss_\alpha{}^{\beta\alpha^\prime}W_{\beta\gamma\alpha^\prime}{}^\gamma .
 \end{align}
 The conclusion now follows from \cref{new-tangential-invariants}.
\end{proof}

Our third scalar conformal submanifold invariant is the following:

\begin{proposition}
 \label{Wm}
 Let $i \colon Y^{\submfdim} \to (X^{\mfdim},g)$ be an immersion with $1 \leq \submfdim < \mfdim$ and $\mfdim \geq 3$, with $\mfdim \not= 4$.
 Then
 \begin{multline}
  \label{Iformula}
  \Wm := (\submfdim - 1)\bigl( -\oDelta \trFi + 2\trFi\mP_\alpha{}^\alpha \bigr) + (\submfdim - 6) \Bigl[ \bigl( \tfss_{\alpha\beta}^2 - W_{\alpha\gamma\beta}{}^\gamma \bigr) \mP^{\alpha\beta} \\
   + \onabla^\alpha \bigl( \mC_\alpha - \Di^{\beta\alpha^\prime}\tfss_{\alpha\beta\alpha^\prime} \bigr) + \frac{\submfdim - 3}{\mfdim - 4}\mB_\alpha{}^\alpha - (\submfdim-3)\lv\Di\rv^2 \Bigr]
 \end{multline}
 is conformally invariant of weight $-4$.
\end{proposition}

\begin{remark}
 \label{rk:Wm-dim-1}
 If $\submfdim=1$, then $\mI = -10\lv\Di\rv^2 + \frac{10}{\mfdim-4}\mB_\alpha{}^\alpha$ is conformally invariant by \cref{rk:conformal-geodesic}.
\end{remark}

\begin{remark}
 If $k=6$, then the operator $-\oDelta + 2\mP_\alpha{}^\alpha$ is conformally invariant on densities of weight $-2$.
 This explains the invariance of $\Wm$ in this dimension.
 A similar phenomenon occurs for other conformal invariants below.
\end{remark}

\begin{remark}
 \label{rk:ell3m4}
 Interpreting $(\submfdim-3)/(\mfdim-4)=1$ when $\submfdim=3$ and $\mfdim=4$ extends the definition of $\Wm$ to these dimensions.
 A straightforward modification of the proof of \cref{Wm} shows that $\Wm$ remains conformally invariant of weight $-4$.
\end{remark}

\begin{remark}
 \label{rk:pole}
 Since the Bach tensor is conformally invariant in dimension four, the residue of $\Wm$ at $\mfdim=4$ is conformally invariant.
 In other words, $(\mfdim-4)\Wm$ is conformally invariant and defined in all dimensions $1 \leq \submfdim < \mfdim$ and $\mfdim \geq 3$.
 Again, a similar phenomenon occurs for other conformal invariants below.
\end{remark}

\begin{remark}
 \label{rk:tyrrell}
 If $i \colon Y^{\submfdim} \to (X^{\mfdim},g)$ is a minimal immersion into an Einstein manifold with $\Ric(g) = \lambda(\mfdim - 1)g$, then
 \begin{equation*}
  \Wm = (\submfdim - 1)\left( -\oDelta\trFi + 2\lambda(\submfdim - 3)\trFi \right) .
 \end{equation*}
 On minimal four-dimensional hypersurfaces in Poincar\'e--Einstein manifolds, this realizes the last integrand in Tyrrell's formula~\cite{Tyrrell2022}*{Equation~(1.4)} for the renormalized area as a constant multiple of $\Wm$.
 In particular, the conformal invariance of $\Wm$ explains the finiteness of that integral.
\end{remark}

\begin{proof}
 \Cref{rk:Wm-dim-1} implies that we can assume $\submfdim \geq 2$.
 
 First, set
 \begin{equation}
  \label{eqn:defn-I1}
  I_1 := (\submfdim - 1)\left( -\oDelta \trFi + (\submfdim - 4)\trFi\mP_\alpha{}^\alpha \right) .
 \end{equation}
 Since $G$ is a conformal submanifold invariant of weight $-2$, we conclude from Equations~\eqref{eqn:submanifold-basic-operator-linearizations} and \cref{new-tangential-invariants} that
 \begin{equation}
  \label{eqn:conf-I1}
  I_1^\bullet = -(\submfdim - 1)(\submfdim - 6)\left( \onabla^\alpha \circ \trFi \circ \onabla_\alpha \right)\Upsilon .
 \end{equation}
 
 Second, set
 \begin{equation}
  \label{eqn:defn-I2}
  I_2 := \onabla^\alpha\bigl(  \mC_\alpha - \Di^{\beta\alpha^\prime}\tfss_{\alpha\beta\alpha^\prime} \bigr) + \frac{\submfdim - 4}{2(\mfdim - 4)}\mB_\alpha{}^\alpha - \frac{\submfdim - 4}{2}\lv\Di\rv^2 .
 \end{equation}
 Combining Equation~\eqref{eqn:submanifold-basic-operator-linearizations} and \cref{new-tangential-invariants} yields
 \begin{equation}
  \label{eqn:conf-I2}
  I_2^\bullet = \left( \onabla^\alpha \circ \bigl( \tfss_{\alpha\beta}^2 - W_{\alpha\gamma\beta}{}^\gamma \bigr) \circ \onabla^\beta \right)\Upsilon .
 \end{equation}
 
 Third, set
 \begin{equation}
  \label{eqn:defn-I3}
  I_3 := \left( \tfss_{\alpha\beta}^2 - W_{\alpha\gamma\beta}{}^\gamma \right)\mP^{\alpha\beta} - (\submfdim - 1)\trFi\mP_\alpha{}^\alpha + \frac{\submfdim - 2}{2(\mfdim - 4)}\mB_\alpha{}^\alpha - \frac{\submfdim - 2}{2}\lv\Di\rv^2 .
 \end{equation}
 \Cref{new-tangential-invariants} implies that
 \begin{multline*}
  I_3^\bullet = -\left(\tfss_{\alpha\beta}^2 - W_{\alpha\gamma\beta}{}^\gamma \right)\onabla^\alpha\onabla^\beta\Upsilon + (\submfdim - 1)\trFi\oDelta\Upsilon \\
   - (\submfdim - 2)\mC_\alpha\Upsilon^\alpha + (\submfdim - 2)\Upsilon^\beta\Di^{\alpha\alpha^\prime}\tfss_{\alpha\beta\alpha^\prime} .
 \end{multline*}
 Combining this with \cref{div-F} yields
 \begin{equation}
  \label{eqn:conf-I3}
  I_3^\bullet = -\left( \onabla^\alpha \circ \bigl( \tfss_{\alpha\beta}^2 - W_{\alpha\gamma\beta}{}^\gamma - (\submfdim - 1)\trFi g_{\alpha\beta} \bigr) \circ \onabla^\beta \right)\Upsilon .
 \end{equation}

 Finally, observe that
 \begin{equation}
  \label{eqn:Wm-via-Is}
  \Wm = I_1 + (\submfdim - 6)(I_2 + I_3) .
 \end{equation}
 We conclude from Equations~\eqref{eqn:conf-I1}, \eqref{eqn:conf-I2}, and~\eqref{eqn:conf-I3} that $\Wm$ is conformally invariant of weight $-4$.
\end{proof}

Unlike $\Wm$, our fourth scalar conformal submanifold invariant is trivial for codimension one submanifolds.

\begin{proposition}
 \label{Wn}
 Let $i \colon Y^{\submfdim} \to (X^{\mfdim},g)$ be an immersion with $1 \leq \submfdim < \mfdim$ and $\mfdim \geq 3$, with $\mfdim \not= 4$.
 Then
 \begin{multline}
  \label{Jformula}
   \Wn := -\oDelta W_{\alpha\beta}{}^{\alpha\beta} + 2W_{\alpha\beta}{}^{\alpha\beta}\mP_\gamma{}^\gamma - 2(\submfdim - 6)\biggl[ \onabla^\alpha \mC_\alpha - \mP^{\alpha\beta}W_{\alpha\gamma\beta}{}^{\gamma} \\
    + \Di^{\alpha\alpha^\prime}W_{\alpha\beta\alpha^\prime}{}^{\beta} + \tfss^{\alpha\beta\alpha^\prime}\mC_{\alpha\alpha^\prime\beta} + \frac{\submfdim - 3}{\mfdim - 4}\mB_{\alpha}{}^{\alpha} \biggr]
 \end{multline}
 is conformally invariant of weight $-4$.
\end{proposition}

\begin{remark}
 \label{Wn-zero-on-hypersurfaces}
 The definition~\eqref{eqn:second-fundamental-form-dual} of the second fundamental form and the definitions~\eqref{eqn:mP}--\eqref{eqn:mB} of $\mP_{\alpha\beta}$, $\mC_\alpha$, and $\mB_{\alpha\beta}$ imply that $\Wn=0$ on hypersurfaces.
 The key observation is that, on hypersurfaces,
 \begin{equation*}
  \onabla^\alpha \mC_\alpha = -\nabla^\alpha C_{\alpha\beta}{}^\beta - \tfss^{\alpha\beta\alpha^\prime}C_{\alpha\alpha^\prime\beta} = -B_\alpha{}^\alpha + W_{\alpha\gamma\beta}{}^\gamma \Sch^{\alpha\beta} - \tfss^{\alpha\beta\alpha^\prime}C_{\alpha\alpha^\prime\beta} .
 \end{equation*}
 This remark includes the case $\submfdim=3$ and $\mfdim=4$ upon replacing $(k-3)/(n-4)$ by~$1$ in the last term of Equation~\eqref{Jformula}.
\end{remark}

\begin{remark}
 \label{2I-plus-J}
 Note that
 \begin{multline*}
  2\Wm + \Wn = -\oDelta\lv\tfss\rv^2 + 2\lv\tfss\rv^2\mP_\alpha{}^\alpha + 2(\submfdim-6) \Bigl[ \tfss^2_{\alpha\beta}\mP^{\alpha\beta} - \onabla^\alpha(\Di^{\beta\alpha^\prime}\tfss_{\alpha\beta\alpha^\prime}) \\
   - (\submfdim-3)\lv\Di\rv^2 - \Di^{\alpha\alpha^\prime}W_{\alpha\beta\alpha^\prime}{}^\beta - \tfss^{\alpha\beta\alpha^\prime}\mC_{\alpha\alpha^\prime\beta} \Bigr]
 \end{multline*}
 is defined when $\mfdim=4$.
 It is straightforward to adapt the proofs of \cref{Wm,Wn} to conclude that it is conformally invariant for $1 \leq \submfdim < \mfdim$ and $\mfdim \geq 3$.
\end{remark}

\begin{remark}
 If $\submfdim=1$, then $\Wn = -\frac{20}{\mfdim-4}\mB_\alpha{}^\alpha$ is conformally invariant by \cref{rk:conformal-geodesic}.
\end{remark}

\begin{remark}
 Analogous to \cref{rk:tyrrell}, if $i \colon Y^{\submfdim} \to (X^{\mfdim},g)$ is a minimal immersion into an Einstein manifold with $\Ric(g) = \lambda(\mfdim - 1) g$, then
 \begin{equation*}
  \Wn = -\oDelta W_{\alpha\beta}{}^{\alpha\beta} + 2\lambda(\submfdim - 3) W_{\alpha\beta}{}^{\alpha\beta} .
 \end{equation*}
\end{remark}

\begin{proof}
 First, set
 \begin{equation*}
  J_1 := -\oDelta W_{\alpha\beta}{}^{\alpha\beta} + (\submfdim - 4)W_{\alpha\beta}{}^{\alpha\beta}\mP_\gamma{}^\gamma .
 \end{equation*}
 Since $W_{\alpha\beta}{}^{\alpha\beta}$ is conformally invariant of weight $-2$, we deduce from Equation~\eqref{eqn:submanifold-basic-operator-linearizations} and \cref{new-tangential-invariants} that
 \begin{equation}
  \label{eqn:Wn-W}
  J_1^\bullet = -(\submfdim - 6)\left( \onabla^\alpha \circ W_{\beta\gamma}{}^{\beta\gamma} \circ \onabla_\alpha \right)\Upsilon .
 \end{equation}
 
 Second, set
 \begin{align*}
  J_2 & := \onabla^\alpha \mC_\alpha + \frac{\submfdim - 4}{2(\mfdim - 4)}\mB_\alpha{}^\alpha .
 \end{align*}
 Equation~\eqref{eqn:submanifold-basic-operator-linearizations} and \cref{new-tangential-invariants} imply that
 \begin{equation}
  \label{eqn:Wn-2}
  J_2^\bullet = -\left( \onabla^\alpha \circ W_{\alpha\gamma\beta}{}^\gamma \circ \onabla^\beta \right)\Upsilon .
 \end{equation}
 
 Third, set
 \begin{equation*}
  J_3 := \left( W_{\alpha\gamma\beta}{}^\gamma - \frac{1}{2}W_{\gamma\delta}{}^{\gamma\delta}g_{\alpha\beta} \right)\mP^{\alpha\beta} - \tfss^{\alpha\beta\alpha^\prime}\mC_{\alpha\alpha^\prime\beta} - \Di^{\alpha\alpha^\prime}W_{\alpha\beta\alpha^\prime}{}^\beta - \frac{\submfdim - 2}{2(\mfdim - 4)}\mB_\alpha{}^\alpha .
 \end{equation*}
 Combining \cref{new-tangential-invariants,div-F} yields
 \begin{equation}
  \label{eqn:Wn-3}
  J_3^\bullet = -\left( \onabla^\alpha \circ \Bigl( W_{\alpha\gamma\beta}{}^\gamma - \frac{1}{2}W_{\gamma\delta}{}^{\gamma\delta}g_{\alpha\beta} \Bigr) \circ \onabla^\beta \right)\Upsilon .
 \end{equation}

 The conclusion now follows from Equations~\eqref{eqn:Wn-W}, \eqref{eqn:Wn-2}, and~\eqref{eqn:Wn-3} and the observation
 \begin{equation*}
  \Wn = J_1 - 2(\submfdim - 6)(J_2 - J_3) . \qedhere
 \end{equation*}
\end{proof}

\subsection{The fourth-order extrinsic $Q$-curvature}
\label{subsec:invariants/q}
Our main goal in this subsection is to derive a decomposition of the form~\eqref{eqn:submanifold-decomposition} for the critical extrinsic $Q$-curvature for submanifolds of dimension $\submfdim = 4$.
This is easily accomplished using a representation of $Q$ derived in~\cite{CaseGrahamKuo2023} together with our knowledge of the invariant
$\mathcal{I}$ identified in \cref{Wm}.
Our analysis applies more generally in dimensions $3\leq \submfdim < \mfdim$ with $\mfdim \neq 4$ to 
the extrinsic scalar $Q_4$ of order $4$ derived in~\cite{CaseGrahamKuo2023} and given by Equation~\eqref{Qdecompose} below, for which $Q_4=Q$ in the critical case
$\submfdim =4$. 

Recall the formula
\begin{equation}\label{Qbar}
  \oQ_4 =-\oDelta\,\otrSch -2\lv\oSch\rv^2
  +\frac{\submfdim}{2}\otrSch^2 
\end{equation}
for the intrinsic $Q$-curvature of order $4$ in general dimension $\submfdim$.
Define
 \begin{multline}
   \label{Q4Gformula}
   \QtrFi := (\submfdim - 2)\oDelta\trFi
   - (\submfdim - 6)\onabla^\alpha\bigl(  \mC_\alpha -
   \Di^{\beta\alpha^\prime}\tfss_{\alpha\beta\alpha^\prime} \bigr)
   - 2(\submfdim -   4)\trFi\mP_\alpha{}^\alpha\\
   - (\submfdim - 4)^2\Fi_{\alpha\beta}\mP^{\alpha\beta} 
    - \frac{(\submfdim - 4)(\submfdim - 5)}{\mfdim - 4}\mB_\alpha{}^\alpha +
   (\submfdim - 4)(\submfdim - 5)\lv\Di\rv^2 . 
 \end{multline}
Observe that $\QtrFi$ is a divergence in the critical case $\submfdim =4$. 

\begin{proposition}
 \label{general-Q-formula}
 Let $i \colon Y^{\submfdim} \to (X^{\mfdim},g)$ be an immersion with $3 \leq \submfdim < \mfdim$ and $\mfdim \not= 4$.
 Then
      \begin{equation}
  \label{eqn:general-Q4}
  Q_4 = \oQ_4 + \QtrFi + \Wm + 2\lv \Fi\rv^2 - \frac{\submfdim}{2}\trFi^2 . 
 \end{equation}
\end{proposition} 

\begin{remark}
 Since $Q_4$ depends only on a background conformal class $[g]$ on
$X$ and a representative $h \in i^\ast[g]$ on $Y$, we deduce from  Equation~\eqref{eqn:general-Q4} that the same is true of $\QtrFi$.
\end{remark}

\begin{remark}
 The conformal transformation law of $\QtrFi$ can be written in terms of the operator
 \begin{equation*}
  \PtrFi := P_4 - \oP_4 - \frac{\submfdim-4}{2}\left( \Wm + 2\lv\Fi\rv^2 - \frac{\submfdim}{2}\trFi^2 \right) ,
 \end{equation*}
 where $P_4$ is the extrinsic Paneitz operator defined in~\cite{CaseGrahamKuo2023} and $\oP_4$ is the intrinsic Paneitz operator.
 Indeed, Equation~\eqref{eqn:general-Q4} and the formula~\cite{CaseGrahamKuo2023}*{Equation~(5.12)} for $P_4 - \oP_4$ imply that $\PtrFi$ is equivalently written
 \begin{equation*}
  \PtrFi = \onabla^\alpha \circ (4\Fi_{\alpha\beta} - (\submfdim-2)\trFi g_{\alpha\beta}) \circ \onabla^\beta + 
  \frac{\submfdim-4}{2}\QtrFi ,
 \end{equation*}
 and in particular is second-order.
 Since $\Wm + 2\lv\Fi\rv^2 - \frac{\submfdim}{2}\trFi^2$ is conformally invariant of weight $-4$, it follows from the conformal transformation laws of $P_4$ and $\oP_4$ that if $\hh = e^{2\Upsilon}h$, then
 \begin{align*}
  \PtrFi^{\hh} &
  = e^{(-\submfdim/2-2)\Upsilon} \circ \PtrFi^h \circ 
e^{(\submfdim/2-2)\Upsilon}, & \text{for all $\submfdim \geq 3$} , \\
e^{4\Upsilon}\QtrFi^{\hh} &
= \QtrFi^h  + \PtrFi^h\Upsilon , & \text{if $\submfdim = 4$} . 
 \end{align*} 
\end{remark}

\begin{proof}
 Case, Graham, and Kuo showed~\cite{CaseGrahamKuo2023}*{Theorem~5.3} that 
 \begin{equation}\label{Qdecompose}
  Q_4 = \overline{Q}_4 + \widetilde{Q}_4,
 \end{equation}
 where
\begin{equation*}
  \widetilde{Q}_4 := - \oDelta\trFi - 2\lv\Fi\rv^2 + \frac{\submfdim}{2}\trFi^2 -
  4\Fi_{\alpha\beta}\oSch^{\alpha\beta} + \submfdim\trFi\otrSch - \frac{2}{\mfdim - 4}\mB_\alpha{}^\alpha +
  2\lv\Di\rv^2 . 
\end{equation*}
 Adding Equations~\eqref{Iformula} and~\eqref{Q4Gformula} gives
  \begin{multline*}
    \Wm + \QtrFi = -\oDelta\trFi +6\trFi\mP_\alpha{}^\alpha
+(\submfdim - 6) \left( \tfss_{\alpha\beta}^2
-W_{\alpha\gamma\beta}{}^\gamma \right) \mP^{\alpha\beta}\\
- (\submfdim -  4)^2\Fi_{\alpha\beta}\mP^{\alpha\beta}-\frac{2}{\mfdim -
  4}\mB_\alpha{}^\alpha
+2\lv\Di\rv^2 . 
  \end{multline*}
  We deduce that
  \begin{multline*}
    \widetilde{Q}_4 -\big(\Wm + \QtrFi \big) =
    -2|\Fi|^2 +\frac{\submfdim}{2}\trFi^2 -4\Fi_{\alpha\beta}\oSch^{\alpha\beta} 
    + \submfdim \trFi\otrSch - 6\trFi\mP_\alpha{}^\alpha\\
    -(\submfdim - 6) \left( \tfss_{\alpha\beta}^2
-W_{\alpha\gamma\beta}{}^\gamma \right) \mP^{\alpha\beta}
+ (\submfdim -  4)^2\Fi_{\alpha\beta}\mP^{\alpha\beta}.
  \end{multline*}
This reduces to $2\lv \Fi\rv^2 - \frac{\submfdim}{2}\trFi^2$ upon using Equations~\eqref{eqn:gcP}, \eqref{eqn:defn-Fialkow}, and~\eqref{eqn:mP} to substitute
\begin{equation*}
  \oSch^{\alpha\beta}= \mP^{\alpha\beta}-\Fi^{\alpha\beta},\qquad
  \otrSch = \mP_\alpha{}^\alpha - \trFi,\qquad
  \tfss_{\alpha\beta}^2-W_{\alpha\gamma\beta}{}^\gamma
  =(\submfdim -2)\Fi_{\alpha\beta}+\trFi g_{\alpha\beta} . \qedhere
\end{equation*}  
\end{proof}

The desired decomposition of $Q$ when $\submfdim = 4$ is an immediate
consequence:

\begin{proposition}
 \label{Q decomposition}
 Let $i \colon Y^4 \to (X^{\mfdim},g)$ be an immersion with $\mfdim >4$.
 Then 
 \begin{equation}
  \label{eqn:Q4-formula}
  Q = 2\oPf +\mathcal{W}_Q
   - \onabla^\alpha \left( \onabla_\alpha\otrSch - 2\onabla_\alpha\trFi -
   2\mC_\alpha + 2\Di^{\beta\alpha^\prime}\tfss_{\alpha\beta\alpha^\prime}
   \right)
 \end{equation}
with $\mathcal{W}_Q=- \frac{1}{4}\lv \oW \rv^2 + \Wm + 2\lv \Fi\rv^2 - 2\trFi^2$ . 
 \end{proposition} 
\begin{proof}
Recall that $2\oPf = \frac{1}{4}\lv\oW\rv^2 - 2\lv\oSch\rv^2 + 2\otrSch^2$
when $\submfdim=4$.
Substituting Equations~\eqref{Qbar} and~\eqref{Q4Gformula} into
Equation~\eqref{eqn:general-Q4} gives Equation \eqref{eqn:Q4-formula}.
\end{proof}

Our formula for the renormalized area of a minimal four-manifold in a Poincar\'e--Einstein manifold immediately follows.

\begin{proof}[Proof of \cref{renormalized-area-special}]
 Apply \cref{main-theorem,Q decomposition}.
\end{proof}

\subsection{Comparison to other conformal submanifold invariants}
\label{subsec:invariants/comparison}

Our derivation of the invariants of Subsection~\ref{subsec:invariants/construct} 
allows us to generalize to general dimension and codimension scalar
conformal submanifold invariants discovered previously in special cases.
In this subsection we identify these generalizations and defer to the
appendix the proofs, which are long but straightforward computations
using the formulas developed in Subsection~\ref{subsec:invariants/construct}.

Blitz, Gover, and Waldron
derived~\cite{BlitzGoverWaldron2021}*{Theorem~1.2} a conformal invariant
$W\!m$ for four-dimensional hypersurfaces which involves second derivatives of $\tfss_{\alpha\beta\alpha^\prime}$.
We extend their invariant to higher dimensions and higher codimensions.

\begin{proposition}
 \label{blitz-gover-waldron}
 Let $i \colon Y^{\submfdim} \to (X^{\mfdim},g)$ be an immersion with $3 \leq \submfdim < \mfdim$ and $\submfdim \not= 6$.
 Define the 
weight $-4$ scalar conformal submanifold invariant
 \begin{align*}
  \MoveEqLeft[0] W\!m := \frac{\submfdim(\submfdim-1)}{4(\submfdim-6)}(2\Wm + \Wn) + \frac{\submfdim-3}{2}\mK_1 - \frac{\submfdim^2-2\submfdim+3}{2(\submfdim-1)}\mK_2 - \frac{\submfdim-3}{2} W_{\alpha\beta\alpha^\prime}{}^\beta W^{\alpha\gamma\alpha^\prime}{}_\gamma \\
   & \quad - \frac{\submfdim-3}{4}W_{\alpha\beta\alpha^\prime\gamma} W^{\alpha\beta\alpha^\prime\gamma} - \frac{\submfdim-3}{2}W_{\alpha\beta\gamma\delta}\tfss^{\alpha\gamma\alpha^\prime}\tfss^{\beta\delta}{}_{\alpha^\prime} - \frac{\submfdim-3}{2}W_{\alpha\beta\alpha^\prime\beta^\prime}\tfss^{\gamma\alpha\alpha^\prime}\tfss_\gamma{}^{\beta\beta^\prime} \\
   & \quad - \frac{\submfdim^2 - 3\submfdim + 6}{2}\tfss^2_{\alpha\beta}\Fi^{\alpha\beta} - \frac{\submfdim(\submfdim-1)}{2(\submfdim-6)}\trFi\lv\tfss\rv^2 - \frac{\submfdim-3}{2}\tfss^{\alpha\beta\alpha^\prime}\tfss_{\alpha\beta\beta^\prime}\tfss^{\gamma\delta\beta^\prime}\tfss_{\gamma\delta\alpha^\prime} \\
   & \quad - \frac{\submfdim-3}{2}\lv\tfss^2\rv^2 + (\submfdim-3)\tfss^{\alpha\beta\alpha^\prime}\tfss^{\gamma\delta}{}_{\alpha^\prime}\tfss_{\alpha\gamma\beta^\prime}\tfss_{\beta\delta}{}^{\beta^\prime}.
 \end{align*}
 If $\submfdim = 4$ and $\mfdim = 5$, then
 \begin{align*}
  W\!m & = \frac{1}{2}\tfss^{\alpha\beta\alpha^\prime}\oDelta\tfss_{\alpha\beta\alpha^\prime} + \frac{4}{3}\onabla^\alpha(\tfss_\alpha{}^{\beta\alpha^\prime}\onabla^\gamma\tfss_{\gamma\beta\alpha^\prime}) + \frac{3}{2}\oDelta\lv\tfss\rv^2 - \frac{7}{2}\otrSch\lv\tfss\rv^2 \\
   & \quad - 6\tfss^{\alpha\beta\alpha^\prime} C_{\alpha\alpha^\prime\beta} + 4\tfss^2_{\alpha\beta}\oSch^{\alpha\beta} - 6H^{\alpha^\prime}\tr\tfss_{\alpha^\prime}^3 + 12H^{\alpha^\prime}\tfss^{\alpha\beta}{}_{\alpha^\prime}\Fi_{\alpha\beta} 
 \end{align*}
 equals the invariant of the same name
 derived in~\cite{BlitzGoverWaldron2021}*{Theorem~1.2}, where we recall the conventions~\eqref{eqn:conventions}.
\end{proposition}

\begin{remark}
 \label{rk:add-Wn-and-mK2}
 One can add any multiple of $\Wn$ and $\mK_2$ to $W\!m$ without losing the fact that $W\!m$ equals the invariant derived in~\cite{BlitzGoverWaldron2021} when evaluated at four-dimensional hypersurfaces.
 By \cref{2I-plus-J}, our choice of multiple of $\Wn$ removes the pole at $\mfdim=4$.
 Our choice of multiple of $\mK_2$ arises from our proof of \cref{blitz-gover-waldron} in the appendix.
\end{remark}

Juhl derived~\cite{Juhl2022}*{Proposition~8.1} two conformal invariants $\mJ_1$ and $\mJ_2$ for hypersurfaces which involve a transverse derivative of $R_{abcd}$.
On compact four-dimensional hypersurfaces, $\int \mJ_1$ equals the global conformal invariant identified by Astaneh and Solodukhin~\cite{AstanehSolodukhin2021}*{Equation~(29)}.
We extend Juhl's invariants to higher codimension.

\begin{proposition}
 \label{juhl1}
 Let $i \colon Y^{\submfdim} \to (X^{\mfdim},g)$ be an immersion with $4 \leq \submfdim < \mfdim$ and $\submfdim \not= 6$.
 Define the weight $-4$ scalar conformal submanifold invariant
 \begin{align*}
   \mJ_1 & := -\frac{\submfdim-2}{2(\submfdim-3)(\submfdim-6)}(2\Wm +\Wn) +
   \frac{\submfdim-2}{\submfdim-3}(\mK_1 + \mK_2 + \tfss_{\alpha\beta}^2\Fi^{\alpha\beta}) + \tfss^2_{\alpha\beta}W^{\alpha\gamma\beta}{}_\gamma \\ 
    & \quad + \frac{\submfdim-2}{(\submfdim-3)(\submfdim-6)}\trFi\lv\tfss\rv^2 + \tfss^{\alpha\gamma\alpha^\prime}\tfss^{\beta\delta}{}_{\alpha^\prime}W_{\alpha\beta\gamma\delta} - \tfss^{\alpha\beta\alpha^\prime}\tfss_{\alpha\beta}{}^{\beta^\prime}W_{\alpha^\prime\gamma\beta^\prime}{}^\gamma \\
    & \quad + \tfss^{\gamma\alpha\alpha^\prime}\tfss_\gamma{}^{\beta\beta^\prime}(2W_{\alpha\alpha^\prime\beta\beta^\prime} - W_{\alpha\beta\alpha^\prime\beta^\prime}) - \frac{1}{2}W_{\alpha\beta\alpha^\prime\gamma}W^{\alpha\beta\alpha^\prime\gamma} + W_{\alpha\beta\alpha^\prime}{}^\beta W^{\alpha\gamma\alpha^\prime}{}_\gamma .
 \end{align*}
 If \codimone, then
 \begin{align*}
  \mJ_1 & = \frac{\submfdim-4}{(\submfdim-3)(\submfdim-6)}\left( \oDelta\lv\tfss\rv^2 - \frac{\submfdim-2}{\submfdim-4}\otrSch\lv\tfss\rv^2\right) - \frac{1}{\submfdim-3}\onabla^\alpha\onabla^\beta\tfss^2_{\alpha\beta} +\tfss^{\alpha\beta\alpha^\prime}\nabla_{\alpha^\prime}W_{\alpha\gamma\beta}{}^\gamma \\
   & \quad + \frac{\submfdim-2}{(\submfdim-1)^2}(\onabla^\beta\tfss_{\beta\alpha\alpha^\prime})(\onabla_\gamma\tfss^{\gamma\alpha\alpha^\prime}) - \frac{\submfdim-2}{\submfdim-3}\tfss^2_{\alpha\beta}\oSch^{\alpha\beta} - 2H_{\alpha^\prime}\tfss^{\alpha\beta\alpha^\prime}W_{\alpha\gamma\beta}{}^\gamma
 \end{align*}
 equals the invariant derived in~\cite{Juhl2022}*{Equation~(8.1)}.  
\end{proposition}

\begin{proposition}
 \label{juhl2}
 Let $i \colon Y^{\submfdim} \to (X^{\mfdim},g)$ be an immersion with $4 \leq \submfdim < \mfdim$ and $\submfdim \not= 6$.
 Define the weight $-4$ scalar conformal submanifold invariant
 \begin{multline*}
  \mJ_2 := -\frac{1}{2(\submfdim-3)(\submfdim-6)}(2\Wm +\Wn) +
  \frac{1}{\submfdim-3}(\mK_1+ \mK_2) \\
   - \frac{\submfdim-4}{\submfdim-3}\tfss^2_{\alpha\beta}\Fi^{\alpha\beta} + \frac{1}{(\submfdim-3)(\submfdim-6)}\trFi\lv\tfss\rv^2.
 \end{multline*}
 If \codimone, then
 \begin{align*}
  \mJ_2 & = -\tfss^{\alpha\beta\alpha^\prime}\nabla_{\alpha^\prime}\Sch_{\alpha\beta} - \tfss^{\alpha\beta\alpha^\prime}\tfss_{\alpha\beta}{}^{\beta^\prime}\Sch_{\alpha^\prime\beta^\prime} + \tfss^{\alpha\beta\alpha^\prime}\onabla_\alpha\onabla_\beta H_{\alpha^\prime} + H^{\alpha^\prime}\tfss^{\alpha\beta}{}_{\alpha^\prime}\oSch_{\alpha\beta} \\
   & \quad - \frac{1}{\submfdim-3}\onabla^\alpha\onabla^\beta\tfss^2_{\alpha\beta} + \frac{\submfdim-5}{2(\submfdim-3)(\submfdim-6)}\oDelta\lv\tfss\rv^2 - \frac{1}{(\submfdim-3)(\submfdim-6)}\otrSch\lv\tfss\rv^2 \\
   & \quad + \frac{\submfdim-4}{\submfdim-3}\tfss^2_{\alpha\beta}\oSch^{\alpha\beta} - \frac{\submfdim-3}{\submfdim-2}H^{\alpha^\prime}\tr\tfss^3_{\alpha^\prime} - \frac{\submfdim-1}{\submfdim-2}H^{\alpha^\prime}\tfss^{\alpha\beta}{}_{\alpha^\prime}W_{\alpha\gamma\beta}{}^\gamma - \frac{3}{2}\lv H \rv^2 \lv \tfss \rv^2 \\
   & \quad + \frac{\submfdim}{(\submfdim-1)^2}(\onabla^\beta\tfss_{\beta\alpha\alpha^\prime})(\onabla_\gamma\tfss^{\gamma\alpha\alpha^\prime})
 \end{align*}
 equals the invariant in~\cite{Juhl2022}*{Equation~(8.2)}.  
\end{proposition}

\begin{remark}
 As in \cref{rk:add-Wn-and-mK2}, one can add multiples of $\Wn$ or $\mK_2$ to these invariants without losing the fact that our invariants generalize Juhl's invariants.
 With our choices, we see that $\mJ_1 \equiv (\submfdim-2)\mJ_2$ modulo polynomials in the Weyl tensor and the trace-free second fundamental form, as was observed already by Juhl~\cite{Juhl2022}*{Equation~(1.19)} in the case of hypersurfaces.
\end{remark}

Chalabi et al.\ derived~\cite{ChalabiHerzogOBannonRobinsonSisti2022}*{Equations~(3.2) and~(3.3)} natural submanifold scalars whose integrals are conformal invariants of compact four-dimensional submanifolds in arbitrary codimension.
We find two pointwise conformal invariants for general dimensional submanifolds that, when restricted to dimension four, equal their scalars modulo divergences, and hence explain the conformal invariance of their integrals.

\begin{proposition}
 \label{chalabi1}
 Let $i \colon Y^{\submfdim} \to (X^{\mfdim},g)$ be an immersion with $2 \leq \submfdim < \mfdim$.
 Define the weight $-4$ scalar conformal submanifold invariant
 \begin{align*}
   \mN_1 & := \frac{1}{2}(2\Wm + \Wn)
   - \frac{\submfdim-6}{2}\Bigl( \mK_1 + \mK_2 \Bigr)
   + \frac{\submfdim-6}{2}\Bigl( \tfss^{\gamma\alpha\alpha^\prime}\tfss_\gamma{}^{\beta\beta^\prime}W_{\alpha\beta\alpha^\prime\beta^\prime} \\
   & \quad - W_{\alpha\beta\alpha^\prime}{}^\beta W^{\alpha\gamma\alpha^\prime}{}_\gamma + \frac{1}{2}W_{\alpha\beta\gamma\alpha^\prime}W^{\alpha\beta\gamma\alpha^\prime} - \tfss^2_{\alpha\beta}W^{\alpha\gamma\beta}{}_{\gamma} \\
   & \quad - \tfss^{\alpha\gamma\alpha^\prime}\tfss^{\beta\delta}{}_{\alpha^\prime}W_{\alpha\beta\gamma\delta} + \tfss^{\alpha\beta\alpha^\prime}\tfss_{\alpha\beta}{}^{\beta^\prime}W_{\alpha^\prime\gamma\beta^\prime}{}^\gamma - 2\tfss^{\gamma\alpha\alpha^\prime}\tfss_\gamma{}^{\beta\beta^\prime}W_{\alpha\alpha^\prime\beta\beta^\prime} \Bigr).
 \end{align*}
 If $\submfdim = 4$, then
 \begin{equation*}
  \mN_1 = -\frac{1}{2}\oDelta\lv\tfss\rv^2 + 2\onabla^\beta( \Di^{\alpha\alpha^\prime} \tfss_{\alpha\beta\alpha^\prime} ) + \mJ_1^{\mathrm{CHO}^+} ,
 \end{equation*}
 where
 \begin{multline*}
  \mJ_1^{\mathrm{CHO}^+} := 2\lv\Di\rv^2 + \tfss^{\alpha\beta\alpha^\prime}\nabla_{\alpha^\prime}W_{\alpha\gamma\beta}{}^\gamma + \frac{1}{\mfdim-1}R\lv\tfss\rv^2 - \frac{1}{\mfdim-2}\lv\tfss\rv^2 R_{\alpha^\prime}{}^{\alpha^\prime} \\
   - \frac{2}{\mfdim-2}\tfss^2_{\alpha\beta}R^{\alpha\beta} + \lv H \rv^2\lv\tfss\rv^2- 2H^{\alpha^\prime}\tr\tfss^3_{\alpha^\prime} - 2H^{\alpha^\prime}\tfss^{\alpha\beta}{}_{\alpha^\prime}W_{\alpha\gamma\beta}{}^\gamma
 \end{multline*}
 equals the scalar in~\cite{ChalabiHerzogOBannonRobinsonSisti2022}*{Equation~(3.2)}.
\end{proposition}

\begin{proposition}
 \label{chalabi}
 Let $i \colon Y^{\submfdim} \to (X^{\mfdim},g)$ be an immersion with $2 \leq \submfdim < \mfdim$ and $\submfdim \not\in \{ 3, 6 \}$.
 Define the weight $-4$ scalar conformal submanifold invariant
 \begin{align*}
  \mN_2 & := \frac{\mfdim-4}{(\submfdim-3)(\submfdim-6)}\Wn - \frac{2}{\submfdim-1}\left( W_{\alpha\beta cd}W^{\alpha\beta cd} - W_{\alpha c \beta d}W^{\alpha c \beta d} + W_{c \alpha d}{}^\alpha W^{c\beta d}{}_\beta \right) \\
   & \quad - \frac{2(\mfdim-\submfdim-1)}{(\submfdim-1)(\submfdim-3)}\mK_1 - \frac{2(\mfdim-5\submfdim+11)}{(\submfdim-1)(\submfdim-3)}\mK_2 - \frac{2(\mfdim-3\submfdim+5)}{(\submfdim-1)(\submfdim-3)} W_{\alpha\beta\alpha^\prime}{}^\beta W^{\alpha\gamma\alpha^\prime}{}_\gamma \\
   & \quad + \frac{\mfdim-\submfdim-1}{(\submfdim-1)(\submfdim-3)}W_{\alpha\beta\alpha^\prime\gamma}W^{\alpha\beta\alpha^\prime\gamma} - \frac{2(\mfdim-\submfdim-1)}{(\submfdim-1)(\submfdim-3)}\tfss^{\alpha\gamma\alpha^\prime}\tfss^{\beta\delta}{}_{\alpha^\prime}W_{\alpha\beta\gamma\delta} \\
   & \quad - \frac{2(\mfdim-3\submfdim+5)}{(\submfdim-1)(\submfdim-3)}\tfss^2_{\alpha\beta}W^{\alpha\gamma\beta}{}_\gamma + \frac{2(\mfdim-5\submfdim+11)}{(\submfdim-1)(\submfdim-3)}\tfss^{\gamma\alpha\alpha^\prime}\tfss_\gamma{}^{\beta\beta^\prime}W_{\alpha\beta\alpha^\prime\beta^\prime} \\
   & \quad + \frac{2(\mfdim-3\submfdim+5)}{(\submfdim-1)(\submfdim-3)}\tfss^{\alpha\beta\alpha^\prime}\tfss_{\alpha\beta}{}^{\beta^\prime}W_{\alpha^\prime\gamma\beta^\prime}{}^\gamma - \frac{4(\mfdim-2\submfdim+2)}{(\submfdim-1)(\submfdim-3)}\tfss^{\gamma\alpha\alpha^\prime}\tfss_\gamma{}^{\beta\beta^\prime}W_{\alpha\alpha^\prime\beta\beta^\prime} .
 \end{align*}
 If $\submfdim = 4$, then
 \begin{equation*}
  \mN_2 = \frac{3\mfdim - 10}{6}\oDelta W_{\alpha\beta}{}^{\alpha\beta} - \frac{4(\mfdim-5)}{3}\onabla^\alpha\mC_\alpha + \mJ_2^{\mathrm{CHO}^+} ,
 \end{equation*}
 where
 \begin{equation}
  \label{eqn:CHO2}
  \begin{split}
   \mJ_2^{\mathrm{CHO}^+} & := \frac{1}{3}\nabla^{\alpha^\prime}\nabla_{\alpha^\prime}W_{\alpha\beta}{}^{\alpha\beta} + \frac{\mfdim-10}{3}H^{\alpha^\prime}\nabla_{\alpha^\prime}W_{\alpha\beta}{}^{\alpha\beta} - \frac{\mfdim-4}{\mfdim-1}RW_{\alpha\beta}{}^{\alpha\beta} \\
    & \quad + \frac{\mfdim-4}{\mfdim-2}R_{\alpha^\prime}{}^{\alpha^\prime}W_{\alpha\beta}{}^{\alpha\beta} + \frac{4(\mfdim-5)}{3(\mfdim-2)}R^{\alpha\beta}W_{\alpha\gamma\beta}{}^\gamma - \frac{4}{3}W_{\alpha\beta\alpha^\prime}{}^\beta\onabla^\alpha H^{\alpha^\prime} \\
    & \quad - \frac{2(\mfdim-5)}{3}\tfss^{\alpha\beta\alpha^\prime}\nabla_{\alpha^\prime}W_{\alpha\gamma\beta}{}^\gamma + \frac{8(\mfdim-5)}{3}H^{\alpha^\prime}\tfss^{\alpha\beta}{}_{\alpha^\prime}W_{\alpha\gamma\beta}{}^\gamma \\
    & \quad - \frac{4(\mfdim+1)}{3}\Di^{\alpha\alpha^\prime}W_{\alpha\beta\alpha^\prime}{}^\beta - \frac{5(\mfdim -4)}{3}\lv H\rv^2 W_{\alpha\beta}{}^{\alpha\beta}
  \end{split}
 \end{equation}
 equals the scalar in~\cite{ChalabiHerzogOBannonRobinsonSisti2022}*{Equation~(3.3)}.
 \end{proposition}
 
 \begin{remark}
  Equation~\eqref{eqn:CHO2} corrects two typos in~\cite{ChalabiHerzogOBannonRobinsonSisti2022}*{Equation~(3.3)}.
 \end{remark}

\section{Proof of \cref{submanifold-alexakis} in dimensions two and four}
\label{sec:alexakis}

We conclude this paper by proving \cref{submanifold-alexakis} in dimensions two and four.
First we identify spanning sets for the spaces of natural submanifold scalars of weight $-2$ and $-4$ modulo tangential divergences and scalar conformal submanifold invariants.
Then we compute the conformal linearization of the integral of an arbitrary linear combination of elements of our spanning sets in order to prove \cref{submanifold-alexakis}.

We begin with some terminology.
Let $r,s,t,u \in \bN_0$ and let $I$ be a natural submanifold tensor of weight~$w$ taking values in $(T^\ast Y)^{\otimes r} \otimes (TY)^{\otimes s} \otimes (N^\ast Y)^{\otimes t} \otimes (NY)^{\otimes u}$.
The \defn{tensor weight} of $I$ is defined to be $w - r - t + s + u$.
In particular,
\begin{enumerate}
 \item the tensor weight and the weight coincide for scalars;
 \item contraction between an upper and a lower index preserves the tensor weight;
 \item each of $g_{\alpha\beta}$, $g^{\alpha\beta}$, $g_{\alpha^\prime\beta^\prime}$, and $g^{\alpha^\prime\beta^\prime}$ has tensor weight $0$, so we can raise and lower indices without changing the tensor weight;
 \item $\tfss_{\alpha\beta\alpha^\prime}$ and $H_{\alpha^\prime}$ have tensor weight $-1$;
 \item the various projections of $R_{abcd}$ have tensor weight $-2$; and
 \item if $T$ has tensor weight $w$, then $\onabla T$ has tensor weight $w-1$.
\end{enumerate}

\subsection{The two-dimensional case}
\label{subsec:alexakis/2}

The proof of the two-dimensional case of \cref{submanifold-alexakis} is straightforward.

\begin{proof}[Proof of \cref{low-dimension-submanifold-alexakis} when $\submfdim=2$]
 Let $I$ be a natural scalar on $2$-dimensional submanifolds $Y$ of $\mfdim$-dimensional Riemannian manifolds $(X,g)$ whose integral over compact $Y$ is invariant under conformal rescaling of $g$.
 Then $I$ has weight $-2$.
 Evaluation of the tensor weight of tensors of the form~\eqref{eqn:partial contraction} shows that $I$ is a linear combination of complete contractions of $L \otimes L$, $\onabla L$, and projections of $\Rm$.
 The tensor $\onabla L$ has exactly one normal index, and hence cannot be completely contracted, while a projection of $\Rm$ can only be completely contracted if it has an even number of tangential indices.
 Since the Weyl tensor is trace-free, we see that
 \begin{equation}
  \label{eqn:reduce-weyl}
  W_{a\alpha^\prime b}{}^{\alpha^\prime} = -W_{a\alpha b}{}^\alpha ,
 \end{equation}
 and thus $W_{\alpha\beta}{}^{\alpha\beta} = -W_{\alpha\alpha^\prime}{}^{\alpha\alpha^\prime} = W_{\alpha^\prime\beta^\prime}{}^{\alpha^\prime\beta^\prime}$.
 Equation~\eqref{eqn:gcJ} now implies that
 \begin{equation*}
  I \in \vspan \{ W_{\alpha\beta}{}^{\alpha\beta} , \Sch_{\alpha^\prime}{}^{\alpha^\prime} , \otrSch, \lv H \rv^2 , \lv\tfss\rv^2 \} .
 \end{equation*}
 
 Clearly $W_{\alpha\beta}{}^{\alpha\beta}$ and $\lv \tfss \rv^2$ are conformally invariant of weight $-2$.
 In dimension two, $\oPf = \otrSch$.
 Therefore it suffices to show that if $I = a_1\lv H \rv^2 + a_2\Sch_{\alpha^\prime}{}^{\alpha^\prime}$, then $a_1=a_2=0$.
 The assumption of conformal invariance implies that
 \begin{equation*}
  0 = \int_Y (I \darea)^\bullet = -\int_Y \left( 2a_1H^{\alpha^\prime}\nabla_{\alpha^\prime}\Upsilon + a_2\nabla^{\alpha^\prime}\nabla_{\alpha^\prime}\Upsilon \right) \darea
 \end{equation*}
 for all $\Upsilon \in C^\infty(Y)$.
 Choosing first $\Upsilon$ such that $\nabla_{\alpha^\prime}\Upsilon=0$ along $Y$ implies that $a_2=0$;
 then taking a non-minimal immersion implies that $a_1=0$.
\end{proof}

\subsection{The four-dimensional case}
\label{subsec:alexakis/4}

We begin by finding a spanning set for the space of natural submanifold scalars of weight $-4$ modulo tangential divergences and scalar conformal submanifold invariants.
As a means to organize this set, we say that $I^\bullet$ \defn{depends only on the transverse $j$-jet of the conformal factor} if it vanishes for all $\Upsilon \in C^\infty(Y)$ such that $\nabla_{\alpha_1^\prime} \dotsm \nabla_{\alpha_k^\prime} \Upsilon = 0$ along $Y$ for all nonnegative integers $k \leq j$.

\begin{lemma}
 \label{lem:4d-invariants}
 Let $\mfdim > 4$.
 Denote by $\sI_4$ the space of natural submanifold scalars of weight $-4$ for immersions $i \colon Y^4 \to (X^{\mfdim},g)$.
 Let $\sI_{4,\divsymb} \subseteq \sI_4$ be the subspace of natural tangential divergences and let $\sI_{4,c} \subseteq \sI_4$ be the subspace of conformal submanifold invariants.
 Set
 \begin{align*}
  \sI_4^0 & := \{ \lp \Fi, \oSch \rp , W_{\alpha\beta\alpha^\prime}{}^\beta\Di^{\alpha\alpha^\prime} , \lv\tfss\rv^2\otrSch , \lp \tfss^2 , \oSch \rp , \otrSch^2 , \lv \oSch \rv^2 , \trFi\otrSch , \lv \Di \rv^2 \} , \\
  \sI_4^1 & := \{ H^{\alpha^\prime}\oDelta H_{\alpha^\prime} , H^{\alpha^\prime}\onabla^\alpha\Di_{\alpha\alpha^\prime} , H^{\alpha^\prime}\onabla^\alpha W_{\alpha\beta\alpha^\prime}{}^\beta , \lv H \rv^4 , \lv H \rv^2 \lv \tfss \rv^2 ,  \\
   & \qquad H^{\alpha^\prime}\tfss_{\alpha\beta\alpha^\prime}\tfss^{\alpha\beta\beta^\prime}H_{\beta^\prime} , \lv H \rv^2\otrSch , \trFi\lv H \rv^2, H^{\alpha^\prime} \tr \tfss^3_{\alpha^\prime} , H^{\alpha^\prime} \tfss_{\alpha\beta\alpha^\prime} \Fi^{\alpha\beta} , \\
   & \qquad H^{\alpha^\prime} \tfss_{\alpha\beta\alpha^\prime} \oSch^{\alpha\beta} , H^{\alpha^\prime} H^{\beta^\prime} W_{\alpha^\prime\alpha\beta^\prime}{}^\alpha , H^{\alpha^\prime} \tfss^{\alpha\beta\beta^\prime} W_{\alpha\alpha^\prime\beta\beta^\prime}, H^{\alpha^\prime} \mC_{\alpha^\prime} \} , \\
  \sI_4^2 & := \{ \trFi\Sch_{\alpha^\prime}{}^{\alpha^\prime} , \Sch^{\alpha^\prime\beta^\prime}W_{\alpha^\prime\alpha\beta^\prime}{}^\alpha , \lv H\rv^2\Sch_{\alpha^\prime}{}^{\alpha^\prime}, \lv\tfss\rv^2 \Sch_{\alpha^\prime}{}^{\alpha^\prime} , H^{\alpha^\prime} H^{\beta^\prime} \Sch_{\alpha^\prime\beta^\prime} , \\
   & \qquad \tfss_{\alpha\beta}{}^{\alpha^\prime} \tfss^{\alpha\beta\beta^\prime}\Sch_{\alpha^\prime\beta^\prime}, \otrSch \Sch_{\alpha^\prime}{}^{\alpha^\prime} , \Sch_{\alpha^\prime}{}^{\alpha^\prime} \Sch_{\beta^\prime}{}^{\beta^\prime} , \Sch_{\alpha^\prime\beta^\prime}\Sch^{\alpha^\prime\beta^\prime} \} , \\
  \sI_4^3 & := \{ H^{\alpha^\prime} \nabla_{\alpha^\prime} \trSch \} , \\
  \sI_4^4 & := \{ \Delta \trSch \} ,
 \end{align*}
 where we recall our conventions~\eqref{eqn:conventions}.
 Then
 \begin{equation*}
  \sI_4 = \sI_{4,\divsymb} + \sI_{4,c} + \vspan \bigcup_{j=0}^4 \sI_4^j .
 \end{equation*}
 Moreover, for each integer $0 \leq j \leq 4$, the conformal linearizations of the elements of $\sI_4^j$ depend only on the transverse $j$-jet of the conformal factor.
\end{lemma}

\begin{proof}
 It is clear from the discussion of Subsection~\ref{subsec:conformal-submanifold-invariants} that if $I \in \sI_4^j$, then $I^\bullet$ depends only on the transverse $j$-jet of the conformal factor.
 
 Evaluation of the tensor weight of tensors of the form~\eqref{eqn:partial contraction} shows that if $I$ is a natural submanifold scalar of weight $-4$, then
 \begin{equation*}
  I \in \Contr( \nabla^2\Rm , L \otimes \onabla^2 L , \onabla L \otimes \onabla L ,  L \otimes \nabla\Rm, \onabla L \otimes \Rm, L^{\otimes 4} , L^{\otimes 2} \otimes \Rm, \Rm^{\otimes 2} ) ,
 \end{equation*}
 where $\Contr( \, \dotsm )$ denotes the space of linear combinations of complete contractions of the various projections of its arguments.
 We consider such contractions in decreasing order of the total number of derivatives of $L$ and $\Rm$ taken.
 
 \underline{Step 1: Analyze $\Contr( \nabla^2\Rm , L \otimes \onabla^2L ,\onabla L \otimes \onabla L)$}.
 Given two natural submanifold scalars $A$ and $B$ of weight $-4$, write $A \sim B$ if
 \begin{equation*}
  A - B \in \sI_{4,\divsymb} + \sI_{4,c} + \Contr( L \otimes \nabla \Rm, \onabla L \otimes \Rm, L^{\otimes 4} , L^{\otimes 2} \otimes \Rm, \Rm^{\otimes 2} ) .
 \end{equation*}
 
 We first consider $\Contr( \nabla^2\Rm )$.
 The definition~\eqref{eqn:second-fundamental-form-dual} of the second fundamental form implies that if $V_\alpha$ is a partial contraction of a projection of $\nabla\Rm$ with values in $T^\ast Y$, then $\nabla^\alpha V_\alpha \sim \onabla^\alpha V_\alpha \sim 0$.
 By combining this with the Ricci identity, we see that if $I$ is a complete contraction of a projection of $\nabla^2\Rm$ with at least one tangential index on a covariant derivative, then $I \sim 0$.
 Therefore
 \begin{multline*}
  \Contr( \nabla^2\Rm ) \subseteq \Contr( \nabla_{\alpha^\prime}\nabla_{\beta^\prime} R_{abcd} ) + \sI_{4,\divsymb} + \sI_{4,c} \\
   + \Contr(L \otimes \nabla\Rm, \onabla L \otimes \Rm, L^{\otimes 4} , L^{\otimes 2} \otimes \Rm, \Rm^{\otimes 2}) .
 \end{multline*}
 \Cref{Wm} implies that $\mB_\alpha{}^\alpha \sim (\mfdim-4)\lv\Di\rv^2$.
 Therefore
 \begin{equation}
  \label{eqn:nabla2Rm-cases}
  \begin{split}
   \nabla^{\alpha^\prime}\nabla_{\alpha^\prime}\trSch & \sim \Delta\trSch , \\
   \nabla^{\alpha^\prime}C_{\alpha^\prime\alpha}{}^\alpha & \sim \mB_\alpha{}^\alpha \sim (\mfdim-4)\lv\Di\rv^2 .
  \end{split}
 \end{equation}
 Equations~\eqref{eqn:nabla2Rm-cases} and direct computation yield
 \begin{align*}
  \nabla^{\alpha^\prime}\nabla_{\alpha^\prime}\Sch_\alpha{}^\alpha & = \nabla^{\alpha^\prime}C_{\alpha^\prime\alpha}{}^\alpha + \nabla^{\alpha^\prime}\nabla_\alpha\Sch^\alpha{}_{\alpha^\prime} \sim (\mfdim-4)\lv\Di\rv^2 , \\
  \nabla^{\alpha^\prime}\nabla_{\alpha^\prime}\Sch_{\beta^\prime}{}^{\beta^\prime} & = \nabla^{\alpha^\prime}\nabla_{\alpha^\prime}\trSch - \nabla^{\alpha^\prime}\nabla_{\alpha^\prime}\Sch_{\alpha}{}^{\alpha} \sim \Delta\trSch - (\mfdim-4)\lv\Di\rv^2 , \\
  \nabla^{\alpha^\prime}\nabla^{\beta^\prime}\Sch_{\alpha^\prime\beta^\prime} & = \nabla^{\alpha^\prime}\nabla_{\alpha^\prime}\trSch - \nabla^{\alpha^\prime}\nabla^\alpha\Sch_{\alpha\alpha^\prime} \sim \Delta\trSch .
 \end{align*}
 Similarly, Equations~\eqref{eqn:weyl-bianchi}, \eqref{eqn:W-to-C}, and~\eqref{eqn:nabla2Rm-cases} yield
 \begin{align*}
  \nabla^{\alpha^\prime}\nabla_{\alpha^\prime} W_{\alpha\beta}{}^{\alpha\beta} & = 2\nabla^{\alpha^\prime}\nabla_\alpha W_{\alpha^\prime\beta}{}^{\alpha\beta} + 6\nabla^{\alpha^\prime} C_{\alpha\alpha^\prime}{}^\alpha \sim -6(\mfdim-4)\lv\Di\rv^2 , \\
  \nabla^{\alpha^\prime}\nabla^{\beta^\prime} W_{\alpha^\prime\alpha\beta^\prime}{}^\alpha & \sim (\mfdim-3)\nabla^{\alpha^\prime} C_{\alpha^\prime\alpha}{}^\alpha \sim (\mfdim-3)(\mfdim-4)\lv\Di\rv^2 .
 \end{align*}
 We conclude from Equation~\eqref{eqn:reduce-weyl} that
 \begin{multline}
  \label{eqn:nabla2Rm-contr}
  \Contr(\nabla^2\Rm) \subseteq \vspan \{ \Delta\trSch , \lv\Di\rv^2 \} + \sI_{4,\divsymb} + \sI_{4,c} \\
   + \Contr( L \otimes \nabla\Rm , \onabla L \otimes \Rm , L^{\otimes 4} , L^{\otimes 2} \otimes \Rm , \Rm^{\otimes 2} ).
 \end{multline}
 
 We now consider $\Contr( L \otimes \onabla^2L , \onabla L \otimes \onabla L)$.
 Since $\Contr( \onabla( L \otimes \onabla L)) \subseteq \sI_{4,\divsymb}$, we see that
 \begin{equation*}
  \Contr( L \otimes \onabla^2L , \onabla L \otimes \onabla L) \subseteq \Contr( L \otimes \onabla^2L ) + \sI_{4,\divsymb} .
 \end{equation*}
 The Ricci identity and Gauss equation imply that any element of $\Contr( L \otimes \onabla^2L )$ is a linear combination of
 \begin{equation*}
  H^{\alpha^\prime}\oDelta H_{\alpha^\prime} , \tfss^{\alpha\beta\alpha^\prime}\oDelta\tfss_{\alpha\beta\alpha^\prime} , H^{\alpha^\prime}\onabla^\alpha\onabla^\beta\tfss_{\alpha\beta\alpha^\prime}, \tfss^{\alpha\beta\alpha^\prime}\onabla_\alpha\onabla_\beta H_{\alpha^\prime}, \tfss^{\alpha\beta\alpha^\prime}\onabla_{\alpha}\onabla^\gamma\tfss_{\gamma\beta\alpha^\prime} ,
 \end{equation*}
 and elements of $\Contr( L^{\otimes 4}, L^{\otimes 2} \otimes \Rm)$.
 Equations~\eqref{eqn:defn-Di} and~\eqref{eqn:gcD} imply that
 \begin{equation}
  \label{eqn:simplify-Lnabla2L}
  \begin{aligned}
   \tfss^{\alpha\beta\alpha^\prime}\onabla_\alpha\onabla_\beta H_{\alpha^\prime} & \sim H^{\alpha^\prime}\onabla^\alpha\onabla^\beta\tfss_{\alpha\beta\alpha^\prime} \sim -3H^{\alpha^\prime}\onabla^\alpha\Di_{\alpha\alpha^\prime} \sim 3H^{\alpha^\prime}\oDelta H_{\alpha^\prime} , \\
   \tfss^{\alpha\beta\alpha^\prime}\onabla_{\alpha}\onabla^\gamma\tfss_{\gamma\beta\alpha^\prime} & \sim 3\tfss^{\alpha\beta\alpha^\prime}\onabla_\alpha\onabla_\beta H_{\alpha^\prime} \sim 9H^{\alpha^\prime}\oDelta H_{\alpha^\prime} .
  \end{aligned}
 \end{equation}
 Therefore
 \begin{multline*}
  \Contr( L \otimes \onabla^2L , \onabla L \otimes \onabla L ) \subseteq \vspan \{ H^{\alpha^\prime}\oDelta H_{\alpha^\prime}, \tfss^{\alpha\beta\alpha^\prime}\oDelta\tfss_{\alpha\beta\alpha^\prime} \} \\
   + \sI_{4,\divsymb} + \sI_{4,c} + \Contr(L \otimes \nabla\Rm, \onabla L \otimes \Rm, L^{\otimes 4} , L^{\otimes 2} \otimes \Rm, \Rm^{\otimes 2}) .
 \end{multline*}
 Similarly, Equation~\eqref{eqn:defn-Di} implies that $\lv\Di\rv^2 \sim -H^{\alpha^\prime}\oDelta H_{\alpha^\prime}$, and Equations~\eqref{eqn:defn-Di}, \eqref{eqn:gcdL}, and~\eqref{eqn:simplify-Lnabla2L} imply that
 \begin{align*}
  \tfss^{\alpha\beta\alpha^\prime}\oDelta\tfss_{\alpha\beta\alpha^\prime} & = \tfss^{\alpha\beta\alpha^\prime}\onabla^\gamma(\onabla_\alpha\tfss_{\gamma\beta\alpha^\prime} + 2\onabla_{[\gamma}\tfss_{\alpha]\beta\alpha^\prime}) \\
  & \sim \tfss^{\alpha\beta\alpha^\prime}(\onabla_\alpha\onabla^\gamma\tfss_{\gamma\beta\alpha^\prime} - \onabla_\alpha\onabla_\beta H_{\alpha^\prime}) \\
  & \sim 6H^{\alpha^\prime}\oDelta H_{\alpha^\prime} .
 \end{align*}
 We conclude from Statement~\eqref{eqn:nabla2Rm-contr} that
 \begin{multline}
  \label{eqn:second-order}
   \Contr( \nabla^2\Rm , L \otimes \onabla^2L , \onabla L \otimes \onabla L) \subseteq \vspan \{ \Delta\trSch , H^{\alpha^\prime} \oDelta H_{\alpha^\prime} \} + \sI_{4,\divsymb} \\
    + \sI_{4,c} + \Contr( L \otimes \nabla\Rm, \onabla L \otimes \Rm , L^{\otimes 4} , L^{\otimes 2} \otimes \Rm , \Rm^{\otimes 2}) .
 \end{multline}

 \underline{Step 2: Analyze $\Contr( L \otimes \nabla\Rm , \onabla L \otimes\Rm )$}.
 Given two natural submanifold scalars $A$ and $B$ of weight $-4$, write $A \approx B$ if
 \begin{equation*}
  A - B \in \sI_{4,\divsymb} + \sI_{4,c} + \Contr( L^{\otimes 4} , L^{\otimes 2} \otimes \Rm , \Rm^{\otimes 2}) .
 \end{equation*}
 
 Since $\Contr( \onabla( L \otimes \Rm) ) \subseteq \sI_{4,\divsymb}$, we deduce from Equation~\eqref{eqn:second-fundamental-form-dual} that
 \begin{equation*}
  \Contr( L \otimes \nabla\Rm , \onabla L \otimes \Rm) \subseteq \Contr( L \otimes \nabla\Rm) + \sI_{4,\divsymb} + \Contr( L^{\otimes 2} \otimes \Rm) .
 \end{equation*}
 In particular, it suffices to analyze $\Contr( L \otimes \nabla\Rm )$.
 
 First, Equations~\eqref{eqn:weyl-bianchi} and~\eqref{eqn:second-fundamental-form-dual}--\eqref{eqn:defn-Di} imply that
 \begin{align*}
  \nabla_\alpha\Sch_{\beta\alpha^\prime} & \equiv \onabla_\alpha\onabla_\beta H_{\alpha^\prime} + \onabla_\alpha\Di_{\beta\alpha^\prime} , \\
  \nabla_{\alpha^\prime}\Sch_{\alpha\beta} & \equiv -\mC_{\alpha\alpha^\prime\beta} + \onabla_\alpha\onabla_\beta H_{\alpha^\prime} + \onabla_\alpha\Di_{\beta\alpha^\prime} , \\
  \nabla^\gamma W_{\alpha\gamma\beta\alpha^\prime} & \equiv \onabla^\gamma W_{\alpha\gamma\beta\alpha^\prime} , \\
  \nabla_{\alpha^\prime} W_{\alpha\gamma\beta}{}^\gamma & \equiv \onabla_\alpha W_{\alpha^\prime\gamma\beta}{}^\gamma + \onabla_\gamma W_{\alpha\alpha^\prime\beta}{}^\gamma + 2\mC_{\alpha\alpha^\prime\beta} + \mC_{\alpha^\prime} g_{\alpha\beta} ,
 \end{align*}
 where $\equiv$ denotes equivalence modulo partial contractions of $L \mathop{\otimes} \Rm \mathop{\otimes} g$.
 Recalling Equation~\eqref{eqn:gcD}, we deduce that any complete contraction of a projection of $L_{\alpha\beta\alpha^\prime}\nabla^{\alpha^\prime} R_{abcd}$ is a linear combination of
 \begin{align*}
  H^{\alpha^\prime}\nabla_{\alpha^\prime}\trSch , \\
  H^{\alpha^\prime}\nabla_{\alpha^\prime}\Sch_\alpha{}^\alpha & \approx -H^{\alpha^\prime} \mC_{\alpha^\prime} + H^{\alpha^\prime}\oDelta H_{\alpha^\prime} + H^{\alpha^\prime}\onabla^\alpha\Di_{\alpha\alpha^\prime} , \\
  H^{\alpha^\prime}\nabla_{\alpha^\prime} W_{\alpha\beta}{}^{\alpha\beta} & \approx 2H^{\alpha^\prime}\onabla^\alpha W_{\alpha\beta\alpha^\prime}{}^\beta + 6H^{\alpha^\prime} \mC_{\alpha^\prime} , \\
  \tfss^{\alpha\beta\alpha^\prime}\nabla_{\alpha^\prime}\Sch_{\alpha\beta} & \approx -\tfss^{\alpha\beta\alpha^\prime}\mC_{\alpha\alpha^\prime\beta} + 3\lv\Di\rv^2 + \Di^{\alpha\alpha^\prime}W_{\alpha\beta\alpha^\prime}{}^\beta \\
   & \quad - 3H^{\alpha^\prime}\onabla^\alpha\Di_{\alpha\alpha^\prime} - H^{\alpha^\prime}\onabla^\alpha W_{\alpha\beta\alpha^\prime}{}^\beta , \\
  \tfss^{\alpha\beta\alpha^\prime}\nabla_{\alpha^\prime} W_{\alpha\gamma\beta}{}^\gamma & \approx 3\Di^{\alpha\alpha^\prime}W_{\alpha\beta\alpha^\prime}{}^\beta + \tfss^{\alpha\beta\alpha^\prime}\onabla^\gamma W_{\alpha\alpha^\prime\beta\gamma} + 2\tfss^{\alpha\beta\alpha^\prime}\mC_{\alpha\alpha^\prime\beta} .
 \end{align*}
 Similarly, any complete contraction of a projection of $L_{\alpha\beta\alpha^\prime}\nabla_aR_{bcd}{}^{\alpha^\prime}$ is a linear combination of
 \begin{align*}
  H^{\alpha^\prime}\nabla^\alpha W_{\alpha\beta\alpha^\prime}{}^\beta & \approx H^{\alpha^\prime}\onabla^\alpha W_{\alpha\beta\alpha^\prime}{}^\beta , \\
  H^{\alpha^\prime}\nabla^\alpha\Sch_{\alpha\alpha^\prime} & \approx H^{\alpha^\prime}\oDelta H_{\alpha^\prime} + H^{\alpha^\prime}\onabla^\alpha\Di_{\alpha\alpha^\prime} , \\
  \tfss^{\alpha\beta\alpha^\prime}\nabla_\alpha\Sch_{\beta\alpha^\prime} & \approx 3\lv\Di\rv^2 + \Di^{\alpha\alpha^\prime}W_{\alpha\beta\alpha^\prime}{}^\beta - 3H^{\alpha^\prime}\onabla^\alpha\Di_{\alpha\alpha^\prime} - H^{\alpha^\prime}\onabla^\alpha W_{\alpha\beta\alpha^\prime}{}^\beta , \\
  \tfss^{\alpha\beta\alpha^\prime}\nabla_\alpha W_{\beta\gamma\alpha^\prime}{}^\gamma & \approx 3\Di^{\alpha\alpha^\prime}W_{\alpha\beta\alpha^\prime}{}^\beta , \\
  \tfss^{\alpha\beta\alpha^\prime}\nabla^\gamma W_{\alpha\gamma\beta\alpha^\prime} & \approx \tfss^{\alpha\beta\alpha^\prime}\onabla^\gamma W_{\alpha\gamma\beta\alpha^\prime} , 
 \end{align*}
 and
 \begin{align*}
  H^{\alpha^\prime}\nabla^{\beta^\prime} W_{\alpha^\prime\alpha\beta^\prime}{}^\alpha & \approx -(\mfdim-3)H^{\alpha^\prime} \mC_{\alpha^\prime} - H^{\alpha^\prime}\nabla^\alpha W_{\alpha\beta\alpha^\prime}{}^\beta , \\
  H^{\alpha^\prime}\nabla^{\beta^\prime}\Sch_{\alpha^\prime\beta^\prime} & = H^{\alpha^\prime}\nabla_{\alpha^\prime}\trSch - H^{\alpha^\prime} \nabla^\alpha \Sch_{\alpha\alpha^\prime} , \\
  \tfss^{\alpha\beta\alpha^\prime}\nabla^{\beta^\prime} W_{\alpha\alpha^\prime\beta\beta^\prime} & \approx -(\mfdim-3)\tfss^{\alpha\beta\alpha^\prime}\mC_{\alpha\alpha^\prime\beta} - \tfss^{\alpha\beta\alpha^\prime}\nabla^\gamma W_{\alpha\gamma\beta\alpha^\prime} .
 \end{align*}
 Combining these computations with \cref{rk:mK}, \cref{Wm,Wn}, and Statement~\eqref{eqn:second-order} yields
 \begin{equation}
  \label{eqn:first-order}
  \begin{aligned}
   \MoveEqLeft \Contr( \nabla^2\Rm , L \otimes \onabla^2L , \onabla L \otimes \onabla L , L \otimes \nabla\Rm , \onabla L \otimes \Rm) \\
   & \subseteq \vspan \{ \Delta\trSch , H^{\alpha^\prime}\oDelta H_{\alpha^\prime} , H^{\alpha^\prime}\nabla_{\alpha^\prime}\trSch , H^{\alpha^\prime}\onabla^\alpha\Di_{\alpha\alpha^\prime} , \lv \Di \rv^2 , \\
   & \qquad\qquad H^{\alpha^\prime}\mC_{\alpha^\prime} , H^{\alpha^\prime}\onabla^\alpha W_{\alpha\beta\alpha^\prime}{}^\beta , \Di^{\alpha\alpha^\prime}W_{\alpha\beta\alpha^\prime}{}^\beta \} \\
   & \quad + \sI_{4,\divsymb} + \sI_{4,c} + \Contr( L^{\otimes 4} , L^{\otimes 2} \otimes \Rm , \Rm^{\otimes 2}) .
  \end{aligned}
 \end{equation}
 
 \underline{Step 3: Analyze $\Contr( L^{\otimes 4}, L^{\otimes 2} \otimes \Rm , \Rm^{\otimes 2} )$}.
 It is clear that
 \begin{equation}
  \label{eqn:obvious-invariants}
  \Contr( \tfss^{\otimes 4} , \tfss^{\otimes 2} \otimes W , W^{\otimes 2}) \subseteq \sI_{4,c} .
 \end{equation}
 Denote by
 \begin{equation*}
  \sI_{(0)} := \Contr( L^{\otimes 4} , L^{\otimes 2} \otimes \Rm , \Rm^{\otimes 2} ) \cap \bigcup_{j=0}^4 \sI_4^j
 \end{equation*}
 the set of undifferentiated scalars in $\bigcup \sI_4^j$.
 We will show that
 \begin{equation}
  \label{eqn:step3-claim}
  \Contr( L^{\otimes 4} , L^{\otimes 2} \otimes \Rm , \Rm^{\otimes 2} ) \subseteq \sI_{4,\divsymb} + \sI_{4,c} + \vspan \sI_{(0)}  .
 \end{equation}
 This together with Statement~\eqref{eqn:first-order} then completes the proof.
 
 We first show that
 \begin{equation}
  \label{eqn:step3-rm2-claim}
  \begin{aligned}
   \MoveEqLeft[0] \Contr( \Rm^{\otimes 2} ) \subseteq \vspan \Bigl\{ W_{\alpha^\prime\alpha\beta^\prime}{}^\alpha \Sch^{\alpha^\prime\beta^\prime} , \Sch_{\alpha^\prime}{}^{\alpha^\prime}\Sch_{\beta^\prime}{}^{\beta^\prime} , \Sch^{\alpha^\prime\beta^\prime}\Sch_{\alpha^\prime\beta^\prime} , \trFi\Sch_{\alpha^\prime}{}^{\alpha^\prime} , \otrSch \Sch_{\alpha^\prime}{}^{\alpha^\prime} , \\
    & \qquad\qquad \Di^{\alpha\alpha^\prime}W_{\alpha\beta\alpha^\prime}{}^{\beta} , H^{\alpha^\prime}\onabla^\alpha W_{\alpha\beta\alpha^\prime}{}^{\beta} , \lv \Di \rv^2 , H^{\alpha^\prime}\onabla^\alpha \Di_{\alpha\alpha^\prime} , H^{\alpha^\prime}\oDelta H_{\alpha^\prime} , \\
    & \qquad\qquad \otrSch^2 , \lv\oSch\rv^2 , \trFi\otrSch , \lp \Fi, \oSch \rp \Bigr\}  + \Contr( L^{\otimes 4} , L^{\otimes 2} \otimes \Rm ) + \sI_{4,\divsymb} + \sI_{4,c} .
  \end{aligned}
 \end{equation}
 To that end, for each nonnegative integer $j$, let $\Contr_j(\Rm^{\otimes 2})$ be the set of elements of $\Contr(\Rm^{\otimes 2})$ that can be expressed as a linear combination of complete contractions of projections of $W \mathop{\otimes} \Sch$ and $\Sch^{\otimes 2}$ with at least $j$ tangential contractions.
 Statement~\eqref{eqn:obvious-invariants} and the decomposition~\eqref{eqn:defn-weyl} of $\Rm$ imply that
 \begin{equation}
  \label{eqn:step3-prereduce}
  \Contr( \Rm^{\otimes 2} ) \subseteq \Contr_0( \Rm^{\otimes 2} ) + \sI_{4,c} ,
 \end{equation}
 while an index count implies that $\Contr_j(\Rm^{\otimes 2}) = \{ 0 \}$ for all $j > 3$.

 Using Equation~\eqref{eqn:reduce-weyl} yields
 \begin{equation*}
  \Contr_0(\Rm^{\otimes 2}) \subseteq \Contr( \Sch_{\alpha^\prime\beta^\prime}\Sch_{\gamma^\prime\delta^\prime} ) + \Contr_1( \Rm^{\otimes 2} ) .
 \end{equation*}
 Therefore
 \begin{equation}
  \label{eqn:step3-reduce0}
  \Contr_0(\Rm^{\otimes 2}) \subseteq \vspan \{ \Sch_{\alpha^\prime}{}^{\alpha^\prime}\Sch_{\beta^\prime}{}^{\beta^\prime} , \Sch^{\alpha^\prime\beta^\prime}\Sch_{\alpha^\prime\beta^\prime} \} + \Contr_1( \Rm^{\otimes 2} ) .
 \end{equation}
 
 Using Equation~\eqref{eqn:reduce-weyl} again yields
 \begin{equation*}
  \Contr_1( \Rm^{\otimes 2}) \subseteq \Contr( W_{\alpha\alpha^\prime\beta\beta^\prime}\Sch_{\gamma^\prime\delta^\prime} , \Sch_{\alpha\beta}\Sch_{\alpha^\prime\beta^\prime}, \Sch_{\alpha\alpha^\prime}\Sch_{\beta\beta^\prime} ) + \Contr_2(\Rm^{\otimes 2} ) .
 \end{equation*}
 Equation~\eqref{eqn:defn-Fialkow} implies that the Fialkow tensor $\Fi$ can be written as a linear combination of partial contractions of $L^2_{\alpha\beta}g_{\gamma\delta}$ and $W_{\alpha\beta\gamma\delta}g_{\epsilon\zeta}$.
 Combining this observation with Equations~\eqref{eqn:defn-Di}, \eqref{eqn:gcP}, and~\eqref{eqn:reduce-weyl} yields
 \begin{multline}
  \label{eqn:step3-reduce1}
   \Contr_1( \Rm^{\otimes 2}) \subseteq \Contr_2(\Rm^{\otimes 2}) + \Contr(L^{\otimes 2} \otimes \Rm) + \sI_{4,\divsymb} \\
    + \vspan \{ W_{\alpha^\prime\alpha\beta^\prime}{}^\alpha\Sch^{\alpha^\prime\beta^\prime} , \otrSch \Sch_{\alpha^\prime}{}^{\alpha^\prime} , \lv \Di \rv^2 ,H^{\alpha^\prime}\onabla^{\alpha}\Di_{\alpha\alpha^\prime} , H^{\alpha^\prime}\oDelta H_{\alpha^\prime} \} .
 \end{multline}
 
 Applying Equation~\eqref{eqn:reduce-weyl} yet again yields
 \begin{equation*}
  \Contr_2( \Rm^{\otimes 2}) \subseteq \Contr( W_{\alpha\beta\gamma\alpha^\prime}\Sch_{\delta\beta^\prime} , W_{\alpha\beta\gamma\delta}\Sch_{\alpha^\prime\beta^\prime} , \Sch_{\alpha\beta}\Sch_{\gamma\delta} ) + \Contr_3(\Rm^{\otimes 2}) .
 \end{equation*}
 Similar to the previous paragraph, Equations~\eqref{eqn:defn-Di}, \eqref{eqn:gcP}, and~\eqref{eqn:defn-Fialkow} yield
 \begin{multline}
  \label{eqn:step3-reduce2}
   \Contr_2( \Rm^{\otimes 2} ) \subseteq \vspan \{ \Di^{\alpha\alpha^\prime}W_{\alpha\beta\alpha^\prime}{}^{\beta} , H^{\alpha^\prime}\onabla^\alpha W_{\alpha\beta\alpha^\prime}{}^\beta , \trFi \Sch_{\alpha^\prime}{}^{\alpha^\prime} , \otrSch^2 , \lv \oSch \rv^2 \} \\
    + \Contr_3(\Rm^{\otimes 2}) + \Contr( L^{\otimes 4} , L^{\otimes 2} \otimes \Rm ) + \sI_{4,\divsymb} + \sI_{4,c} .
 \end{multline}
 
 An index count and Equation~\eqref{eqn:defn-Fialkow} imply that
 \begin{align*}
  \Contr_3(\Rm^{\otimes 2}) \subseteq \Contr( W_{\alpha\beta\gamma\delta}\Sch_{\epsilon\zeta} ) \subseteq \Contr( \Fi_{\alpha\beta}\Sch_{\gamma\delta}) + \Contr( L^{\otimes 2} \otimes \Rm)  .
 \end{align*}
 Combining this with Equations~\eqref{eqn:gcP} and~\eqref{eqn:defn-Fialkow} yields
 \begin{equation}
  \label{eqn:step3-reduce3}
  \Contr_3(\Rm^{\otimes 2}) \subseteq \vspan \{ \trFi \otrSch , \lp \Fi, \oSch \rp \} + \Contr( L^{\otimes 4} , L^{\otimes 2} \otimes \Rm ) + \sI_{4,c} .
 \end{equation}

 Statements~\eqref{eqn:step3-prereduce}--\eqref{eqn:step3-reduce3} together yield Statement~\eqref{eqn:step3-rm2-claim}.
 
 We now show that
 \begin{equation}
  \label{eqn:step3-lrm-claim}
  \begin{aligned}
   \MoveEqLeft[1] \Contr( L^{\otimes 2} \otimes \Rm ) \subseteq \vspan \Bigl\{ \lv H \rv^2\otrSch , \lv\tfss\rv^2 \otrSch , \lp \tfss^2 , \oSch \rp , H^{\alpha^\prime}H^{\beta^\prime}W_{\alpha^\prime\alpha\beta^\prime}{}^\alpha , \\
    & \quad \trFi \lv H \rv^2 , H^{\alpha^\prime}\tfss_{\alpha\beta\alpha^\prime}\Fi^{\alpha\beta} , H^{\alpha^\prime}\tfss_{\alpha\beta\alpha^\prime}\oSch^{\alpha\beta} , H^{\alpha^\prime}\tfss^{\alpha\beta\beta^\prime} W_{\alpha\alpha^\prime\beta\beta^\prime} , H^{\alpha^\prime} H^{\beta^\prime}\Sch_{\alpha^\prime\beta^\prime} , \\
    & \quad \lv H \rv^2\Sch_{\alpha^\prime}{}^{\alpha^\prime} ,  \lv\tfss\rv^2\Sch_{\alpha^\prime}{}^{\alpha^\prime} , \tfss_{\alpha\beta}{}^{\alpha^\prime}\tfss^{\alpha\beta\beta^\prime}\Sch_{\alpha^\prime\beta^\prime} \Bigr\} + \Contr(L^{\otimes 4}) + \sI_{4,c} .
  \end{aligned}
 \end{equation}
 Equations~\eqref{eqn:defn-weyl} and~\eqref{eqn:reduce-weyl} imply that
 \begin{multline*}
  \Contr( L^{\otimes 2} \otimes \Rm ) \subseteq \Contr( L^{\otimes 2} \otimes W_{\alpha\beta\alpha^\prime\beta^\prime} , L^{\otimes 2} \otimes W_{\alpha\alpha^\prime\beta\beta^\prime} , L^{\otimes 2} \otimes W_{\alpha\beta\gamma\delta} , \\ L^{\otimes 2} \otimes \Sch_{\alpha^\prime\beta^\prime} , L^{\otimes 2} \otimes \Sch_{\alpha\beta} ) .
 \end{multline*}
 On the one hand, combining Equations~\eqref{eqn:defn-Fialkow} and~\eqref{eqn:reduce-weyl} with Statement~\eqref{eqn:obvious-invariants} yields
 \begin{align*}
  \Contr(L^{\otimes 2} \otimes W_{\alpha\beta\alpha^\prime\beta^\prime} ) & \subseteq \sI_{4,c} , \\
  \Contr( L^{\otimes 2} \otimes W_{\alpha\alpha^\prime\beta\beta^\prime} ) & \subseteq \vspan \{ H^{\alpha^\prime}H^{\beta^\prime} W_{\alpha^\prime\alpha\beta^\prime}{}^\alpha, H^{\alpha^\prime}\tfss^{\alpha\beta\beta^\prime}W_{\alpha\alpha^\prime\beta\beta^\prime} \} \\
   & \quad + \Contr( L^{\otimes 2} \otimes W_{\alpha\beta\gamma\delta} ) + \sI_{4,c} , \\
  \Contr( L^{\otimes 2} \otimes W_{\alpha\beta\gamma\delta} ) & \subseteq \Contr( L^{\otimes 2} \otimes \Fi ) + \Contr( L^{\otimes 4} ) + \sI_{4,c} .
 \end{align*}
 On the other hand, Equation~\eqref{eqn:gcP} implies that
 \begin{equation*}
  \Contr( L^{\otimes 2} \otimes \Sch_{\alpha\beta} ) \subseteq \Contr( L^{\otimes 2} \otimes \oSch , L^{\otimes 2} \otimes \Fi , L^{\otimes 4} ) ,
 \end{equation*}
 while direct computation yields
 \begin{align*}
  \Contr( L^{\otimes 2} \otimes \Sch_{\alpha^\prime\beta^\prime} ) & \subseteq \vspan \{ \lv H \rv^2\Sch_{\alpha^\prime}{}^{\alpha^\prime} , \lv\tfss\rv^2\Sch_{\alpha^\prime}{}^{\alpha^\prime} , H^{\alpha^\prime}H^{\beta^\prime}\Sch_{\alpha^\prime\beta^\prime} , \tfss_{\alpha\beta}{}^{\alpha^\prime}\tfss^{\alpha\beta\beta^\prime}\Sch_{\alpha^\prime\beta^\prime} \} , \\
  \Contr( L^{\otimes 2} \otimes \oSch ) & \subseteq \vspan \{ \lv H \rv^2\otrSch , \lv\tfss\rv^2\otrSch , H^{\alpha^\prime}\tfss_{\alpha\beta\alpha^\prime}\oSch^{\alpha\beta} , \lp \tfss^2 , \oSch \rp  \} , \\
  \Contr( L^{\otimes 2} \otimes \Fi ) & \subseteq \vspan \{ \trFi \lv H \rv^2 , H^{\alpha^\prime}\tfss_{\alpha\beta\alpha^\prime}\Fi^{\alpha\beta} \} + \sI_{4,c} .
 \end{align*}
 Combining these observations yields Statement~\eqref{eqn:step3-lrm-claim}.
 
 Finally, one directly checks using Equation~\eqref{eqn:obvious-invariants} that
 \begin{equation}
  \label{eqn:step3-reduceLLLL}
  \Contr( L^{\otimes 4} ) \subseteq \vspan \{ \lv H\rv^4 , \lv H\rv^2 \lv \tfss\rv^2 , H^{\alpha^\prime}\tfss_{\alpha\beta\alpha^\prime}\tfss^{\alpha\beta\beta^\prime}H_{\beta^\prime} , H^{\alpha^\prime} \tr\tfss^3_{\alpha^\prime} \} + \sI_{4,c} .
 \end{equation}
 
 Combining Statements~\eqref{eqn:step3-rm2-claim}, \eqref{eqn:step3-lrm-claim}, and~\eqref{eqn:step3-reduceLLLL} yields Statement~\eqref{eqn:step3-claim}.
\end{proof}

\Cref{lem:4d-invariants} reduces the verification of \cref{submanifold-alexakis} for four-dimensional submanifolds to an analysis of the conformal linearization of the integral of linear combinations of elements of $\bigcup_{j=0}^4\sI_4^j$.

\begin{proof}[Proof of \cref{low-dimension-submanifold-alexakis} when $\submfdim = 4$]
 Let $I$ be a natural scalar on $4$-dimensional submanifolds $Y$ of $\mfdim$-dimensional Riemannian manifolds $(X,g)$ whose integral over compact $Y$ is invariant under conformal rescaling of $g$.
 Then $I$ has weight $-4$.
 By \cref{lem:4d-invariants}, we may assume that $I \in \vspan \bigcup_{j=0}^4 \sI_4^j$.
 We will show that $I = a(\otrSch^2 - \lv\oSch\rv^2)$ for some constant $a\in\bR$.
 The final conclusion then follows from the fact that
 \begin{equation*}
  \oPf = \frac{1}{8}\lv\oW\rv^2 + \otrSch^2 - \lv\oSch\rv^2
 \end{equation*}
 in dimension four.
 
 \underline{Step 1: Show no dependence on $\sI_4^4$}.
 By definition, there is a constant $a\in\bR$ such that $I \equiv a\Delta \trSch \mod \vspan \bigcup_{j=0}^3\sI_4^j$.
 Let $\Upsilon \in C^\infty(X)$.
 Suppose that $\Upsilon$, $\Upsilon_{\alpha^\prime}$, $\Upsilon_{\alpha^\prime\beta^\prime}$, and $\Upsilon_{\alpha^\prime\beta^\prime\gamma^\prime}$ all vanish along $Y$.
 Then
 \begin{equation*}
  0 = \int_Y (I \darea )^\bullet = -a\int_Y \Upsilon_{\alpha^\prime}{}^{\alpha^\prime}{}_{\beta^\prime}{}^{\beta^\prime} \darea .
 \end{equation*}
 Since $\Upsilon$ is otherwise arbitrary, we conclude that $a=0$.
 Therefore $I \in \vspan \bigcup_{j=0}^3 \sI_4^j$.
 
 \underline{Step 2: Show no dependence on $\sI_4^3$}.
 Step~1 implies that there is a constant $a \in \bR$ such that $I \equiv aH^{\alpha^\prime}\nabla_{\alpha^\prime} \trSch \mod \vspan \bigcup_{j=0}^2 \sI_4^j$.
 Suppose that $\Upsilon \in C^\infty(X)$ is such that $\Upsilon$, $\Upsilon_{\alpha^\prime}$, and $\Upsilon_{\alpha^\prime\beta^\prime}$ all vanish along $Y$.
 Then
 \begin{equation*}
  0 = \int_Y ( I \darea )^\bullet = -a\int_Y H^{\alpha^\prime}\Upsilon_{\beta^\prime}{}^{\beta^\prime}{}_{\alpha^\prime} \darea .
 \end{equation*}
 Since $\Upsilon$ is otherwise arbitrary, we conclude that $a=0$.
 Therefore $I \in \vspan \bigcup_{j=0}^2 \sI_4^j$.
 
 \underline{Step 3: Show no dependence on $\sI_4^2$}.
 Step~2 implies that there are constants $a_1,\dotsc,a_9 \in \bR$ such that
 \begin{equation}
  \label{eqn:step3-general}
  \begin{split}
   I & \equiv a_1\lv H\rv^2 \Sch_{\alpha^\prime}{}^{\alpha^\prime} + a_2H^{\alpha^\prime} H^{\beta^\prime} \Sch_{\alpha^\prime\beta^\prime} + a_3\trFi\Sch_{\alpha^\prime}{}^{\alpha^\prime} \\
    & \quad + a_4W_{\alpha^\prime\alpha\beta^\prime}{}^\alpha \Sch^{\alpha^\prime\beta^\prime} + a_5\lv\tfss\rv^2 \Sch_{\alpha^\prime}{}^{\alpha^\prime} + a_6\tfss^{\alpha\beta}{}_{\alpha^\prime} \tfss_{\alpha\beta\beta^\prime}\Sch^{\alpha^\prime\beta^\prime} \\
    & \quad + a_7\otrSch\Sch_{\alpha^\prime}{}^{\alpha^\prime} + a_8\Sch_{\alpha^\prime}{}^{\alpha^\prime}\Sch_{\beta^\prime}{}^{\beta^\prime} + a_9\Sch_{\alpha^\prime\beta^\prime}\Sch^{\alpha^\prime\beta^\prime} \mod \vspan \sI_4^0 \cup \sI_4^1 .
  \end{split}
 \end{equation}
 Suppose that $\Upsilon \in C^\infty(X)$ is such that $\Upsilon$ and $\Upsilon_{\alpha^\prime}$ vanish along $Y$.
 Then
 \begin{equation}
  \label{eqn:step3-general-variation}
  \begin{split}
  0 & = \int_Y (I \darea)^\bullet \\
   & = -\int_Y \Bigl( ( a_1\lv H\rv^2 + a_3\trFi + a_5\lv\tfss\rv^2 + a_7\otrSch + 2a_8\Sch_{\gamma^\prime}{}^{\gamma^\prime}) g_{\alpha^\prime\beta^\prime} + a_2H_{\alpha^\prime} H_{\beta^\prime} \\
   & \qquad\qquad + a_4W_{\alpha^\prime\alpha\beta^\prime}{}^\alpha + a_6\tfss^{\alpha\beta}{}_{\alpha^\prime} \tfss_{\alpha\beta\beta^\prime} + 2a_9\Sch_{\alpha^\prime\beta^\prime} \Bigr)\Upsilon^{\alpha^\prime\beta^\prime} \darea .
  \end{split}
 \end{equation}
 Since $\Upsilon$ is otherwise arbitrary, we deduce that
 \begin{multline*}
  0 = ( a_1\lv H\rv^2 + a_3\trFi + a_5\lv\tfss\rv^2 + a_7\otrSch + 2a_8\Sch_{\gamma^\prime}{}^{\gamma^\prime}) g_{\alpha^\prime\beta^\prime} \\
   + a_2H_{\alpha^\prime} H_{\beta^\prime} + a_4W_{\alpha^\prime\alpha\beta^\prime}{}^\alpha  + a_6\tfss^{\alpha\beta}{}_{\alpha^\prime} \tfss_{\alpha\beta\beta^\prime} + 2a_9\Sch_{\alpha^\prime\beta^\prime} .
 \end{multline*}
 Contracting this with $\Sch^{\alpha^\prime\beta^\prime}$ and comparing with Equation~\eqref{eqn:step3-general} yields
 \begin{equation*}
  I \equiv -a_8\Sch_{\alpha^\prime}{}^{\alpha^\prime}\Sch_{\beta^\prime}{}^{\beta^\prime} - a_9\Sch_{\alpha^\prime\beta^\prime}\Sch^{\alpha^\prime\beta^\prime} \mod \vspan \sI_4^0 \cup \sI_4^1 .
 \end{equation*}
 Repeating the variational computation~\eqref{eqn:step3-general-variation} yields $a_8\Sch_{\gamma^\prime}{}^{\gamma^\prime} g_{\alpha^\prime\beta^\prime} + a_9\Sch_{\alpha^\prime\beta^\prime}=0$.
 Therefore $I \in \vspan \sI_4^0 \cup \sI_4^1$.
 
 \underline{Step 4: Show no dependence on $\sI_4^1$}.
 Step~3 implies that there are constants $a_1,\dotsc,a_{14} \in \bR$ such that
 \begin{align*}
  I & \equiv a_1H^{\alpha^\prime}\oDelta H_{\alpha^\prime} + a_2H^{\alpha^\prime}\onabla^\alpha\Di_{\alpha\alpha^\prime} + a_3H^{\alpha^\prime}\onabla^\alpha W_{\alpha\beta\alpha^\prime}{}^\beta + a_4\lv H\rv^4 \\
   & \quad + a_5\lv H\rv^2\lv\tfss\rv^2 + a_6 H^{\alpha^\prime}\tfss_{\alpha\beta\alpha^\prime}\tfss^{\alpha\beta\beta^\prime}H_{\beta^\prime} + a_7\lv H\rv^2\otrSch + a_8 \trFi \lv H\rv^2 \\
   & \quad + a_9 H^{\alpha^\prime} \tr\tfss^3_{\alpha^\prime} + a_{10}H^{\alpha^\prime} \tfss^{\alpha\beta\beta^\prime}W_{\alpha\alpha^\prime\beta\beta^\prime} + a_{11}H^{\alpha^\prime}\tfss_{\alpha\beta\alpha^\prime}\Fi^{\alpha\beta} \\
   & \quad + a_{12}H^{\alpha^\prime} \tfss_{\alpha\beta\alpha^\prime}\oSch^{\alpha\beta} + a_{13}H^{\alpha^\prime} H^{\beta^\prime} W_{\alpha^\prime\alpha\beta^\prime}{}^\alpha + a_{14} H^{\alpha^\prime} \mC_{\alpha^\prime} \mod \vspan \sI_4^0 .
 \end{align*}
 Suppose that $\Upsilon \in C^\infty(X)$ is such that $\Upsilon$ vanishes along $Y$.
 Then
 \begin{align*}
  0 & = \int_Y ( I \darea )^\bullet \\
   & = -\int_Y \Bigl( 2a_1\oDelta H_{\alpha^\prime} + a_2\onabla^\alpha\Di_{\alpha\alpha^\prime} + a_3\onabla^\alpha W_{\alpha\beta\alpha^\prime}{}^\beta + 4a_4\lv H\rv^2H_{\alpha^\prime} + 2a_5\lv\tfss\rv^2H_{\alpha^\prime} \\
    & \qquad + 2a_6H^{\beta^\prime}\tfss_{\alpha\beta\beta^\prime}\tfss^{\alpha\beta}{}_{\alpha^\prime} + 2a_7\otrSch H_{\alpha^\prime} + 2a_8 \trFi H_{\alpha^\prime} + a_{9}\tr\tfss^3_{\alpha^\prime} \\
    & \qquad + a_{10}\tfss^{\alpha\beta\beta^\prime}W_{\alpha\alpha^\prime\beta\beta^\prime} + a_{11}\Fi^{\alpha\beta}\tfss_{\alpha\beta\alpha^\prime} + a_{12}\oSch^{\alpha\beta}\tfss_{\alpha\beta\alpha^\prime} \\
    & \qquad + 2a_{13}H^{\beta^\prime} W_{\alpha^\prime\alpha\beta^\prime}{}^\alpha + a_{14}\mC_{\alpha^\prime} \Bigr) \Upsilon^{\alpha^\prime} \darea .
 \end{align*}
 Since $\Upsilon$ is otherwise arbitrary, we deduce that
 \begin{align*}
  0 & = 2a_1\oDelta H_{\alpha^\prime} + a_2\onabla^\alpha\Di_{\alpha\alpha^\prime} + a_3\onabla^\alpha W_{\alpha\beta\alpha^\prime}{}^\beta + 4a_4\lv H\rv^2H_{\alpha^\prime} + 2a_5\lv\tfss\rv^2H_{\alpha^\prime} \\
   & \quad + 2a_6H^{\beta^\prime}\tfss_{\alpha\beta\beta^\prime}\tfss^{\alpha\beta}{}_{\alpha^\prime} + 2a_7\otrSch H_{\alpha^\prime} + 2a_8 \trFi H_{\alpha^\prime} + a_{9}\tr\tfss^3_{\alpha^\prime} + a_{10}\tfss^{\alpha\beta\beta^\prime}W_{\alpha\alpha^\prime\beta\beta^\prime} \\
   & \quad + a_{11}\Fi_{\alpha\beta}\tfss^{\alpha\beta}{}_{\alpha^\prime} + a_{12}\oSch^{\alpha\beta}\tfss_{\alpha\beta\alpha^\prime} + 2a_{13}H^{\beta^\prime} W_{\alpha^\prime\alpha\beta^\prime}{}^\alpha + a_{14}\mC_{\alpha^\prime} .
 \end{align*}
 Contracting with $H^{\alpha^\prime}$ and comparing this with our initial formula for $I$ implies that
 \begin{multline*}
  I \equiv -a_1H^{\alpha^\prime}\oDelta H_{\alpha^\prime} - 3a_4\lv H\rv^4 - a_5\lv H\rv^2 \lv\tfss\rv^2 - a_6 H^{\alpha^\prime}\tfss_{\alpha\beta\alpha^\prime}\tfss^{\alpha\beta\beta^\prime}H_{\beta^\prime} \\
   - a_7\lv H\rv^2\otrSch - a_8\trFi\lv H\rv^2 - a_{13}H^{\alpha^\prime} H^{\beta^\prime} W_{\alpha^\prime\alpha\beta^\prime}{}^\alpha \mod \vspan \sI_4^0 .
 \end{multline*}
 Again let $\Upsilon \in C^\infty(X)$ be such that $\Upsilon \rv_Y = 0$.
 Then
 \begin{align*}
  0 & = \int_Y \left( I \darea \right)^\bullet \\
   & = \int_Y \Bigl( 2a_1\oDelta H_{\alpha^\prime} + 12a_4\lv H \rv^2H_{\alpha^\prime} + 2a_5\lv\tfss\rv^2H_{\alpha^\prime} + 2a_6H^{\beta^\prime}\tfss_{\alpha\beta\beta^\prime}\tfss^{\alpha\beta}{}_{\alpha^\prime} \\
    & \qquad  + 2a_7\otrSch H_{\alpha^\prime} + 2a_8\trFi H_{\alpha^\prime} + 2a_{13}H^{\beta^\prime}W_{\alpha^\prime\alpha\beta^\prime}{}^\alpha \Bigr)\Upsilon^{\alpha^\prime} \darea .
 \end{align*}
 It follows that
 \begin{equation*}
  I \equiv 3a_4\lv H \rv^4 \mod \vspan \sI_4^0 .
 \end{equation*}
 Repeating this variational argument one more time yields $I \in \vspan \sI_4^0$.
 
 \underline{Step 5: Show $I$ is proportional to $\otrSch^2 - \lv\oSch\rv^2$}.
 Step~4 implies that there are constants $a_1,\dotsc,a_8 \in \bR$ such that
 \begin{multline}
  \label{eqn:step5-expression}
  I = a_1(\otrSch^2 - \lv \oSch\rv^2) + a_2\otrSch^2 + a_3\lp\Fi,\oSch\rp + a_4\lv \Di\rv^2 \\
   + a_5W_{\alpha\beta\alpha^\prime}{}^\beta \Di^{\alpha\alpha^\prime} + a_6\lv\tfss\rv^2\otrSch + a_7\lp \tfss^2, \oSch\rp + a_8\trFi\otrSch .
 \end{multline}
 Let $\Upsilon \in C^\infty(X)$.
 By computing the conformal variation and integrating by parts, we deduce that
 \begin{align*}
  0 & = \int_Y ( I \darea )^\bullet \\
  & = -\int_Y \Bigl( 2a_2\oDelta\otrSch + a_3\onabla^\alpha\onabla^\beta \Fi_{\alpha\beta} - 2a_4\onabla^\alpha(\tfss_{\alpha\beta\alpha^\prime}\Di^{\beta\alpha^\prime}) - a_5\onabla^\alpha(\tfss_\alpha{}^{\beta\alpha^\prime}W_{\beta\gamma\alpha^\prime}{}^\gamma) \\
   & \qquad\qquad + a_6\oDelta\lv\tfss\rv^2 + a_7\onabla^\alpha\onabla^\beta\tfss^2_{\alpha\beta} + a_8\oDelta \trFi \Bigr) \Upsilon \darea .
 \end{align*}
 Since $\Upsilon$ is arbitrary, we deduce from this and Equation~\eqref{eqn:div-tfss2} that
 \begin{multline}
  \label{eqn:step5-noconf}
  0 = 2a_2\oDelta\otrSch + a_3\onabla^\alpha\onabla^\beta \Fi_{\alpha\beta} - 2(a_4+a_7)\onabla^\alpha(\tfss_{\alpha\beta\alpha^\prime}\Di^{\beta\alpha^\prime}) + a_8\oDelta \trFi  \\
    + \left( a_6 + \frac{a_7}{2} \right)\oDelta\lv\tfss\rv^2 - (a_5+a_7)\onabla^\alpha(\tfss_\alpha{}^{\beta\alpha^\prime}W_{\beta\gamma\alpha^\prime}{}^\gamma) - a_7\onabla^\alpha(\tfss^{\beta\gamma\alpha^\prime}W_{\beta\alpha\gamma\alpha^\prime}) .
 \end{multline}
 In dimension four, the conformal linearizations of the terms in this expression are
 \begin{equation*}
  \begin{split}
  \bigl( \oDelta\otrSch \bigr)^\bullet & = -\oDelta^2\Upsilon - 2\onabla^\alpha(\otrSch \Upsilon_\alpha) , \\
  \bigl( \onabla^\alpha\onabla^\beta\Fi_{\alpha\beta} \bigr)^\bullet & = \onabla^\alpha( 2\Fi_{\alpha\beta}\Upsilon^\beta - \trFi \Upsilon_\alpha ) , \\
  \bigl( \onabla^\alpha(\tfss_{\alpha\beta\alpha^\prime}\Di^{\beta\alpha^\prime}) \bigr)^\bullet & = -\onabla^\alpha( \tfss^2_{\alpha\beta}\Upsilon^\beta ) , \\
  \bigl( \oDelta\lv\tfss\rv^2 \bigr)^\bullet & = -2\onabla^\alpha( \lv\tfss\rv^2\Upsilon_\alpha ) , \\
  \bigl( \oDelta\trFi \bigr)^\bullet & = -2\onabla^\alpha(\trFi\Upsilon_\alpha) , \\
  \bigl( \onabla^\alpha(\tfss^{\beta\gamma\alpha^\prime}W_{\beta\alpha\gamma\alpha^\prime}) \bigr)^\bullet & = 0 , \\
  \bigl( \onabla^\alpha ( \tfss_{\alpha\beta\alpha^\prime}W^{\beta\gamma\alpha^\prime}{}_\gamma ) \bigr)^\bullet & = 0 .
  \end{split}
 \end{equation*}
 Since only $(\oDelta\otrSch)^\bullet$ depends on the full tangential four-jet of $\Upsilon$, we see that $a_2=0$.
 Taking arbitrary $\Upsilon$ and applying \cref{technical-linear-independence} below yields
 \begin{align*}
  a_3 &= 0, & a_4 & = -a_7, & a_8 & = 0, & a_6 & = -\frac{a_7}{2}, && \text{if $\mfdim > 5$}, \\
  a_3 & = 0, & a_4 & = -a_7, & a_6 & = -\frac{a_7}{2} - \frac{a_8}{6}, &&&& \text{if $\mfdim = 5$} .
 \end{align*}
 Applying \cref{technical-linear-independence} below to Equation~\eqref{eqn:step5-noconf} yields
 \begin{align*}
  a_7 & = 0, & a_5 & = 0, && \text{if $\mfdim > 5$} , \\
  a_7 & = 0 , &&&& \text{if $\mfdim = 5$} .
 \end{align*}
 Inserting these back into Equation~\eqref{eqn:step5-expression} yields $I = a_1(\otrSch^2 - \lv\oSch\rv^2)$.
\end{proof}

We conclude with a technical lemma used above which establishes the linear independence of certain sets of natural submanifold tensors.

\begin{lemma}
 \label{technical-linear-independence}
 Each of the sets
 \begin{equation*}
  \left\{ \Fi_{\alpha\beta} , \tfss_{\alpha\beta}^2, \lv\tfss\rv^2g_{\alpha\beta} \right\} \quad\text{and}\quad \left\{ \onabla^\alpha(\tfss^{\beta\gamma\alpha^\prime}W_{\alpha\beta\gamma\alpha^\prime}) \right\}
 \end{equation*}
 of natural tensors on $4$-dimensional submanifolds of $5$-dimensional Riemannian manifolds is linearly independent.
 
 If $\mfdim \geq 6$, then each of the sets
 \begin{equation*}
  \left\{ \Fi_{\alpha\beta} , \tfss_{\alpha\beta}^2, \lv\tfss\rv^2g_{\alpha\beta}, \trFi g_{\alpha\beta} \right\} \quad\text{and}\quad \left\{ \onabla^\alpha(\tfss^{\beta\gamma\alpha^\prime}W_{\alpha\beta\gamma\alpha^\prime}) , \onabla^\alpha(\tfss_\alpha{}^{\beta\alpha^\prime}W_{\beta\gamma\alpha^\prime}{}^\gamma) \right\}
 \end{equation*}
 of natural tensors on $4$-dimensional submanifolds of $\mfdim$-dimensional Riemannian manifolds is linearly independent.
\end{lemma}

\begin{proof}
 First observe that the embedding $\bR \times S^3 \times \{0\} \hookrightarrow \bR \times \bR^4 \times \bR^{\mfdim-5}$ of a cylinder into flat Euclidean space is such that $\tfss^2_{\alpha\beta}$ is not proportional to the induced metric.
 Hence, by Equation~\eqref{eqn:defn-Fialkow}, it suffices to prove the following claim:
 If $\mfdim = 5$, then
 \begin{equation*}
  W_{\alpha\gamma\beta}{}^\gamma \quad\text{and}\quad \onabla^\alpha(\tfss^{\beta\gamma\alpha^\prime}W_{\alpha\beta\gamma\alpha^\prime})
 \end{equation*}
 are nonzero natural submanifold tensors, and if $\mfdim \geq 6$, then each of the sets
 \begin{equation*}
  \left\{ W_{\alpha\gamma\beta}{}^\gamma , W_{\gamma\delta}{}^{\gamma\delta}g_{\alpha\beta} \right\} \quad\text{and}\quad \left\{ \onabla^\alpha(\tfss^{\beta\gamma\alpha^\prime}W_{\alpha\beta\gamma\alpha^\prime}) , \onabla^\alpha(\tfss_\alpha{}^{\beta\alpha^\prime}W_{\beta\gamma\alpha^\prime}{}^\gamma) \right\}
 \end{equation*}
 is linearly independent.
 We prove this claim by computing some local examples.
 To that end,  we do not use Einstein summation notation in the rest of this proof.

 Set $Y := \{ (x,0) \in \bR^4 \times \bR^{\mfdim-4} \} \subseteq \bR^{\mfdim}$ and denote $p := (0,0) \in Y$.
 Equip $\bR^{\mfdim}$ with the metric
 \begin{equation*}
  g = \sum_{a=1}^{\mfdim} e^{2f_a} (dx^a)^2
 \end{equation*}
 for $f_a \colon \bR^{\mfdim} \to \bR$ functions satisfying $f_a(p)=0$ and $df_a(p)=0$.
 Mirroring our abstract index notation, we let $a \in \{ 1, \dotsc, \mfdim \}$ and $\alpha \in \{ 1, \dotsc, 4 \}$ and $\alpha^\prime \in \{ 5, \dotsc , \mfdim \}$.
 Unless otherwise indicated, we use the convention that distinct characters are not equal;
 e.g.\ $\Sch_{\alpha\beta}$ denotes a component with $\alpha,\beta \in \{ 1, \dotsc, 4 \}$ and $\alpha \not= \beta$.
 
 The nonvanishing Christoffel symbols of $g$ are
 \begin{align*}
  \Gamma_{aa}^a & = \partial_af_a , & \Gamma_{ab}^a & = \partial_bf_a , & \Gamma_{aa}^b = -e^{2(f_a-f_b)}\partial_bf_a .
 \end{align*}
 Thus the only nonvanishing components of the second fundamental form are
 \begin{equation*}
  L_{\alpha\alpha\alpha^\prime} = -e^{2f_\alpha}\partial_{\alpha^\prime}f_{\alpha} .
 \end{equation*}
 In particular, $L=0$ at $p$.
 Since the Christoffel symbols all vanish at $p$, we compute that
 \begin{equation*}
  R_{ab}{}^c{}_d = \partial_a\Gamma_{bd}^c - \partial_b\Gamma_{ad}^c
 \end{equation*}
 at $p$, where distinct indices can be equal in the above display.
 Thus the only nonvanishing components of $\Rm$ at $p$ are
 \begin{align*}
  R_{abab} & = -\partial_{aa}^2f_b - \partial_{bb}^2f_a , & R_{acbc} & = -\partial^2_{ab}f_c ,
 \end{align*}
 and those that can be obtained from the symmetries of $\Rm$.
 Therefore
 \begin{align*}
  W_{abab} & = -\partial_{aa}^2f_b - \partial_{bb}^2f_a - \Sch_{aa} - \Sch_{bb} , \\
  W_{acbc} & = -\partial^2_{ab}f_c - \Sch_{ab} , 
 \end{align*}
 and
 \begin{align*}
  \Sch_{aa} & = -\frac{1}{\mfdim-1}\sum_{b \not= a} (\partial^2_{aa}f_b + \partial^2_{bb}f_a) + \frac{1}{(\mfdim-1)(\mfdim-2)}\sum_{\substack{b,c \not=a \\ b\not=c}}\partial_{bb}^2f_c , \\
  \Sch_{ab} & = -\frac{1}{\mfdim-2}\sum_{c \not\in \{ a,b \}} \partial^2_{ab}f_c ,
 \end{align*}
 at $p$.
 Additionally, the only nonvanishing components of $\onabla L$ at $p$ are
 \begin{align*}
  \onabla_{\alpha}L_{\alpha\alpha\alpha^\prime} & = -\partial^2_{\alpha\alpha^\prime}f_\alpha , & \onabla_{\beta}L_{\alpha\alpha\alpha^\prime} & = -\partial^2_{\beta\alpha^\prime}f_\alpha .
 \end{align*}
 
 We now simplify our computation by taking $f_{\alpha^\prime}=0$ and assuming at $p$ that
 \begin{equation}
  \label{eqn:assume-parallel-H}
  \sum_{\beta=1}^4 \partial^2_{\alpha\alpha^\prime}f_\beta = 0 \quad\text{and}\quad \sum_{\beta^\prime=5}^{\mfdim} \partial^2_{\beta^\prime\beta^\prime}f_\alpha = 0 
 \end{equation}
 for each $\alpha \in \{ 1, \dotsc, 4\}$ and $\alpha^\prime \in \{ 5, \dotsc, \mfdim \}$.
 This assumption and our formula for $\Sch_{ab}$ yield
 \begin{align*}
  \Sch_{\alpha\alpha^\prime} & = \frac{1}{\mfdim-2}\partial^2_{\alpha\alpha^\prime}f_\alpha , \\
  \onabla_\alpha\tfss_{\alpha\alpha\alpha^\prime} & = -\partial^2_{\alpha\alpha^\prime}f_\alpha , \\
  \onabla_\beta\tfss_{\alpha\alpha\alpha^\prime} & = -\partial^2_{\beta\alpha^\prime}f_\alpha ,
 \end{align*}
 at $p$.  It follows that, at $p$,
 \begin{align*}
  \sum_{\gamma=1}^4 W_{\alpha\gamma\beta}{}^\gamma & = -\frac{\mfdim-4}{\mfdim-2}\sum_{\gamma\not\in\{\alpha,\beta\}} \partial^2_{\alpha\beta}f_\gamma , \\
  \sum_{\alpha,\beta=1}^4 W_{\alpha\beta}{}^{\alpha\beta} & = -\frac{2(\mfdim-4)(\mfdim-5)}{(\mfdim-1)(\mfdim-2)}\sum_{\alpha\not=\beta} \partial_{\alpha\alpha}^2f_\beta ,
 \end{align*}
 and
 \begin{align*}
  \sum_{\alpha,\beta,\gamma,\alpha^\prime} \onabla^\alpha( \tfss^{\beta\gamma\alpha^\prime}W_{\alpha\beta\gamma\alpha^\prime}) & = -\sum_{\alpha\not=\beta}\sum_{\alpha^\prime}(\partial^2_{\beta\alpha^\prime}f_\alpha)^2 + \frac{1}{\mfdim-2}\sum_\alpha\sum_{\alpha^\prime} (\partial_{\alpha\alpha^\prime}^2f_\alpha)^2 , \\
  \sum_{\alpha,\beta,\gamma,\alpha^\prime}\onabla^\alpha(\tfss_{\alpha}{}^{\beta\alpha^\prime}W_{\beta\gamma\alpha^\prime}{}^\gamma) & = -\frac{\mfdim-5}{\mfdim-2}\sum_{\alpha}\sum_{\alpha^\prime} (\partial^2_{\alpha\alpha^\prime}f_\alpha)^2 .
 \end{align*}
 Taking
 \begin{align*}
  f_1 & = s(x^2)^2 + tx^2x^3 + ux^1x^5 , & f_2 & = vx^1x^5 , \\
  f_3 & = -(u+v)x^1x^5 , & f_4 & = 0 ,
 \end{align*}
 with $s,t,u,v \in \bR$ yields our claim.
\end{proof}

\appendix
\section{Computations of conformal submanifold invariants}
\label{sec:extra}

In this appendix we prove the five propositions in Subsection~\ref{subsec:invariants/comparison} relating our scalar conformal submanifold invariants of Subsection~\ref{subsec:invariants/construct} to previously-known invariants.
To that end, we first establish a useful formula for the inner product $\lp \tfss, \oDelta\tfss \rp$ which is closely related to Simons' formula~\cite{Simons1968}*{Theorem~4.2.1} for the Laplacian of the second fundamental form of a minimal submanifold.

\begin{lemma}
 \label{basic-invariants}
 Let $i \colon Y^{\submfdim} \to (X^{\mfdim},g)$ be an immersion with  $3 \leq \submfdim < \mfdim$.
 Then
 \begin{equation}
  \label{eqn:simons}
  \begin{split}
  \tfss^{\alpha\beta\alpha^\prime}\oDelta\tfss_{\alpha\beta\alpha^\prime} &
  = \tfss^{\alpha\beta\alpha^\prime}\Bigl(
  -\submfdim\onabla_\alpha\Di_{\beta\alpha^\prime} - \onabla_\alpha
  W_{\beta\gamma\alpha^\prime}{}^\gamma + \onabla^\gamma
  W_{\gamma\alpha\alpha^\prime\beta}
  + \otrSch\tfss_{\alpha\beta\alpha^\prime} \\ 
   \MoveEqLeft + \submfdim \oSch_\alpha{}^\gamma\tfss_{\gamma\beta\alpha^\prime} - W_\alpha{}^\gamma{}_\beta{}^\delta\tfss_{\gamma\delta\alpha^\prime} - W_\alpha{}^\gamma{}_{\alpha^\prime}{}^{\beta^\prime}\tfss_{\gamma\beta\beta^\prime} + 2\Fi_\alpha{}^\gamma\tfss_{\gamma\beta\alpha^\prime} \\ 
   \MoveEqLeft - \tfss_{\alpha\beta\beta^\prime}\tfss^{\gamma\delta\beta^\prime}\tfss_{\gamma\delta\alpha^\prime} - \tfss_{\alpha\delta\alpha^\prime}\tfss^{\gamma\delta\beta^\prime}\tfss_{\gamma\beta\beta^\prime} + 2\tfss^{\gamma\delta}{}_{\alpha^\prime}\tfss_{\alpha\delta\beta^\prime}\tfss_{\beta\gamma}{}^{\beta^\prime} \Bigr) . 
  \end{split}
 \end{equation}
\end{lemma}

\begin{proof}
 Using Equations~\eqref{eqn:gcdL} and~\eqref{eqn:gcD}, we compute that
 \begin{align*}
  \tfss^{\alpha\beta\alpha^\prime}\oDelta\tfss_{\alpha\beta\alpha^\prime} &
  = \tfss^{\alpha\beta\alpha^\prime}\onabla^\gamma\left(
  \onabla_\alpha\tfss_{\gamma\beta\alpha^\prime} +
  W_{\gamma\alpha\alpha^\prime\beta} -
  g_{\gamma\beta}\Di_{\alpha\alpha^\prime} \right) \\ 
   & = \tfss^{\alpha\beta\alpha^\prime}\Bigl(
  \onabla_\alpha\onabla^\gamma\tfss_{\gamma\beta\alpha^\prime} +
  (\submfdim-2)\oSch_\alpha{}^\gamma\tfss_{\gamma\beta\alpha^\prime} +
  \otrSch\tfss_{\alpha\beta\alpha^\prime} -
  \oR_\alpha{}^\gamma{}_\beta{}^\delta\tfss_{\gamma\delta\alpha^\prime} \\ 
    & \qquad -
  \oR_\alpha{}^\gamma{}_{\alpha^\prime}{}^{\beta^\prime}\tfss_{\gamma\beta\beta^\prime}
  + \onabla^\gamma W_{\gamma\alpha\alpha^\prime\beta} -
  \onabla_\beta\Di_{\alpha\alpha^\prime} \Bigr) \\ 
   & = \tfss^{\alpha\beta\alpha^\prime}
  \Bigl(-\submfdim\onabla_\beta\Di_{\alpha\alpha^\prime} - \onabla_\alpha
  W_{\beta\gamma\alpha^\prime}{}^\gamma + \onabla^\gamma
  W_{\gamma\alpha\alpha^\prime\beta} \\ 
    & \qquad -
  \oW_\alpha{}^\gamma{}_\beta{}^\delta\tfss_{\gamma\delta\alpha^\prime} +
  \submfdim\oSch_\alpha{}^\gamma\tfss_{\gamma\beta\alpha^\prime} +
  \otrSch\tfss_{\alpha\beta\alpha^\prime} -
  \oR_\alpha{}^\gamma{}_{\alpha^\prime}{}^{\beta^\prime}\tfss_{\gamma\beta\beta^\prime}
  \Bigr) . 
 \end{align*}
 Combining this with Equations~\eqref{eqn:gcW} and~\eqref{eqn:gcnc} yields Equation~\eqref{eqn:simons}.
\end{proof}

We now prove the results in Subsection~\ref{subsec:invariants/comparison}.

\begin{proof}[Proof of \cref{blitz-gover-waldron}]
 The definitions~\eqref{eqn:defn-Fialkow} and~\eqref{eqn:mC} of $\Fi_{\alpha\beta}$ and $\mC_{\alpha\alpha^\prime\beta}$, respectively, imply that
 \begin{multline}
  \label{eqn:LmC}
   \tfss^{\alpha\beta\alpha^\prime}\mC_{\alpha\alpha^\prime\beta} = \tfss^{\alpha\beta\alpha^\prime}C_{\alpha\alpha^\prime\beta} + H^{\alpha^\prime}\tr\tfss_{\alpha^\prime}^3 - (\submfdim-2)H^{\alpha^\prime}\tfss^{\alpha\beta}{}_{\alpha^\prime}\Fi_{\alpha\beta} \\
    - H^{\alpha^\prime}\tfss^{\alpha\beta\beta^\prime}W_{\alpha\alpha^\prime\beta\beta^\prime} - H^{\alpha^\prime}\tfss^{\alpha\beta}{}_{\alpha^\prime}W_{\alpha\gamma\beta}{}^{\gamma} .
 \end{multline}

 First set
 \begin{multline}
  \label{eqn:defn-cmI}
  \cmI_1 := -\oDelta\lv\tfss\rv^2 + 2\otrSch\lv\tfss\rv^2 + 2(\submfdim-6) \Bigl[ \tfss_{\alpha\beta}^2\oSch^{\alpha\beta} - \onabla^\alpha(\Di^{\beta\alpha^\prime}\tfss_{\alpha\beta\alpha^\prime}) \\
   \qquad\qquad - (\submfdim-3)\lv\Di\rv^2 - \Di^{\alpha\alpha^\prime}W_{\alpha\beta\alpha^\prime}{}^\beta - \tfss^{\alpha\beta\alpha^\prime}\mC_{\alpha\alpha^\prime\beta} \Bigr] .
 \end{multline}
 Equation~\eqref{eqn:gcP} and \cref{2I-plus-J} imply that
 \begin{equation}
  \label{eqn:2Wm+Wn}
  \cmI_1 = 2\Wm + \Wn - 2\trFi\lv\tfss\rv^2 - 2(\submfdim-6)\tfss^2_{\alpha\beta}\Fi^{\alpha\beta} .
 \end{equation}

 Second set
 \begin{multline*}
  \cmI_2 := \tfss^{\alpha\beta\alpha^\prime}\oDelta\tfss_{\alpha\beta\alpha^\prime} + \submfdim\onabla^\alpha(\Di^{\beta\alpha^\prime}\tfss_{\alpha\beta\alpha^\prime}) + \submfdim(\submfdim-1)\lv\Di\rv^2 \\
    + (\submfdim+4)\Di^{\alpha\alpha^\prime}W_{\alpha\beta\alpha^\prime}{}^\beta- \submfdim \tfss^2_{\alpha\beta}\oSch^{\alpha\beta} - \otrSch\lv\tfss\rv^2 + (\submfdim-4)\tfss^{\alpha\beta\alpha^\prime}\mC_{\alpha\alpha^\prime\beta} .
 \end{multline*}
 Combining the formulas in \cref{rk:mK} with \cref{basic-invariants} and Equation~\eqref{eqn:gcD} yields
 \begin{equation}
  \label{eqn:cmI2-invariant}
  \begin{split}
   \cmI_2 & = \mK_1 - \mK_2 - W_{\alpha\beta\alpha^\prime}{}^\beta W^{\alpha\gamma\alpha^\prime}{}_\gamma - \frac{1}{2} W_{\alpha\beta\alpha^\prime\gamma}W^{\alpha\beta\alpha^\prime\gamma} - W_\alpha{}^\gamma{}_\beta{}^\delta \tfss^{\alpha\beta\alpha^\prime}\tfss_{\gamma\delta\alpha^\prime} \\
    & \quad - W_\alpha{}^\beta{}_{\alpha^\prime}{}^{\beta^\prime}\tfss^{\gamma\alpha\alpha^\prime}\tfss_{\gamma\beta\beta^\prime} + 2\tfss^2_{\alpha\beta}\Fi^{\alpha\beta} - \tfss^{\alpha\beta\alpha^\prime}\tfss_{\alpha\beta\beta^\prime}\tfss^{\gamma\delta\beta^\prime}\tfss_{\gamma\delta\alpha^\prime} \\
    & \quad - \lv\tfss^2\rv^2 + 2\tfss^{\alpha\beta\alpha^\prime}\tfss^{\gamma\delta}{}_{\alpha^\prime}\tfss_{\alpha\gamma\beta^\prime}\tfss_{\beta\delta}{}^{\beta^\prime} .
  \end{split}
 \end{equation}
 
 Third set
 \begin{equation*}
  \cmI_3 := \onabla^\beta(\Di^{\alpha\alpha^\prime}\tfss_{\alpha\beta\alpha^\prime}) + \frac{1}{\submfdim-1}\onabla^\beta(\tfss_\beta{}^{\alpha\alpha^\prime}\onabla^\gamma\tfss_{\gamma\alpha\alpha^\prime}) - \frac{\submfdim-4}{\submfdim-1}\Di^{\alpha\alpha^\prime}W_{\alpha\beta\alpha^\prime}{}^{\beta} .
 \end{equation*}
 Equation~\eqref{eqn:gcD} implies that
 \begin{equation}
  \label{eqn:cmI3-invariant}
  \cmI_3 = -\frac{1}{\submfdim-1}\mK_2 .
 \end{equation}
 
 Equations~\eqref{eqn:2Wm+Wn}--\eqref{eqn:cmI3-invariant} imply that
 \begin{equation*}
  W\!m = \frac{\submfdim(\submfdim-1)}{4(\submfdim-6)}\cmI_1 + \frac{\submfdim-3}{2}\cmI_2 + \submfdim\cmI_3 .
 \end{equation*}
 Applying the definitions of $\cmI_1$, $\cmI_2$, and $\cmI_3$ yields
 \begin{align*}
  \MoveEqLeft[0] W\!m = \frac{\submfdim-3}{2}\tfss^{\alpha\beta\alpha^\prime}\oDelta\tfss_{\alpha\beta\alpha^\prime} + \frac{\submfdim}{\submfdim-1}\onabla^\alpha(\tfss_\alpha{}^{\beta\alpha^\prime}\onabla^\gamma\tfss_{\gamma\beta\alpha^\prime}) \\
   & \quad + \frac{1}{\submfdim-6}\left( - \frac{\submfdim(\submfdim-1)}{4}\oDelta\lv\tfss\rv^2 + (4\submfdim-9)\otrSch\lv\tfss\rv^2 \right) - 3(\submfdim-2)\tfss^{\alpha\beta\alpha^\prime} C_{\alpha\alpha^\prime\beta}+ \submfdim\tfss^2_{\alpha\beta}\oSch^{\alpha\beta} \\
   & \quad - \frac{3(\submfdim-2)}{\submfdim-1}\Di^{\alpha\alpha^\prime}W_{\alpha\beta\alpha^\prime}{}^\beta - 3(\submfdim-2)H^{\alpha^\prime}\tr\tfss_{\alpha^\prime}^3 + 3(\submfdim-2)^2H^{\alpha^\prime}\tfss^{\alpha\beta}{}_{\alpha^\prime}\Fi_{\alpha\beta} \\
   & \quad + 3(\submfdim-2)H^{\alpha^\prime}\tfss^{\alpha\beta\beta^\prime}W_{\alpha\alpha^\prime\beta\beta^\prime} + 3(\submfdim-2)H^{\alpha^\prime}\tfss^{\alpha\beta}{}_{\alpha^\prime}W_{\alpha\gamma\beta}{}^\gamma .
 \end{align*}
 Specializing to the case $\submfdim=4$ and $\mfdim=5$ yields the final conclusion.
\end{proof}

\begin{proof}[Proof of \cref{juhl1}]
 Set
 \begin{align*}
  \cmJ_1 & := \tfss^{\alpha\beta\alpha^\prime}\nabla_{\alpha^\prime}W_{\alpha\gamma\beta}{}^\gamma - 2\Di^{\alpha\alpha^\prime}W_{\alpha\beta\alpha^\prime}{}^\beta - 2H^{\alpha^\prime}\tfss^{\alpha\beta}{}_{\alpha^\prime}W_{\alpha\gamma\beta}{}^\gamma - 2\tfss^{\alpha\beta\alpha^\prime}\mC_{\alpha\alpha^\prime\beta} , \\
  \cmJ_2 & := \onabla^\alpha\onabla^\beta\tfss^2_{\alpha\beta} - \frac{1}{2}\oDelta\lv\tfss\rv^2 + (\submfdim-2)\onabla^\alpha(\Di^{\beta\alpha^\prime}\tfss_{\alpha\beta\alpha^\prime}) \\
   & \quad - (\submfdim-4)\tfss^{\alpha\beta\alpha^\prime}\mC_{\alpha\alpha^\prime\beta} - (\submfdim-4)\Di^{\alpha\alpha^\prime}W_{\alpha\beta\alpha^\prime}{}^\beta .
 \end{align*}
 On the one hand, combining Equations~\eqref{eqn:weyl-bianchi}, \eqref{eqn:second-fundamental-form-dual}, and~\eqref{eqn:second-fundamental-form} yields
 \begin{align*}
  \tfss^{\alpha\beta\alpha^\prime}\nabla_{\alpha^\prime} W_{\alpha\gamma\beta}{}^\gamma & = \tfss^{\alpha\beta\alpha^\prime}\Bigl( \onabla_\alpha W_{\alpha^\prime\gamma\beta}{}^\gamma + \onabla_\gamma W_{\alpha\alpha^\prime\beta}{}^\gamma + (\submfdim-2)\mC_{\alpha\alpha^\prime\beta} \\
   & \quad + \tfss_{\alpha}{}^\delta{}_{\alpha^\prime}W_{\delta\gamma\beta}{}^\gamma + \tfss_{\gamma}{}^\delta{}_{\alpha^\prime}W_{\alpha\delta\beta}{}^\gamma - \tfss_{\alpha\beta}{}^{\beta^\prime}W_{\alpha^\prime\gamma\beta^\prime}{}^\gamma \\
   & \quad
 + \tfss_\alpha{}^{\gamma\beta^\prime}W_{\gamma\alpha^\prime\beta\beta^\prime} + \tfss_{\beta}{}^{\gamma\beta^\prime}W_{\alpha\alpha^\prime\gamma\beta^\prime} + 2H_{\alpha^\prime}W_{\alpha\gamma\beta}{}^\gamma \Bigr) .
 \end{align*}
 Combining this with the formulas in \cref{rk:mK} yields
 \begin{equation}
  \label{eqn:juhl1-first-step}
  \begin{split}
   \cmJ_1 & = \mK_1 + \mK_2 + W_{\alpha\beta\alpha^\prime}{}^\beta W^{\alpha\gamma\alpha^\prime}{}_\gamma - \frac{1}{2}W_{\alpha\beta\alpha^\prime\gamma}W^{\alpha\beta\alpha^\prime\gamma} + \tfss^2_{\alpha\beta}W^{\alpha\gamma\beta}{}_\gamma \\
   & \quad + \tfss^{\alpha\gamma\alpha^\prime}\tfss^{\beta\delta}{}_{\alpha^\prime}W_{\alpha\beta\gamma\delta} - \tfss^{\alpha\beta\alpha^\prime}\tfss_{\alpha\beta}{}^{\beta^\prime}W_{\alpha^\prime\gamma\beta^\prime}{}^\gamma \\
  & \quad + 2\tfss^{\gamma\alpha\alpha^\prime}\tfss_\gamma{}^{\beta\beta^\prime}W_{\alpha\alpha^\prime\beta\beta^\prime} - \tfss^{\gamma\alpha\alpha^\prime}\tfss_\gamma{}^{\beta\beta^\prime}W_{\alpha\beta\alpha^\prime\beta^\prime} .
  \end{split}
 \end{equation}
 On the other hand, taking the tangential divergence of Equation~\eqref{eqn:div-tfss2} yields
 \begin{equation}
  \label{eqn:div2tfss}
  \cmJ_2 = -\mK_1 - \mK_2 .
 \end{equation}
 Equations~\eqref{eqn:2Wm+Wn}, \eqref{eqn:juhl1-first-step}, and~\eqref{eqn:div2tfss} imply that
 \begin{equation*}
  \mJ_1 = -\frac{\submfdim-2}{2(\submfdim-3)(\submfdim-6)}\cmI_1 + \cmJ_1 - \frac{1}{\submfdim-3}\cmJ_2 .
 \end{equation*}
 Applying the definitions of $\cmI_1$, $\cmJ_1$, and $\cmJ_2$ yields
 \begin{align*}
  \mJ_1 & = \frac{\submfdim-4}{(\submfdim-3)(\submfdim-6)}\oDelta\lv\tfss\rv^2 - \frac{1}{\submfdim-3}\onabla^\alpha\onabla^\beta\tfss^2_{\alpha\beta} +\tfss^{\alpha\beta\alpha^\prime}\nabla_{\alpha^\prime}W_{\alpha\gamma\beta}{}^\gamma \\
   & \quad + (\submfdim-2)\lv\Di\rv^2 - \frac{\submfdim-2}{(\submfdim-3)(\submfdim-6)}\otrSch\lv\tfss\rv^2 - \frac{\submfdim-2}{\submfdim-3}\tfss^2_{\alpha\beta}\oSch^{\alpha\beta} - 2H_{\alpha^\prime}\tfss^{\alpha\beta\alpha^\prime}W_{\alpha\gamma\beta}{}^\gamma .
 \end{align*}
 Specializing to the case \codimone\ and using Equation~\eqref{eqn:gcD} yields the final conclusion.
\end{proof}

\begin{proof}[Proof of \cref{juhl2}]
 Set
 \begin{align*}
  \cmJ_3 & := \tfss^{\alpha\beta\alpha^\prime}\nabla_{\alpha^\prime}\Sch_{\alpha\beta} + \tfss^{\alpha\beta\alpha^\prime}\tfss_{\alpha\beta}{}^{\beta^\prime}\Sch_{\alpha^\prime\beta^\prime} - \tfss^{\alpha\beta\alpha^\prime}\onabla_\alpha\onabla_\beta H_{\alpha^\prime} - \onabla^\beta(\Di^{\alpha\alpha^\prime}\tfss_{\alpha\beta\alpha^\prime}) \\
   & \quad - (\submfdim-1)\lv\Di\rv^2 - \Di^{\alpha\alpha^\prime}W_{\alpha\beta\alpha^\prime}{}^\beta - \tfss^2_{\alpha\beta}\oSch^{\alpha\beta} + \frac{\submfdim-3}{\submfdim-2}H^{\alpha^\prime}\tr\tfss^3_{\alpha^\prime} \\
   & \quad - H^{\alpha^\prime}\tfss^{\alpha\beta}{}_{\alpha^\prime}\oSch_{\alpha\beta} + H^{\alpha^\prime}\tfss_{\alpha\beta\alpha^\prime}\tfss^{\alpha\beta\beta^\prime}H_{\beta^\prime} + \frac{1}{2}\lv H\rv^2 \lv \tfss \rv^2 + \tfss^{\alpha\beta\alpha^\prime}\mC_{\alpha\alpha^\prime\beta} \\
   & \quad + H^{\alpha^\prime}\tfss^{\alpha\beta\beta^\prime}W_{\alpha\alpha^\prime\beta\beta^\prime} + \frac{1}{\submfdim-2}H^{\alpha^\prime}\tfss^{\alpha\beta}{}_{\alpha^\prime}W_{\alpha\gamma\beta}{}^\gamma .
 \end{align*}
 Using Equations~\eqref{eqn:second-fundamental-form-dual}, \eqref{eqn:second-fundamental-form}, \eqref{eqn:defn-Di}, and~\eqref{eqn:gcP},  we first compute that
 \begin{align*}
  \tfss^{\alpha\beta\alpha^\prime}\nabla_{\alpha^\prime}\Sch_{\alpha\beta} & = \tfss^{\alpha\beta\alpha^\prime}(\nabla_\alpha\Sch_{\alpha^\prime\beta} - C_{\alpha\alpha^\prime\beta} ) \\
   & = \tfss^{\alpha\beta\alpha^\prime}(\onabla_\alpha\Sch_{\alpha^\prime\beta} + \tfss^\gamma{}_{\alpha\alpha^\prime}\Sch_{\beta\gamma} - \tfss_{\alpha\beta}{}^{\beta^\prime}\Sch_{\alpha^\prime\beta^\prime} + H_{\alpha^\prime}\Sch_{\alpha\beta} - C_{\alpha\alpha^\prime\beta} ) \\
   & = \tfss^{\alpha\beta\alpha^\prime} \bigl( \onabla_\alpha\onabla_\beta H_{\alpha^\prime} + \onabla_\alpha\Di_{\beta\alpha^\prime} + \oSch_\alpha{}^\gamma\tfss_{\gamma\beta\alpha^\prime} - H^{\beta^\prime}\tfss^\gamma{}_{\alpha\beta^\prime}\tfss_{\gamma\beta\alpha^\prime} \\
    & \quad + H_{\alpha^\prime}\oSch_{\alpha\beta} - H_{\alpha^\prime}H_{\beta^\prime}\tfss_{\alpha\beta}{}^{\beta^\prime} + H_{\alpha^\prime}\Fi_{\alpha\beta} - \frac{1}{2}\lv H\rv^2\tfss_{\alpha\beta\alpha^\prime} \\
    & \quad + \Fi_\alpha{}^\gamma\tfss_{\gamma\beta\alpha^\prime} - \tfss_{\alpha\beta}{}^{\beta^\prime}\Sch_{\alpha^\prime\beta^\prime} - \mC_{\alpha\alpha^\prime\beta} - H^{\beta^\prime}W_{\alpha\alpha^\prime\beta\beta^\prime} \bigr) .
 \end{align*}
 Combining this with Equation~\eqref{eqn:gcD} and the definition~\eqref{eqn:defn-Fialkow} of $\Fi_{\alpha\beta}$ yields
 \begin{equation}
  \label{eqn:cmJ3}
  \cmJ_3 = \tfss^2_{\alpha\beta}\Fi^{\alpha\beta} .
 \end{equation}
 Equations~\eqref{eqn:2Wm+Wn}, \eqref{eqn:juhl1-first-step}, and~\eqref{eqn:cmJ3} imply that
 \begin{equation*}
  \mJ_2 = -\frac{1}{2(\submfdim-3)(\submfdim-6)}\cmI_1 - \frac{1}{\submfdim-3}\cmJ_2 - \cmJ_3 .
 \end{equation*}
 Combining Equation~\eqref{eqn:div-tfss2} with the definitions of $\cmI_1$, $\cmJ_2$, and $\cmJ_3$ yields
 \begin{align*}
  \mJ_2 & = -\tfss^{\alpha\beta\alpha^\prime}\nabla_{\alpha^\prime}\Sch_{\alpha\beta} - \tfss^{\alpha\beta\alpha^\prime}\tfss_{\alpha\beta}{}^{\beta^\prime}\Sch_{\alpha^\prime\beta^\prime} + \tfss^{\alpha\beta\alpha^\prime}\onabla_\alpha\onabla_\beta H_{\alpha^\prime} + H^{\alpha^\prime}\tfss^{\alpha\beta}{}_{\alpha^\prime}\oSch_{\alpha\beta} \\
   & \quad - \frac{1}{\submfdim-3}\onabla^\alpha\onabla^\beta\tfss^2_{\alpha\beta} + \frac{\submfdim-5}{2(\submfdim-3)(\submfdim-6)}\oDelta\lv\tfss\rv^2 - \frac{1}{(\submfdim-3)(\submfdim-6)}\otrSch\lv\tfss\rv^2 \\
   & \quad + \frac{\submfdim-4}{\submfdim-3}\tfss^2_{\alpha\beta}\oSch^{\alpha\beta} - \frac{\submfdim-3}{\submfdim-2}H^{\alpha^\prime}\tr\tfss^3_{\alpha^\prime} - H^{\alpha^\prime}\tfss^{\alpha\beta\beta^\prime}W_{\alpha\alpha^\prime\beta\beta^\prime} \\
   & \quad - \frac{1}{\submfdim-2}H^{\alpha^\prime}\tfss^{\alpha\beta}{}_{\alpha^\prime}W_{\alpha\gamma\beta}{}^\gamma - H^{\alpha^\prime}\tfss_{\alpha\beta\alpha^\prime}\tfss^{\alpha\beta\beta^\prime}H_{\beta^\prime} - \frac{1}{2}\lv H \rv^2 \lv \tfss \rv^2 \\
   & \quad + \submfdim\lv\Di\rv^2 + 2\Di^{\alpha\alpha^\prime}W_{\alpha\beta\alpha^\prime}{}^\beta .
 \end{align*}
 Specializing to the case \codimone\ yields the final conclusion.
\end{proof}

\begin{proof}[Proof of \cref{chalabi1}]
 Equation~\eqref{eqn:juhl1-first-step} implies that
 \begin{equation*}
  \mN_1 = \Wm + \frac{1}{2}\Wn - \frac{\submfdim-6}{2}\cmJ_1 .
 \end{equation*}
 The definition~\eqref{eqn:mP} of $\mP_{\alpha\beta}$ yields
 \begin{multline*}
   (\submfdim-6)\tfss_{\alpha\beta}^2\mP^{\alpha\beta} +\lv\tfss\rv^2\mP_\alpha{}^\alpha - (\submfdim-3)\lv H \rv^2\lv\tfss\rv^2 - (\submfdim-6)H^{\alpha^\prime}\tr\tfss^3_{\alpha^\prime} \\
  = \frac{\mfdim-\submfdim+2}{(\mfdim-1)(\mfdim-2)}\lv\tfss\rv^2R - \frac{1}{\mfdim-2}\lv\tfss\rv^2R_{\alpha^\prime}{}^{\alpha^\prime} + \frac{\submfdim-6}{\mfdim-2}\tfss_{\alpha\beta}^2R^{\alpha\beta} .
 \end{multline*}
 Combining this with the definitions of $\cmI_1$ and $\cmJ_1$ yields
 \begin{align*}
  \MoveEqLeft[0] \mN_1 = -\frac{1}{2}\oDelta\lv\tfss\rv^2 - (\submfdim-6)\onabla^\beta(\Di^{\alpha\alpha^\prime}\tfss_{\alpha\beta\alpha^\prime}) - \frac{\submfdim-6}{2}\tfss^{\alpha\beta\alpha^\prime}\nabla_{\alpha^\prime}W_{\alpha\gamma\beta}{}^\gamma \\
   & \quad + \frac{\mfdim-\submfdim+2}{(\mfdim-1)(\mfdim-2)}\lv\tfss\rv^2R - \frac{1}{\mfdim-2}\lv\tfss\rv^2 R_{\alpha^\prime}{}^{\alpha^\prime} + \frac{\submfdim-6}{\mfdim-2}\tfss^2_{\alpha\beta}R^{\alpha\beta} + (\submfdim-3)\lv H \rv^2\lv\tfss\rv^2 \\
   & \quad + (\submfdim-6)H^{\alpha^\prime}\tr\tfss^3_{\alpha^\prime} + (\submfdim-6)H_{\alpha^\prime}\tfss^{\alpha\beta\alpha^\prime}W_{\alpha\gamma\beta}{}^\gamma - (\submfdim-3)(\submfdim-6)\lv\Di\rv^2 .
 \end{align*}
 Specializing to the case $\submfdim=4$ yields the final conclusion.
\end{proof}

\begin{proof}[Proof of \cref{chalabi}]
 Set
 \begin{align*}
  \cmN & := \nabla^{\alpha^\prime}\nabla_{\alpha^\prime}W_{\alpha\beta}{}^{\alpha\beta} + (\mfdim-\submfdim-6)H^{\alpha^\prime}\nabla_{\alpha^\prime}W_{\alpha\beta}{}^{\alpha\beta} + 2(\mfdim-2)\onabla^\alpha\mC_\alpha + \oDelta W_{\alpha\beta}{}^{\alpha\beta} \\
   & \quad - 8\onabla^\gamma( \tfss_\gamma{}^{\alpha\alpha^\prime}W_{\alpha\beta\alpha^\prime}{}^\beta) + 2(\mfdim-2)\tfss^{\alpha\beta\alpha^\prime}\mC_{\alpha\alpha^\prime\beta} + 2(\submfdim-1)\mB_\alpha{}^\alpha - 2\trSch W_{\alpha\beta}{}^{\alpha\beta} \\
   & \quad - 2(\mfdim-2)\Sch^{\alpha\beta}W_{\alpha\gamma\beta}{}^\gamma - 4W_{\alpha\beta\alpha^\prime}{}^\beta\onabla^\alpha H^{\alpha^\prime} - 2(\mfdim-4)\lv H\rv^2 W_{\alpha\beta}{}^{\alpha\beta} \\
   & \quad + 4W_{\alpha\beta\alpha^\prime}{}^\beta\onabla_\gamma\tfss^{\gamma\alpha\alpha^\prime} - 2(\mfdim-2)H^{\alpha^\prime}\tfss^{\alpha\beta}{}_{\alpha^\prime} W_{\alpha\gamma\beta}{}^\gamma - 2(\mfdim-2)\Di^{\alpha\alpha^\prime}W_{\alpha\beta\alpha^\prime}{}^\beta .
 \end{align*}
 Observe that the Weyl--Bianchi identity~\eqref{eqn:weyl-bianchi} and the definitions~\eqref{eqn:second-fundamental-form-dual} and~\eqref{eqn:mC} of $\tfss_{\alpha\beta\alpha^\prime}$ and $\mC_{abc}$, respectively, imply that
 \begin{multline}
  \label{eqn:HnablaW}
  H^{\alpha^\prime}\nabla_{\alpha^\prime}W_{\alpha\beta}{}^{\alpha\beta} = 2H^{\alpha^\prime}\onabla^\alpha W_{\alpha\beta\alpha^\prime}{}^\beta + 2(\submfdim-1)H^{\alpha^\prime}\mC_{\alpha^\prime} + 2\lv H\rv^2 W_{\alpha\beta}{}^{\alpha\beta} \\
   + 2H^{\alpha^\prime}\tfss^{\alpha\beta\beta^\prime}W_{\alpha\alpha^\prime\beta\beta^\prime} + 2H^{\alpha^\prime}\tfss^{\alpha\beta}{}_{\alpha^\prime}W_{\alpha\gamma\beta}{}^\gamma .
 \end{multline}
 The Weyl--Bianchi identity also implies that
 \begin{equation}
  \label{eqn:standard-DeltaW}
  \begin{split}
   \Delta W_{abcd} & = 2(\mfdim-3)\nabla_{[a}C_{|cd|b]} + 2\nabla_{[c}C_{|ab|d]} - 2B_{a[c}g_{d]b} + 2B_{b[c}g_{d]a} \\
    & \quad - W_{ab}{}^{ef}W_{efcd} - 2W_{aecf}W_b{}^e{}_d{}^f + 2W_{aedf}W_b{}^e{}_c{}^f + 2\trSch W_{abcd} \\
    & \quad - 2(\mfdim-3)\Sch^e{}_{[a}W_{b]ecd} - 2\Sch^e{}_{[c}W_{d]eab} ,
  \end{split}
 \end{equation}
 where indices enclosed in vertical bars are not included in the skew symmetrization;
 e.g.
 \begin{equation*}
  2\nabla_{[a}C_{|cd|b]} := \nabla_a C_{cdb} - \nabla_b C_{cda} .
 \end{equation*}
 On the one hand, taking the complete tangential trace of Equation~\eqref{eqn:standard-DeltaW} yields
 \begin{multline*}
  \Delta W_{\alpha\beta}{}^{\alpha\beta} = -2(\mfdim-2)\nabla^\alpha C_{\beta\alpha}{}^\beta - 2(\submfdim-1)B_\alpha{}^\alpha + 2\trSch W_{\alpha\beta}{}^{\alpha\beta} \\
   + 2(\mfdim-2)\Sch^{\alpha c}W_{c\beta \alpha}{}^\beta - 2\left( W_{\alpha\beta cd}W^{\alpha\beta cd} + W_{c\alpha d}{}^\alpha W^{c\beta d}{}_\beta - W_{\alpha c \beta d}W^{\alpha c \beta d} \right) .
 \end{multline*}
 Using the definitions~\eqref{eqn:second-fundamental-form}, \eqref{eqn:defn-Di}, \eqref{eqn:mC}, and~\eqref{eqn:mB} of $\tfss_{\alpha\beta\alpha^\prime}$, $\Di_{\alpha\alpha^\prime}$, $\mC_{abc}$, and~$\mB_{\alpha\beta}$, respectively, yields
 \begin{align*}
  \Delta W_{\alpha\beta}{}^{\alpha\beta} & = -2(\mfdim-2)\onabla^\alpha\left( \mC_\alpha + H^{\alpha^\prime}W_{\alpha\beta\alpha^\prime}{}^\beta\right) - 2(\mfdim-2)\tfss^{\alpha\beta\alpha^\prime}\mC_{\alpha\alpha^\prime\beta} \\
   & \quad - 2(\submfdim-1)\mB_\alpha{}^\alpha - 2(\mfdim-2)H^{\alpha^\prime}\tfss^{\alpha\beta\beta^\prime}W_{\alpha\alpha^\prime\beta\beta^\prime} + 2(\mfdim-2)\Sch^{\alpha\beta}W_{\alpha\gamma\beta}{}^\gamma \\
   & \quad - 2(\submfdim-1)(\mfdim-6)H^{\alpha^\prime}\mC_{\alpha^\prime} + 4(\submfdim-1) H^{\alpha^\prime}H^{\beta^\prime}W_{\alpha^\prime\alpha\beta^\prime}{}^\alpha + 2\trSch W_{\alpha\beta}{}^{\alpha\beta} \\
   & \quad + 2(\mfdim-2)\Di^{\alpha\alpha^\prime}W_{\alpha\beta\alpha^\prime}{}^\beta + 2(\mfdim-2)W_{\alpha\beta\alpha^\prime}{}^\beta\onabla^\alpha H^{\alpha^\prime} \\
   & \quad - 2\left( W_{\alpha\beta cd}W^{\alpha\beta cd} + W_{c\alpha d}{}^\alpha W^{c\beta d}{}_\beta - W_{\alpha c \beta d}W^{\alpha c \beta d} \right) .
 \end{align*}
 Using Equation~\eqref{eqn:HnablaW} to eliminate $H^{\alpha^\prime}\mC_{\alpha^\prime}$ yields
 \begin{equation}
  \label{eqn:DeltaW}
  \begin{split}
   \Delta W_{\alpha\beta}{}^{\alpha\beta} & = - 2(\mfdim-2)\onabla^\alpha\mC_\alpha - 8\onabla^\alpha(H^{\alpha^\prime}W_{\alpha\beta\alpha^\prime}{}^\beta) - 2(\mfdim-2)\tfss^{\alpha\beta\alpha^\prime}\mC_{\alpha\alpha^\prime\beta} \\
   & \quad - (\mfdim-6)H^{\alpha^\prime}\nabla_{\alpha^\prime}W_{\alpha\beta}{}^{\alpha\beta} - 2(\submfdim-1)\mB_\alpha{}^\alpha + 8W_{\alpha\beta\alpha^\prime}{}^\beta\onabla^\alpha H^{\alpha^\prime} \\
   & \quad + 2(\mfdim-6)\lv H\rv^2W_{\alpha\beta}{}^{\alpha\beta} + 4(\submfdim-1) H^{\alpha^\prime}H^{\beta^\prime}W_{\alpha^\prime\alpha\beta^\prime}{}^\alpha \\
   & \quad - 8H^{\alpha^\prime}\tfss^{\alpha\beta\beta^\prime}W_{\alpha\alpha^\prime\beta\beta^\prime} + 2(\mfdim-6)H^{\alpha^\prime}\tfss^{\alpha\beta}{}_{\alpha^\prime}W_{\alpha\gamma\beta}{}^\gamma \\
   & \quad + 2\trSch W_{\alpha\beta}{}^{\alpha\beta} + 2(\mfdim-2)\Sch^{\alpha\beta}W_{\alpha\gamma\beta}{}^\gamma + 2(\mfdim-2)\Di^{\alpha\alpha^\prime}W_{\alpha\beta\alpha^\prime}{}^\beta \\
   & \quad - 2\left( W_{\alpha\beta cd}W^{\alpha\beta cd} + W_{c\alpha d}{}^\alpha W^{c\beta d}{}_\beta - W_{\alpha c \beta d}W^{\alpha c \beta d} \right) . 
  \end{split}
 \end{equation}
On the other hand, the definition~\eqref{eqn:second-fundamental-form} of the second fundamental form yields
 \begin{equation}
  \label{eqn:oDeltaW}
  \begin{split}
   \MoveEqLeft \nabla^\gamma\nabla_\gamma W_{\alpha\beta}{}^{\alpha\beta} = \oDelta W_{\alpha\beta}{}^{\alpha\beta} + 4W_{\alpha\beta\alpha^\prime}{}^\beta\onabla_\gamma\tfss^{\gamma\alpha\alpha^\prime} + 4W_{\alpha\beta\alpha^\prime}{}^\beta\onabla^\alpha H^{\alpha^\prime} \\
    & \quad - \submfdim H^{\alpha^\prime}\nabla_{\alpha^\prime}W_{\alpha\beta}{}^{\alpha\beta} - 4\tfss^2_{\alpha\beta}W^{\alpha\gamma\beta}{}_\gamma + 8\tfss^{\gamma\alpha\alpha^\prime}\tfss_\gamma{}^{\beta\beta^\prime}W_{\alpha\beta\alpha^\prime\beta^\prime} \\
    & \quad + 4\tfss^{\alpha\beta\alpha^\prime}\tfss_{\alpha\beta}{}^{\beta^\prime}W_{\alpha^\prime\gamma\beta^\prime}{}^\gamma - 4\tfss^{\gamma\alpha\alpha^\prime}\tfss_\gamma{}^{\beta\beta^\prime}W_{\alpha\alpha^\prime\beta\beta^\prime} \\
    & \quad - 8H^{\alpha^\prime}\tfss^{\alpha\beta}{}_{\alpha^\prime}W_{\alpha\gamma\beta}{}^\gamma - 8H^{\alpha^\prime}\tfss^{\alpha\beta\beta^\prime}W_{\alpha\alpha^\prime\beta\beta^\prime} - 8\onabla^\alpha(H^{\alpha^\prime}W_{\alpha\beta\alpha^\prime}{}^\beta) \\
    & \quad - 4\lv H\rv^2W_{\alpha\beta}{}^{\alpha\beta} + 4(\submfdim-1)H^{\alpha^\prime}H^{\beta^\prime}W_{\alpha^\prime\alpha\beta^\prime}{}^\alpha - 8\onabla^\gamma(\tfss_\gamma{}^{\alpha\alpha^\prime}W_{\alpha\beta\alpha^\prime}{}^\beta) .
  \end{split}
 \end{equation}
 Recall that $\nabla^{\alpha^\prime}\nabla_{\alpha^\prime}W_{\alpha\beta}{}^{\alpha\beta} = \Delta W_{\alpha\beta}{}^{\alpha\beta} - \nabla^\gamma\nabla_\gamma W_{\alpha\beta}{}^{\alpha\beta}$.
 Combining this with Equations~\eqref{eqn:DeltaW} and~\eqref{eqn:oDeltaW} yields
 \begin{multline}
  \label{eqn:cmN-invariant}
  \cmN =  - 2\left( W_{\alpha\beta cd}W^{\alpha\beta cd} + W_{c\alpha d}{}^\alpha W^{c\beta d}{}_\beta - W_{\alpha c \beta d}W^{\alpha c \beta d} \right) + 4\tfss^2_{\alpha\beta}W^{\alpha\gamma\beta}{}_\gamma \\
    - 8\tfss^{\gamma\alpha\alpha^\prime}\tfss_\gamma{}^{\beta\beta^\prime}W_{\alpha\beta\alpha^\prime\beta^\prime} - 4\tfss^{\alpha\beta\alpha^\prime}\tfss_{\alpha\beta}{}^{\beta^\prime}W_{\alpha^\prime\gamma\beta^\prime}{}^\gamma + 4\tfss^{\gamma\alpha\alpha^\prime}\tfss_\gamma{}^{\beta\beta^\prime}W_{\alpha\alpha^\prime\beta\beta^\prime} .
 \end{multline}
 Combining Equations~\eqref{eqn:juhl1-first-step} and~\eqref{eqn:cmN-invariant} yields
 \begin{multline*}
  \mN_2 = \frac{1}{\submfdim-1}\cmN + \frac{\mfdim-4}{(\submfdim-3)(\submfdim-6)}\Wn - \frac{2(\mfdim - \submfdim - 1)}{(\submfdim-1)(\submfdim-3)}\cmJ_1 \\ + \frac{8}{\submfdim-1}\mK_2 + \frac{4}{\submfdim-1}W_{\alpha\beta\alpha^\prime}{}^\beta W^{\alpha\gamma\alpha^\prime}{}_\gamma .
 \end{multline*}
 Now observe that
 \begin{align*}
   \MoveEqLeft \frac{4(\mfdim-\submfdim-1)}{(\submfdim-1)(\submfdim-3)}\Sch^{\alpha\beta}W_{\alpha\gamma\beta}{}^\gamma - \frac{2}{\submfdim-1}\trSch W_{\alpha\beta}{}^{\alpha\beta} + \frac{2(\mfdim-4)}{(\submfdim-3)(\submfdim-6)}\Sch_\alpha{}^\alpha W_{\beta\gamma}{}^{\beta\gamma} \\
    & = \frac{2(\mfdim-4)(\mfdim-\submfdim+2)}{(\submfdim-3)(\submfdim-6)(\mfdim-1)(\mfdim-2)}RW_{\alpha\beta}{}^{\alpha\beta} \\
     & \quad - \frac{2(\mfdim-4)}{(\submfdim-3)(\submfdim-6)(\mfdim-2)}R_{\alpha^\prime}{}^{\alpha^\prime}W_{\alpha\beta}{}^{\alpha\beta} + \frac{4(\mfdim-\submfdim-1)}{(\submfdim-1)(\submfdim-3)(\mfdim-2)}R^{\alpha\beta}W_{\alpha\gamma\beta}{}^\gamma .
 \end{align*}
 Combining this with the formula for $\mK_2$ in \cref{rk:mK}, the definitions of $\cmN$, $\Wn$, and $\cmJ_1$, and Equations~\eqref{eqn:gcD} and~\eqref{eqn:mP} yields
 \begin{align*}
  \MoveEqLeft[0] \mN_2 = \frac{1}{\submfdim-1}\nabla^{\alpha^\prime}\nabla_{\alpha^\prime}W_{\alpha\beta}{}^{\alpha\beta} + \frac{\mfdim-\submfdim-6}{\submfdim-1}H^{\alpha^\prime}\nabla_{\alpha^\prime}W_{\alpha\beta}{}^{\alpha\beta} \\
   & \quad - \frac{4(\mfdim-\submfdim-1)}{(\submfdim-1)(\submfdim-3)}\onabla^\alpha\mC_\alpha + \frac{\submfdim^2 - 5\submfdim + 14 - (\submfdim-1)\mfdim}{(\submfdim-1)(\submfdim-3)(\submfdim-6)}\oDelta W_{\alpha\beta}{}^{\alpha\beta} \\
   & \quad + \frac{2(\mfdim-4)(\mfdim-\submfdim+2)}{(\submfdim-3)(\submfdim-6)(\mfdim-1)(\mfdim-2)}RW_{\alpha\beta}{}^{\alpha\beta} - \frac{2(\mfdim-4)}{(\submfdim-3)(\submfdim-6)(\mfdim-2)}R_{\alpha^\prime}{}^{\alpha^\prime}W_{\alpha\beta}{}^{\alpha\beta} \\
   & \quad + \frac{4(\mfdim-\submfdim-1)}{(\submfdim-1)(\submfdim-3)(\mfdim-2)}R^{\alpha\beta}W_{\alpha\gamma\beta}{}^\gamma - \frac{2(\mfdim-\submfdim-1)}{(\submfdim-1)(\submfdim-3)}\tfss^{\alpha\beta\alpha^\prime}\nabla_{\alpha^\prime}W_{\alpha\gamma\beta}{}^\gamma \\
   & \quad - \frac{4}{\submfdim-1}W_{\alpha\beta\alpha^\prime}{}^\beta\onabla^\alpha H^{\alpha^\prime} + \frac{8(\mfdim-\submfdim-1)}{(\submfdim-1)(\submfdim-3)}H^{\alpha^\prime}\tfss^{\alpha\beta}{}_{\alpha^\prime}W_{\alpha\gamma\beta}{}^\gamma \\
   & \quad - \frac{4(\mfdim-\submfdim+5)}{\submfdim-1}\Di^{\alpha\alpha^\prime}W_{\alpha\beta\alpha^\prime}{}^\beta + \frac{10(\mfdim - 4)}{(\submfdim-1)(\submfdim-6)}\lv H\rv^2 W_{\alpha\beta}{}^{\alpha\beta} .
 \end{align*}
 Specializing to the case $\submfdim=4$ yields the final conclusion.
\end{proof}

\section*{Acknowledgements}
This work was initiated and significantly advanced during the workshop \emph{Partial differential equations and conformal geometry} held at the American Institute of Mathematics (AIM) in August 2022.
We thank AIM for providing an ideal research environment.

\section*{Funding}
JSC acknowledges support from a Simons Foundation Collaboration Grant for Mathematicians, ID 524601.
AW acknowledges support from the Royal Society of New Zealand via Marsden Grant 19-UOA-008 and from a Simons Foundation Collaboration Grant for Mathematicians, ID 686131.

\section*{Declarations}
The authors have no relevant financial or non-financial interests to disclose.

\section*{Data availability}
Data sharing not applicable to this article as no datasets were generated or analyzed during the study.

\bibliography{bib}
\end{document}